\RequirePackage{fix-cm}
\documentclass[twocolumn]{article}

\usepackage{geometry}
\usepackage{amsmath}
\usepackage{amsthm}
\usepackage{a4wide}
\usepackage{graphicx}
\usepackage{latexsym}
\usepackage{epsfig}
\usepackage{eqnarray}
\usepackage{amssymb}
\usepackage{amstext}
\usepackage{amsgen}
\usepackage{amsxtra}
\usepackage{amsgen}
\usepackage{algorithm}
\usepackage{algpseudocode}

\usepackage{todonotes}
\usepackage{xcolor}
\usepackage{array,multirow} 
\usepackage{tabularx}
\usepackage{url} 

\newtheorem{theorem}{Theorem}[section]
\newtheorem{lemma}[theorem]{Lemma}
\theoremstyle{definition}
\newtheorem{example}[theorem]{Example}
\newtheorem{definition}[theorem]{Definition}

\numberwithin{equation}{section}

\newcommand{\N}{\mathbb{N}}

\newcommand{\R}{\mathbb{R}}

\newcommand{\E}{\mathbb{E}} 
\newcommand{\var}{\mathbb{V}\mathrm{ar}}

\newcommand{\X}{\mathcal{X}}
\newcommand{\Y}{\mathcal{Y}}
\newcommand{\V}{\mathcal{V}}
\newcommand{\W}{\mathcal{W}}

\newcommand{\pa}{p_\alpha}
\newcommand{\qa}{q_\alpha}

\newcommand{\der}{\mathrm{d}}
\newcommand{\Der}{\mathrm{D}}

\newcommand{\MG}{\mathcal{M}_f^{\mathrm{G}}}
\newcommand{\MF}{\mathcal{M}_f^{\mathrm{F}}}
\newcommand{\MB}{\mathcal{M}_f^{\mathrm{B}}}
\newcommand{\MIC}{\mathcal{M}_f^{\mathrm{IC}}}
\newcommand{\mb}{\mathcal{M}^{\mathrm{B}}}
\newcommand{\mic}{\mathcal{M}^{\mathrm{IC}}}

\newcommand{\lone}{\ell^1}
\newcommand{\ltwo}{\ell^2}
\newcommand{\ellp}{\ell^p}
\newcommand{\linf}{\ell^{\infty}}

\newcommand{\ua}{u_\alpha}
\newcommand{\uai}{[\ua(f)]_i}
\newcommand{\uhat}{\hat{u}_{\alpha}}
\newcommand{\Uhat}{\hat{U}}

\newcommand{\ICB}{\mathrm{ICB}_J^{p_{\alpha}}}
\newcommand{\ICBq}{\mathrm{ICB}_{\|\cdot\|_1}^{q_{\alpha}}}
\newcommand{\ICBl}{\mathrm{ICB}_{\lone}^{p_{\alpha}}}

\newcommand{\ICBlisoq}{\mathrm{ICB}_{\lone(\R^m)}^{q_{\alpha}}}

\newcommand{\dom}{\mathrm{dom}}
\newcommand{\cont}{\mathrm{cont}}

\DeclareMathOperator*{\argmin}{arg\,min}

\newcommand{\supp}{\mathrm{supp}}
\newcommand{\abs}[1]{\vert #1 \vert}
\newcommand{\sign}{\mathrm{sign}}
\newcommand*{\QED}{\hfill\ensuremath{\square}}
\newcommand{\D}{D}

\begin{document}

\sloppy

\title{Bias-Reduction in Variational Regularization}

\author{Eva-Maria Brinkmann, Martin Burger, Julian Rasch, Camille Sutour\thanks{Institut f\"ur Numerische und Angewandte Mathematik, Westf\"alische Wilhelms-Universit\"at (WWU) M\"unster. Einsteinstr. 62, D 48149 M\"unster, Germany.} }

\maketitle

\begin{abstract}
The aim of this paper is to introduce and study a two-step debiasing method for variational regularization. After solving the standard variational problem,
the key idea is to add a consecutive debiasing step minimizing the data fidelity on an appropriate set, the so-called model manifold. 
The latter is defined by Bregman distances or infimal convolutions thereof, 
using the (uniquely defined) subgradient appearing in the optimality condition of the variational method. 
For particular settings, such as anisotropic $\ell^1$ and TV-type regularization, previously used 
debiasing techniques are shown to be special cases.
The proposed approach is however easily applicable to a wider range of regularizations. 
The two-step debiasing is shown to be well-defined and to optimally reduce bias in a certain setting. 

In addition to visual and PSNR-based evaluations, different notions of bias and variance decompositions are investigated in numerical studies.  
The improvements offered by the proposed scheme are demonstrated and its performance is shown to be comparable to optimal results obtained with Bregman iterations. 
 
\end{abstract}

\section{Introduction}
\label{sec:Intro}

Variational regularization methods with nonquadratic functionals such as total variation 
or $\ell^1$-norms have evolved to a  standard tool in inverse problems \cite{tvzoo1,schusterbuch}, image processing \cite{caselles}, compressed sensing \cite{candes}, 
and recently related fields such as learning theory \cite{Combettes}. 
The popularity of such approaches stems from superior structural properties compared to other regularization approaches. 
$\ell^1$-regularization for example leads to sparse solutions with very accurate or even exact reconstruction 
of the support of the true solution. On the other hand it is known that such methods suffer from a certain bias 
due to the necessary increased weighting of the regularization term with increasing noise.
Two well-known examples are the loss of contrast 
in total variation regularization \cite{tvzoo1,osh-bur-gol-xu-yin} 
or shrinked peak values in $\ell^1$-regularization. 
Accordingly, quantitative values of the solutions have to be taken with care.
 
Several approaches to reduce or eliminate the bias of regularization methods have been considered in literature: 
For $\ell^1$-regularization and similar sparsity-enforcing techniques an ad-hoc approach 
is to determine the support of the solution by the standard variational methods in a first step, 
then use a second debiasing step that minimizes the residual (or a general data fidelity) restricted to that support, 
also known as {\em refitting} \cite{figueiredo,lederer,falcon}. 
A slightly more advanced approach consists in adding a sign-constraint derived from the solution 
of the variational regularization method in addition to the support condition. 
This means effectively that the solution of the debiasing step shares an $\ell^1$-subgradient 
with the solution of the variational regularization method. 
A different and more general approach is to iteratively reduce the bias via Bregman iterations 
\cite{osh-bur-gol-xu-yin} or similar approaches \cite{gilboa,tadmorvese}. 
Recent results for the inverse scale space method in the case of $\ell^1$-regularization 
(respectively certain polyhedral regularization functionals \cite{adaptive1,adaptive2,bregmanbuchkapitel}) 
show that the inverse scale space performs some kind of debiasing. 
Even more, under certain conditions, the variational regularization method and the inverse scale space method provide 
the same subgradient at corresponding settings of the regularization parameters \cite{spectraltv2}. 
Together with a characterization of the solution of the inverse scale space method as a minimizer 
of the residual on the set of elements with the same subgradient, 
this implies a surprising equivalence to the approach of performing a debiasing step with sign-constraints. 
Recently, bias and debiasing in image processing problems were discussed in a more systematic way by Deledalle et al. \cite{deledalle,deledalle2016clear}. 
They distinguish two different types of bias, namely method bias and model bias. 
In particular they suggest a debiasing scheme to reduce the former, which can be applied to some polyhedral one-homogeneous regularizations. 
The key idea of their approach is the definition of suitable spaces, called model subspaces, on which the method bias is minimized.
The remaining model bias is considered as the unavoidable part of the bias, linked to the choice of regularization and hence the solution space of the variational method. 
The most popular example is the staircasing effect that occurs for total variation regularization 
due to the assumption of a piecewise constant solution. 
In the setting of $\ell^1$-regularization a natural model subspace is the set of signals with a given support, 
which yields consistency with the ad-hoc debiasing approach mentioned above.

Based on this observation, the main motivation of this paper is to further develop the approach in the setting of variational regularization 
and unify it with the above-mentioned ideas of debiasing for $\ell^1$-regularization, Bregman iterations, and inverse scale space methods. 

Let us fix the basic notations and give a more detailed discussion of the main idea. 
Given a bounded linear operator $A \colon \X \to \Y $ between Banach spaces, a convex regularization functional $J \colon \X \rightarrow \mathbb{R} \cup \{\infty\}$ 
and a differentiable data fidelity $H: \Y \times \Y \rightarrow \mathbb{R}$, we consider the solution of the variational method
 \begin{equation} \label{variationalmethod0}
      \ua \in \arg \min_{u \in \X} \ H(Au,f) + \alpha J(u)
\end{equation}
as a first step. 
Here $\alpha > 0$ is a suitably chosen regularization parameter. 
This problem has a systematic bias, as we further elaborate on below. 
The optimality condition is given by
\begin{equation} \label{optimality0}
      A^*  \partial_{u} H (A\ua,f) + \alpha \pa = 0,\ \pa \in \partial J(\ua),
\end{equation}
where $\partial_{u} H$ is the derivative of $H$ with respect to the first argument.
Now we proceed to a second step, where we only keep the subgradient $\pa$ and minimize  
\begin{equation} \label{debiasedproblem0}
\uhat \in \arg \min_{u \in \X} \  H(Au,f) \text{ s.t. } \pa \in \partial J(u).
 \end{equation}
Obviously, this problem is only of interest if there is no one-to-one relation between subgradients and primal values $u$, otherwise we always obtain $\uhat=u_\alpha$. 
The most interesting case with respect to applications is the one of $J$ being absolutely one-homogeneous, i.e.
$J(\lambda u) = |\lambda| J(u)$ for all $\lambda \in \R$,
where the subdifferential can be multivalued at least at $u=0$.
 
The debiasing step can be reformulated in an equivalent way as
\begin{equation} \label{debiasedproblem1}
 	\min_{u \in \X} \ H(Au,f) \text{ s.t. } {D}_J^{\pa}(u,\ua) = 0,
\end{equation}
with the (generalized) Bregman distance given by
\begin{equation*}
	{D}_J^{p}(u,v) = J(u)-J(v)-\langle p, u-v \rangle, \quad p \in  \partial J(v).
\end{equation*}
We remark that for absolutely one-homogeneous $J$ this simplifies to 
\begin{equation*}
	{D}_J^{p}(u,v) = J(u)-\langle p, u\rangle, \quad p \in  \partial J(v).
\end{equation*}
The reformulation in terms of a Bregman distance indicates a first connection to Bregman iterations, which we make more precise in the sequel of the paper.

Summing up, we examine the following two-step method:
\begin{enumerate}
\item[1)] Compute the (biased) solution $\ua$ of \eqref{variationalmethod0} with optimality condition \eqref{optimality0},
\item[2)] Compute the (debiased) solution $\uhat$ as the minimizer of \eqref{debiasedproblem0} or equivalently \eqref{debiasedproblem1}.
\end{enumerate} 
In order to relate further to the previous approaches of debiasing $\ell^1$-minimizers given only the support and not the sign, 
as well as the approach with linear model subspaces, we consider another debiasing approach being blind against the sign.  
The natural generalization in the case of an absolutely one-homogeneous functional $J$ is to replace the second step by 
\begin{equation*} 
	\min_{u \in \X} \ H(Au,f)  \text{ s.t. } \ICB(u,\ua)= 0,
\end{equation*}
where 
\begin{align*}
 \ICB(u,\ua) := \big[\D_J^{\pa}(\cdot,\ua) \Box \D_J^{\text{-}\pa}(\cdot,-\ua)\Big](u)
\end{align*}
denotes the infimal convolution between the Bregman distances 
$\D_J^{\pa}(\cdot,\ua)$ and $\D_J^{-\pa}(\cdot,-\ua)$, evaluated at $u \in \X$.
The infimal convolution of two functionals $F$ and $G$ on a Banach space $\X$ is defined as
\begin{align*}
	(F \Box G)(u) &= \inf_{\substack{\phi, \psi \in \X,\\ \phi + \psi = u}} F(\phi) + G(\psi) \\
 &= \inf_{z\in \X} F(u-z) + G(z).
\end{align*}
For the sake of simplicity we carry out all analysis and numerical experiments in this paper for a least-squares data fidelity (related to i.i.d. additive Gaussian noise) 
\begin{equation} \label{quadraticfidelity}
	H(Au,f) = \frac{1}2 \Vert A u - f \Vert_\Y^2
\end{equation}
for some Hilbert space $\Y$, but the basic idea does not seem to change for other data fidelities and noise models.

We show that the sets characterized by the constraints 
\begin{align*}
 D_J^{\pa}(u,\ua) = 0 \quad \text{ and } \quad \ICB(u,\ua)= 0
\end{align*}
constitute a suitable extension of the model subspaces introduced in \cite{deledalle} to 
general variational regularization. 
In particular, we use those manifolds to provide a theoretical basis to define the bias of variational methods and investigate the above approach as a method to reduce it. 
Moreover, we discuss its relation to the statistical intuition of bias.  
At this point it is important to notice that choosing a smaller regularization parameter will also decrease bias, 
but on the other hand strongly increase variance. 
The best we can thus achieve is to reduce the bias at fixed $\alpha$ 
by the two-step scheme while introducing only a small amount of variance. 

The remainder of the paper is organized as follows: 
In Section \ref{sec:motivation} we motivate our approach by considering bias related to the well-known ROF-model \cite{rof} 
and we review a recent approach on debiasing \cite{deledalle}.   
In the next section we introduce our debiasing technique supplemented by some first results.
Starting with a discussion of the classical definition of bias in statistics, we consider a deterministic characterization of bias in Section \ref{sec:BiasModelManifolds}. 
We reintroduce the notion of model and method bias as well as model subspaces as proposed in \cite{deledalle} and extend it to the infinite-dimensional variational setting.  
We furthermore draw an experimental comparison between the bias we consider in this paper and the statistical notion of bias. 
Finally, we comment on the relation of the proposed debiasing to Bregman iterations \cite{osh-bur-gol-xu-yin} and inverse scale space methods \cite{scherzer,gilboa}.
We complete the paper with a description of the numerical implementation via a first-order primal-dual method and show numerical results for signal deconvolution and image denoising.

\section{Motivation}
\label{sec:motivation}
Let us start with an intuitive approach to bias and debiasing in order to further motivate our method. 
To do so, we recall a standard example for denoising, namely the well-known ROF-model \cite{rof}, and 
we rewrite a recent debiasing approach \cite{deledalle} in the setting of our method.  

\subsection{Bias of total variation regularization}
As already mentioned in the introduction, variational regularization methods suffer from a certain bias. 
This systematic error becomes apparent when the regularization parameter is increased. 
Indeed this causes a shift of the overall energy towards the regularizer, and hence a deviation of the reconstruction
from the data in terms of quantitative values.
Intuitively, this can be observed from the discrete version of the classical ROF-model \cite{rof}, i.e. 
\begin{align}
 \ua \in \arg \min_{u \in \R^n} \frac{1}{2} \| u - f \|_2^2 + \alpha \|\Gamma u\|_1,
 \label{eq:rof}
\end{align}
with a discrete gradient operator $\Gamma \in \R^{m \times n}$.
It yields a piecewise constant signal $\ua$ reconstructed from an observation $f \in \mathbb{R}^n$, which has been corrupted by Gaussian noise (see Figure \ref{fig:1DTV_signal}(a)).
\begin{figure*}[th!]
 \center 
 \begin{tabular}{cc}
   \includegraphics[width=0.45\textwidth]{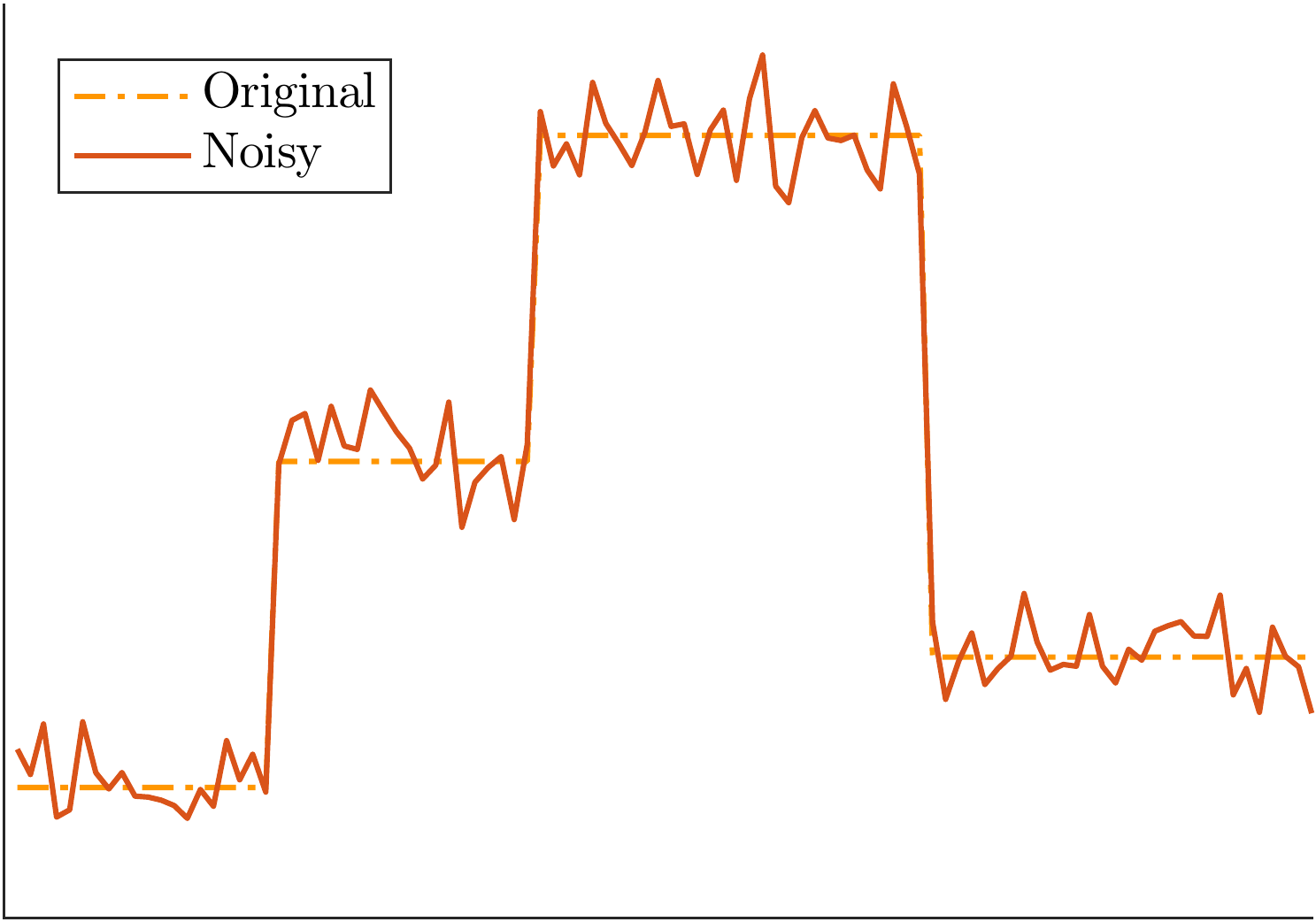}\hfill&
   \includegraphics[width=0.45\textwidth]{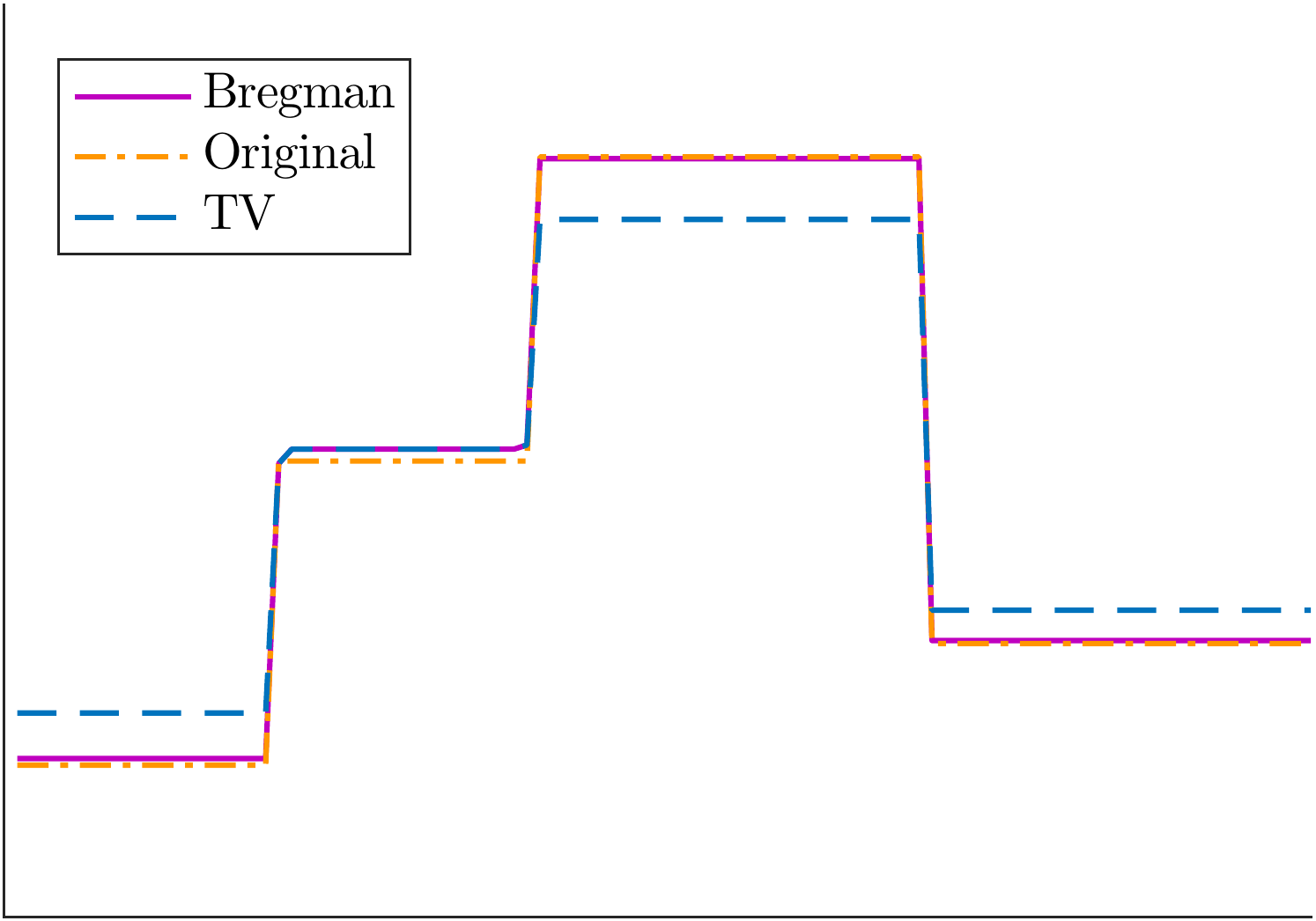} \\
   (a) Original and noisy signal & (b) Denoising with TV regularization \\
   & and Bregman iterations
 \end{tabular}
\caption{Illustration of the bias of the ROF model on a 1D signal. (a) Original signal, and noisy signal corrupted by additive Gaussian noise.
(b) Restoration of the noisy signal with TV regularization and Bregman iterations.
The TV reconstruction recovers the structure of the signal but suffers from a loss of contrast, which is however well recovered with Bregman iterations.}
\label{fig:1DTV_signal}
\end{figure*}
Figure \ref{fig:1DTV_signal}(b) shows the solution of \eqref{eq:rof} together with the true, noiseless signal we aimed to reconstruct.
Even though the structure of the true signal is recovered, the quantitative values of the reconstruction do 
not match the true signal. 
Instead, jumps in the signal have a smaller height, which is often referred to as a loss of contrast. 
Without any further definition, one could intuitively consider this effect as the bias (or one part of the bias) 
of the ROF model. 
Hence, the goal of a bias reduction method would be to restore the proper signal height while keeping the (regularized) structure. 

It has been shown in \cite{osh-bur-gol-xu-yin,benninggroundstates} that this can be achieved by the use of Bregman iterations,
i.e. by iteratively calculating
\begin{align}
 \ua^{k+1} \in \arg \min_{u \in \R^n} \frac{1}{2} \| u - f \|_2^2 + \alpha D_{J}^{\pa^k}(u,\ua^k),
 \label{eq:rof_bregman}
\end{align}
where in our case $J(u) = \|\Gamma u\|_1$, and $\pa^k \in \partial J(\ua^k)$ is a subgradient of the last iterate $\ua^k$. 
Since for total variation regularization the subgradient $\pa^k$ essentially encodes the edge information of the last iterate, its iterative inclusion
allows to keep edges while restoring the correct height of jumps across edges. 
We further elaborate on that in Section \ref{sec:BiasModelManifolds}. 
Indeed, the reconstruction via \eqref{eq:rof_bregman} in Figure \ref{fig:1DTV_signal}(b) shows an almost perfect recovery 
of the true signal even in terms of quantitative values.
This indicates that Bregman iterations are able to reduce or even eliminate our heuristically defined bias.

However, a theoretical basis and justification is still missing, i.e. a proper definition of the bias of variational methods, a proof that 
Bregman iterations indeed reduce the bias in that sense, and in particular a link to the statistical definition and understanding of bias. 
With this paper we aim to define a proper basis for this link, and in particular further establish the connection between 
bias reduction techniques and Bregman distances.

\subsection{Recent debiasing and Bregman distances}
\label{subsec:motivation2}

In order to further motivate the use of Bregman distances for bias reduction let us recall and review a very recent approach on debiasing and 
work out its relation to Bregman distances.
In \cite{deledalle}, Deledalle et al. introduce a debiasing algorithm for anisotropic 
TV-type regularized problems 
\begin{align*}
 \ua \in \arg \min_{u \in \R^n} \frac{1}{2} \| Au - f \|_2^2 + \alpha \|\Gamma u\|_1, 
\end{align*}
with a linear operator $A \in \R^{n \times d}$, a discrete gradient operator $\Gamma \in \R^{n \times m}$ and noisy data $f \in \R^d$.
In \cite{deledalle} the authors argued that the loss of contrast characteristic for this kind of regularization is indeed bias in their sense.
In order to correct for that error, the proposed debiasing method in \cite{deledalle} consists in looking for a debiased solution $\uhat$ 
such that $\Gamma \uhat$ and $\Gamma \ua$ share the same support, but $\uhat$ features the right intensities. 
Mathematically, the solution $\uhat$ of their debiasing problem is given by 
\begin{align}
   \uhat \in \arg \min_{u \in \R^n}  \sup_{z \in F_{\cal I}} \ \frac{1}{2} \|Au - f\|_2^2 + \langle \Gamma u, z \rangle,
   \label{eq:deledalle_deb}
\end{align}
where $F_{\cal I}=\{ z \in \R^m ~|~z_{\cal I} = 0 \}$,
and ${\cal I}$ is the set of indices corresponding to nonzero entries of $\Gamma \ua$. 
We can explicitly compute the supremum (the convex conjugate of the indicator function of the set $F_{\cal I}$), which is
\begin{align*}
 \sup_{z \in F_{\cal I}} \langle \Gamma u, z \rangle  = \begin{cases}
                                                         \infty, &(\Gamma u)_i \neq 0 \text{\small{ for some }} i \notin {\cal I}, \\ 
                                                         0, & \text{ else.}
                                                        \end{cases}
\end{align*}
Hence, $\uhat$ can only be a minimizer of \eqref{eq:deledalle_deb} if $\supp(\Gamma \uhat) \subset \supp(\Gamma \ua )$, thus
\begin{align}
 \uhat \in \arg \min_{u \in \R^n}  \ &\frac{1}{2} \|Au - f\|_2^2 \nonumber \\ 
 \text{ s.t. } &\supp(\Gamma \uhat) \subset \supp(\Gamma \ua ).
 \label{eq:deledalle_deb2}
\end{align}

We can also enforce this support condition using the infimal convolution of two $\ell^1$-Bregman distances.
Defining $J(u) = \|\Gamma u\|_1$, the subdifferential of $J$ at $\ua$ is given by
\begin{align*}
\partial J(\ua) &= \{ \Gamma^T \qa \in \R^n ~|~ \Vert \qa \Vert_\infty \leq 1,\\ 
(\qa)_i &= \text{sign}((\Gamma \ua)_i) \text{ for } (\Gamma \ua)_i \neq 0 \}.
\end{align*} 
In particular $|(\qa)_i| = 1$ on the support of $\Gamma \ua$. 
Let $\qa$ be such a subgradient and consider the $\ell^1$-Bregman distances 
$D_{\|\cdot\|_1}^{\qa}(\cdot, \Gamma \ua)$ and $D_{\|\cdot\|_1}^{-\qa}(\cdot, -\Gamma \ua)$.
According to \cite{moeller14}, their infimal convolution evaluated at $\Gamma u$ is given by:
\begin{align*}
&\quad \ICBq(\Gamma u,\Gamma \ua) \\
&= [ D_{\|\cdot\|_1}^{\qa}(\cdot, \Gamma \ua) \Box D_{\|\cdot\|_1}^{-\qa}(\cdot, -\Gamma \ua)] (\Gamma u) \\
&= \sum_{i=1}^m (1-|(q_{\alpha})_i|)|(\Gamma u)_i|.
\end{align*}
{
We observe that this sum can only be zero if $|(\qa)_i| = 1$ or $(\Gamma u)_i = 0$ for all $i$. 
Assuming that a qualification condition holds, i.e. $\pa = \Gamma^T \qa \in \partial J(\ua)$ with $| (\qa)_i | < 1 $ for $i \notin {\cal I}$, 
i.e. $| (\qa)_i | = 1 \Leftrightarrow (\Gamma \ua)_i \neq 0$,
we can rewrite the above debiasing method \eqref{eq:deledalle_deb} as
}
\begin{align*}
\min_{u \in \R^n} \ \frac{1}{2} \|Au - f\|_2^2  \text{ s.t. } 
\ICBq(\Gamma u,\Gamma \ua) = 0.
\end{align*}
Note that the zero infimal convolution exactly enforces the support condition \eqref{eq:deledalle_deb2} only if 
$| (\qa)_i | < 1$ for all $i \in {\cal I}$. 
Intuitively, since the subdifferential is multivalued at $(\Gamma \ua)_i = 0$, this leads to the question of how to choose $\qa$ properly. 
However, our method does not depend on the choice of a particular $\qa$, but instead we use a unique subgradient $\pa$ coming from the optimality condition of the problem.
We further comment on this in Section \ref{sec:BiasModelManifolds}.

\begin{figure}[t!]
 \center 
 \includegraphics[width=0.4\textwidth]{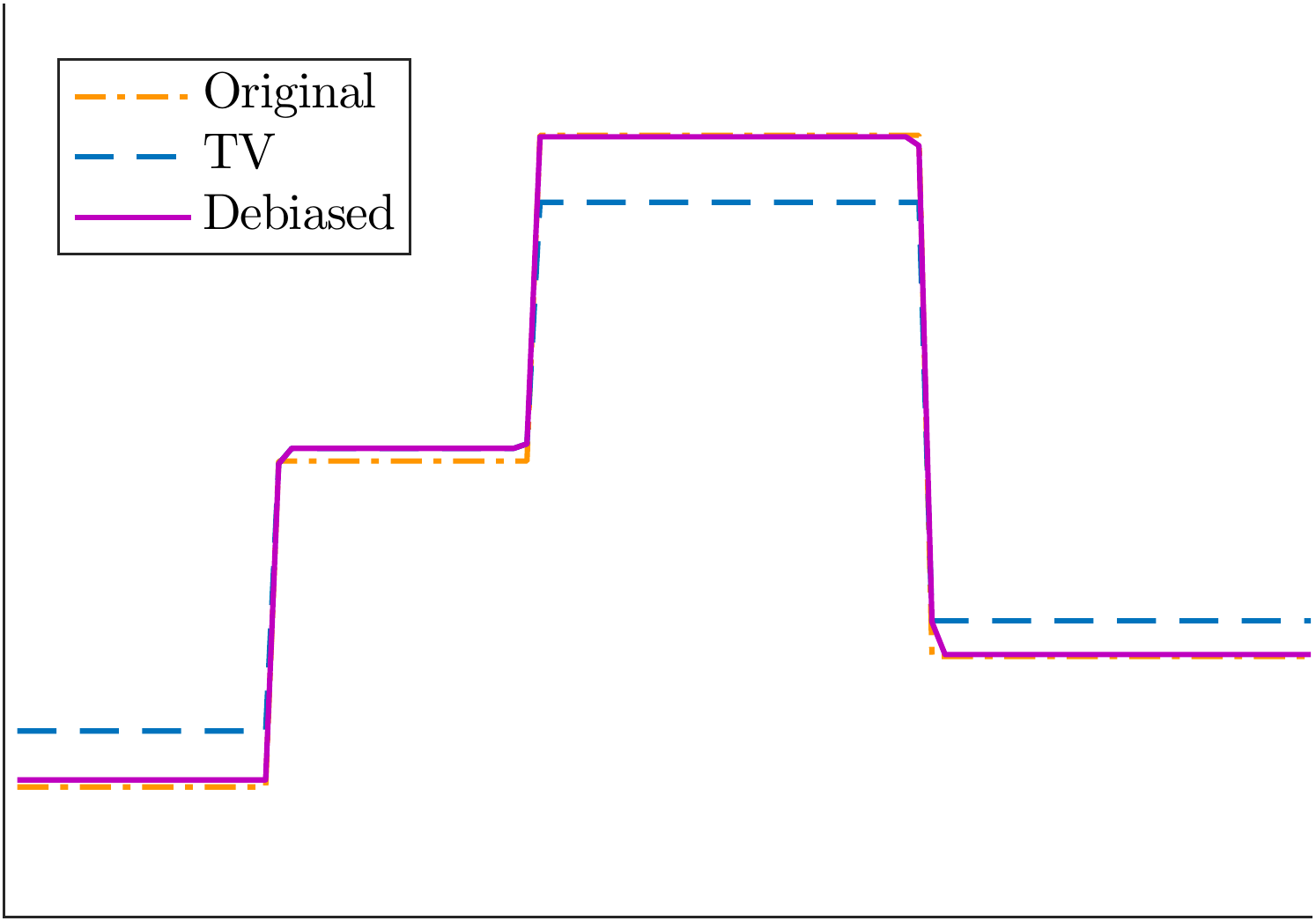}
 \caption{TV denoising of a one-dimensional noisy signal and debiasing using the proposed approach with zero Bregman distance.}
 \label{fig:1DTVsignal2}
\end{figure}

\begin{figure*}[th!]
\center
\begin{tabular}{ccc}
Color image\footnotemark
& Original image & Noisy image\\
\includegraphics[width=0.3\textwidth]{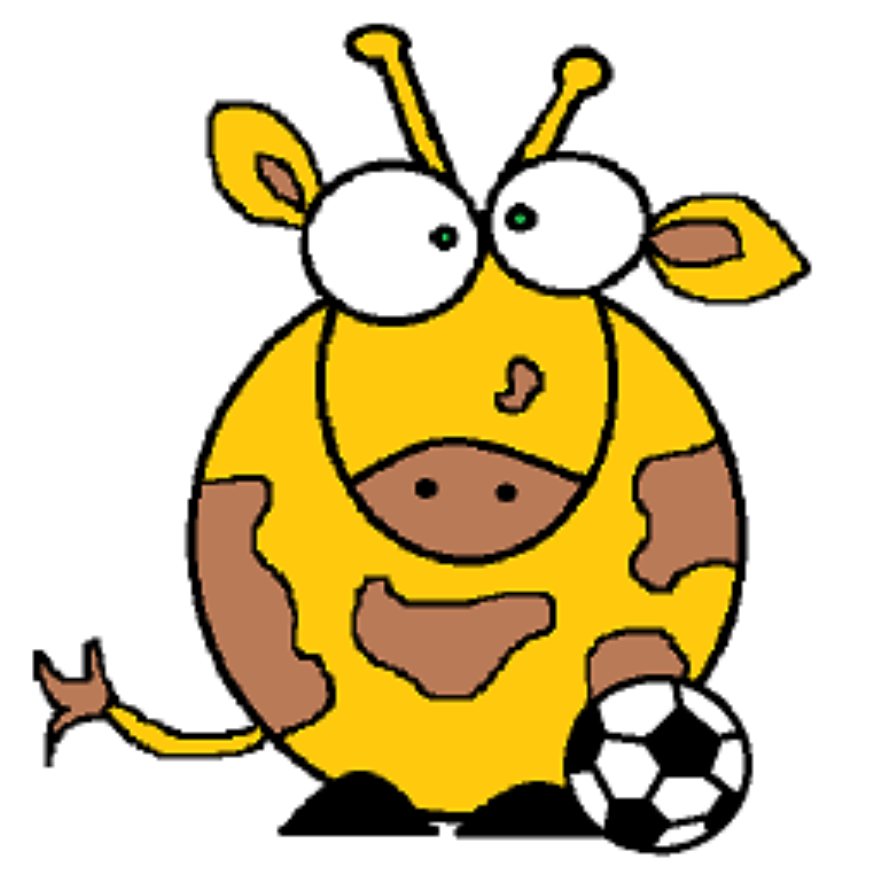}& 
\includegraphics[width=0.3\textwidth]{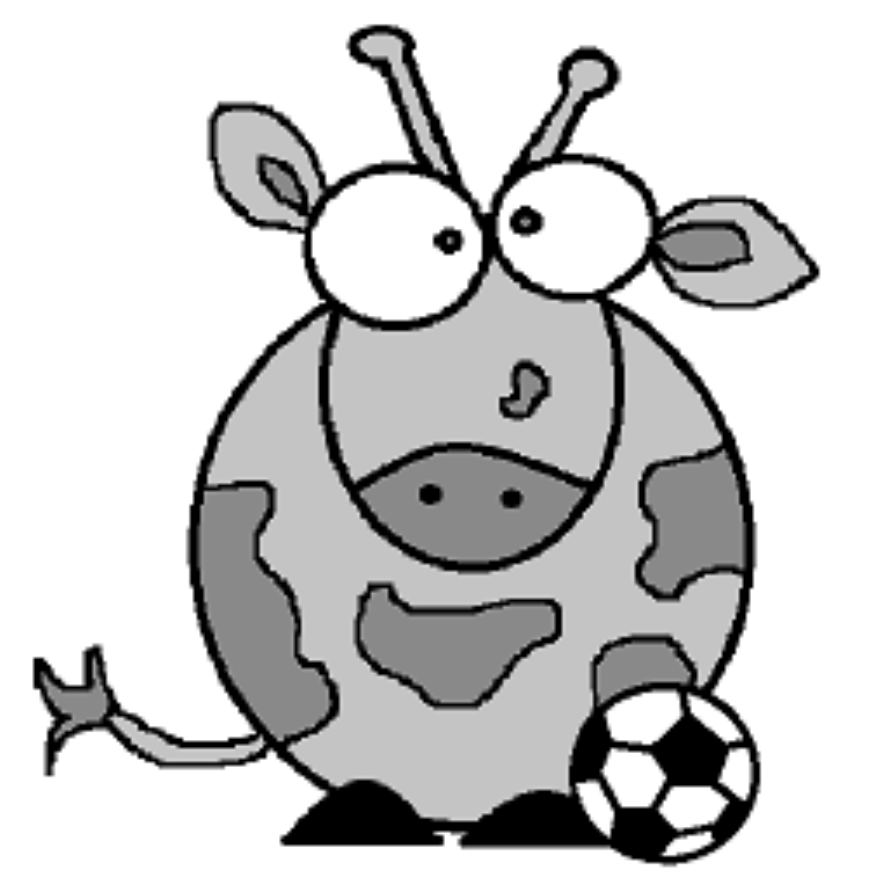}&
\includegraphics[width=0.3\textwidth]{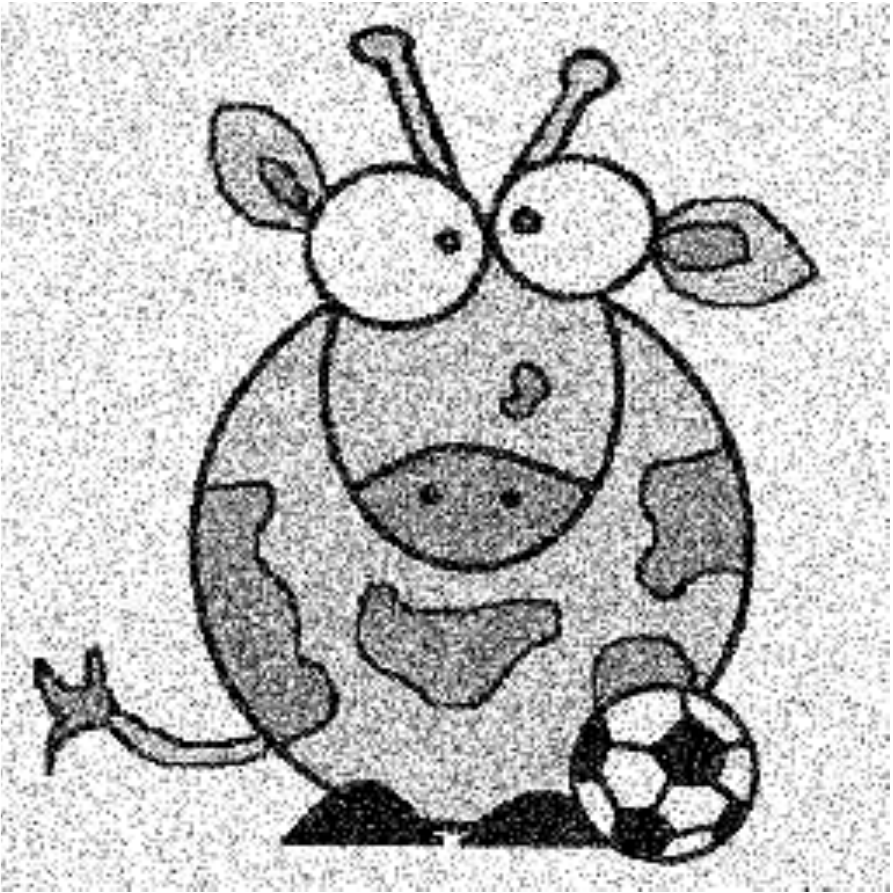}\vspace{1em}\\
TV denoising & Bregman debiasing & ICB debiasing\\ 
\includegraphics[width=0.3\textwidth]{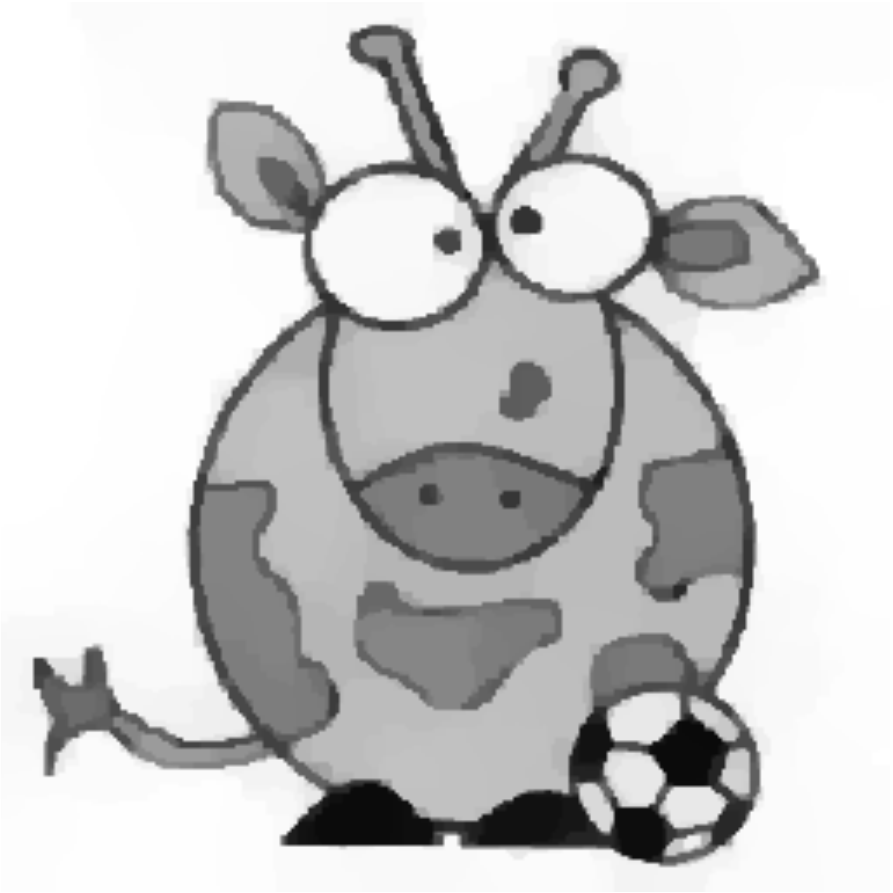}&
\includegraphics[width=0.3\textwidth]{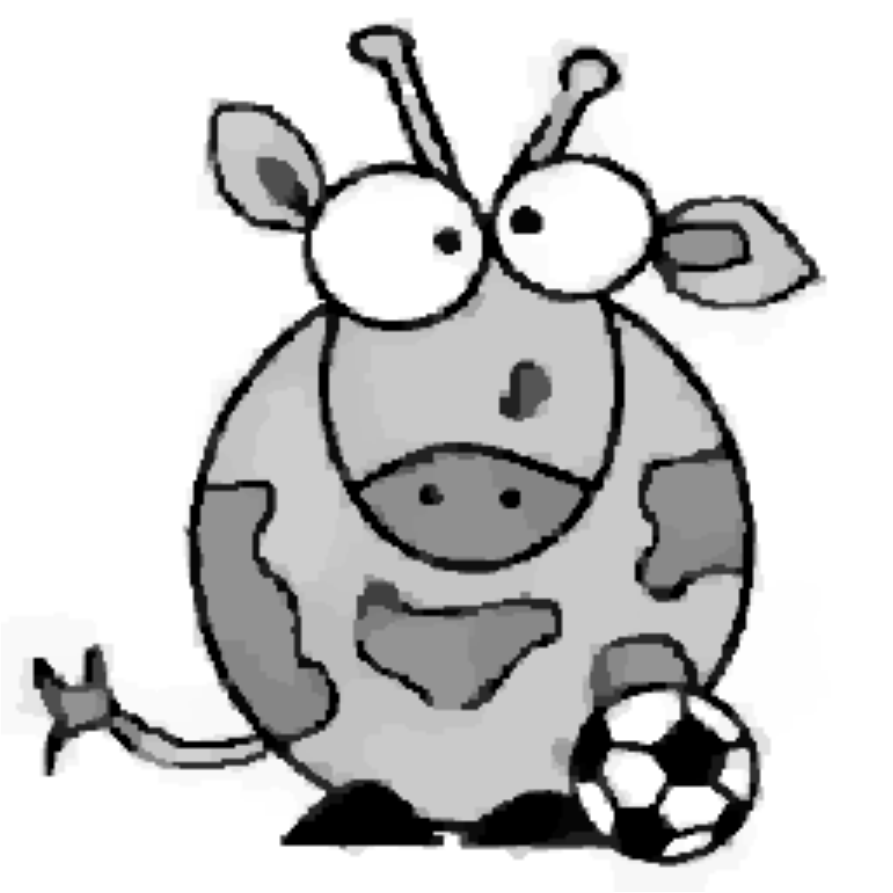}&
\includegraphics[width=0.3\textwidth]{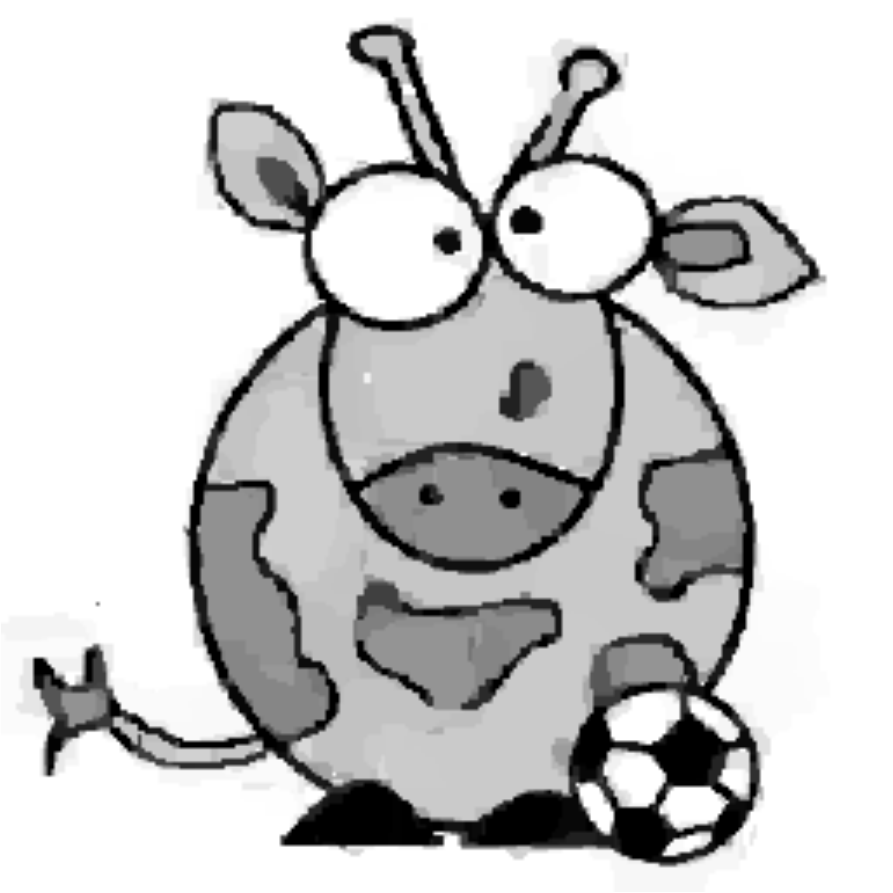}\\
$PSNR = 19.63$ & $PSNR = 22.75$ & $PSNR = 22.70$\\
\end{tabular}
\caption{Denoising of a cartoon image. First row: original image, noisy image corrupted by Gaussian noise.
Second row: TV reconstruction and debiasing using the Bregman distance and its infimal convolution, respectively.
The TV reconstruction recovers well the structures of the images but suffers from a loss of contrast, while the debiased solutions 
allow for a more accurate dynamic.\newline
\footnotesize{$^1$ The color image is provided in order to point out that it is indeed a giraffe and not a cow.}}
\label{fig:firstResults}
\end{figure*}

\section{Debiasing}
\label{sec:Debiasing}

Inspired by the above observations, let us define the following two-step-method for variational regularization on Banach spaces. 
At first we compute a solution $\ua$ of the standard variational method  
\begin{equation}
\label{eq:firstStep}
\begin{alignedat}{5}
&\text{1)} \; \; \; && \ua &&\in \arg \min_{u \in \X} \ \frac{1}{2} \Vert A u - f \Vert_\Y^2 + \alpha J(u),
\end{alignedat}
\end{equation}
where $A \colon \X \to \Y$ is a linear and bounded operator mapping from a Banach space $\X$ to a Hilbert space $\Y$, 
$J \colon \X \rightarrow \mathbb{R} \cup \{\infty\}$ denotes a convex and one-homogeneous regularization functional and $f \in \Y$. We point out that in the following, we will always make the standard identification $\Y^* = \Y$ without further notice.

The first-order optimality condition of \eqref{eq:firstStep} reads:
\begin{equation}
\label{eq:subgradient}
\pa = \frac{1}{\alpha} A^* (f - A\ua), \ \pa \in \partial J(\ua),
\end{equation}
and it is easy to show that this $\pa$ is unique (cf. Section \ref{subsec:well-definedness}, Thm. \ref{existencetheorem}). 
We use this subgradient to carry over information about $\ua$ to a second step.
In the spirit of the previous paragraph the idea is to perform a constrained minimization of the data fidelity term only:
\begin{equation}
\label{eq:secondStepInvConv}
\begin{alignedat}{4}
\begin{split}
&\text{2 a)} && \uhat \in \arg \min_{u \in \X} \; \frac{1}{2} \Vert A u - f \Vert_\Y^2\\
& \; && \; \qquad \text{ s.t. } \ICB(u,\ua)= 0.
\end{split}
\end{alignedat}
\end{equation}

If we reconsider the ad-hoc idea of $\lone$ or TV-type debiasing from the introduction, it can be beneficial to add a sign or direction constraint 
to the minimization, rather than a support condition only. 
This can be achieved by the use of a single Bregman distance. 
Hence it is self-evident to define the following alternative second step:
 \begin{equation}
\label{eq:secondStepBregman}
\begin{alignedat}{4}
\begin{split}
&\text{2 b)} && \uhat \in \arg \min_{u \in \X} \; \frac{1}{2} \Vert A u - f \Vert_\Y^2\\
& \; && \; \qquad \text{ s.t. } {D}^{\pa}_J(u,\ua)= 0.
\end{split}
\end{alignedat}
\end{equation}

We would like to point out that until now we only argued heuristically that the above method actually performs some kind of debiasing for specific problems. 
But since we are able to recover the debiasing method of \cite{deledalle} for $J(u) = \|\Gamma u\|_1$ as a special case, 
at least for this specific choice of regularization (and a finite-dimensional setting) our method is provably a debiasing in their sense. 

However, our method is much more general. 
Since in contrast to \cite{deledalle} it does not depend on a specific representation of $\ua$, it can theoretically be carried out for any suitable regularizer $J$. 
In particular, the method does not even depend on the specific choice of the data term. 
In order to obtain a unique subgradient $\pa$ from the optimality condition it is desirable e.g. to have a differentiable data fidelity,
but if we drop that condition, the data term is theoretically arbitrary. 
Since this generalization requires more technicalities, we focus on a squared Hilbert space norm in this paper in order to work out the basics of the approach.

Before we actually lay a theoretical foundation for our framework
and prove that our method indeed is a debiasing method, 
we show some motivating numerical results and prove the well-definedness of the method.

\subsection{A first illustration}

To give a first glimpse of the proposed method, we revisit the ROF-reconstruction model \eqref{eq:rof} from Section \ref{sec:motivation} and show some numerical results 
in one and two dimensions. 

Taking the subgradient $\pa$ of the TV reconstruction $\ua$ of the one-dimensional signal and performing our debiasing method, we obtain the results in Figure \ref{fig:1DTVsignal2}.
The second step restores the right height of the jumps and yields the same result as the Bregman iterations we performed in Section \ref{sec:motivation}. 

As a second example we perform denoising on a cartoon image corrupted by Gaussian noise. 
The first row of Figure \ref{fig:firstResults} shows the original image and its noisy version.
The left image in the second row is the denoising result obtained with the ROF-model \eqref{eq:rof}. 
We observe that noise has been reduced substantially, but some part of the contrast is lost. 
The second step of our method restores the contrast while keeping the structure of the first solution, 
yielding the two results depicted in the middle and on the right of the second row.

\subsection{Well-definedness of the method}
\label{subsec:well-definedness}
The aim of this section is to show that the method defined above is well-defined, i.e. that there always exists at least one solution to the problem. 
We fix the setup by restricting ourselves to conditions ensuring that the original variational problem \eqref{variationalmethod0} with quadratic data fidelity has a solution. 
The following result can be established by standard arguments:
\begin{theorem} \label{existencetheorem} 
Let $\Y$ be a Hilbert space,  $\X$ be the dual space of some Banach space $\mathcal{Z}$, such that the weak-star convergence in $\X$ is metrizable on bounded sets. 
Moreover, let $A: \X \rightarrow \Y$
be the adjoint of a bounded linear operator $B: \Y \rightarrow \mathcal{Z}$, 
$J$ be the convex conjugate of a proper functional on the predual space $\mathcal{Z}$, 
and let the map $u \mapsto \frac{1}{2}\Vert A u \Vert_\Y^2 + J(u)$ be coercive in $\X$. 
Then the variational problem \eqref{variationalmethod0} with data-fidelity \eqref{quadraticfidelity} has a minimizer $u_\alpha \in \X$ 
and there exists a subgradient $p_\alpha \in \partial J(u_\alpha)$ such that the optimality condition
\begin{equation}
	p_\alpha = \frac{1}\alpha A^* (f - Au_\alpha) = \frac{1}\alpha B(f-Au_\alpha)
	\label{eq:optcond}
\end{equation}
holds. Moreover, if $u_\alpha \neq \tilde u_\alpha$ are two minimizers, then $A u_\alpha = A\tilde u_\alpha$ and the corresponding subgradient is unique, i.e., 
\begin{equation*}
p_\alpha = \frac{1}\alpha  B (f - Au_\alpha) = \frac{1}\alpha  B ( f- A\tilde u_\alpha) = \tilde p_\alpha.
\end{equation*}
\end{theorem}
\begin{proof}
Since the functional $J$ is proper, there exists a nonempty sublevel set of the functional $u \mapsto \frac{1}2 \Vert A u -f \Vert_\Y^2 + \alpha J(u)$, 
and by the coercivity assumption this sublevel set is bounded. 
The Banach-Alaoglu theorem now implies precompactness of the sublevel set in the weak-star topology. 
Since the latter is metrizable on bounded sets, it suffices to show that the objective functional is sequentially weak-star lower semicontinuous 
in order to obtain existence of a minimizer. 
For the regularization functional $J$, this follows from a standard argument for convex conjugates of proper functionals along the lines of \cite{ekelandtemam}. 
The assumption $A=B^*$ guarantees further that $A$ is continuous from the weak-star topology in $\X$ to the weak topology of $\mathcal{Y}$ 
and the weak lower semicontinuity of the norm also implies the weak-star lower semicontinuity of the data fidelity. 
Those arguments together yield the existence of a minimizer.

The first equation of the  optimality condition for the subgradient $p_\alpha$ follows from the fact that the data fidelity is Fr\'echet-differentiable. 
From the argumentation in \cite[Remark 3.2]{Bredies2013} we see that the assumption $A=B^*$ furthermore implies that $A^*$ indeed maps to the predual space $\mathcal{Z}$ (and not to the bigger space $\mathcal{Z}^{**}$), such that \eqref{eq:optcond} holds true.
More precisely, this special property of $A^*$ is derived from the fact that $A$ is sequentially continuous from the weak-star topology of $\X$ to the weak(-star) topology of $\mathcal{Y}$, which implies that it posseses an adjoint which maps $\Y$ into $\mathcal{Z}$ regarded as a closed subspace of $\mathcal{Z}^{**}$ (note that the weak and the weak-star topology coincide on the Hilbert space $\Y$).
Consequently $\pa \in \mathcal{Z}$.

Finally, assume that $u_\alpha$ and $\tilde u_\alpha$ are two solutions, then we find
\begin{equation*}
p_\alpha = B w_\alpha, \quad w_\alpha = \frac{1}\alpha(f-Au_\alpha),
\end{equation*}
and an analogous identity for $\tilde p_\alpha$ respectively $\tilde u_\alpha$. Consequently, we have
\begin{equation*}
(w_\alpha - \tilde w_\alpha) + \frac{1}\alpha A(u_\alpha - \tilde u_\alpha) = 0.
\end{equation*}
Computing the squared norm of the left-hand side, we find
\begin{align*}
\Vert w_\alpha - \tilde w_\alpha\Vert_{\Y}^2  &+ \frac{2}\alpha \langle p_\alpha - \tilde p_\alpha, u_\alpha - \tilde u_\alpha \rangle\\ 
&+ \frac{1}{\alpha^2} \Vert A(u_\alpha - \tilde u_\alpha)\Vert_{\Y}^2 \quad = 0.
\end{align*}

The dual product can be expressed as a symmetric Bregman distance 
\begin{align*}
 D^{\mathrm{sym}}_J(\ua,\tilde u_\alpha) = D^{\tilde p_\alpha}_J(\ua,\tilde u_\alpha) + D^{\pa}_J(\tilde u_\alpha, u_\alpha).
\end{align*}
Hence all three terms are nonnegative and we find in particular 
$A u_\alpha = A\tilde u_\alpha$, $w_\alpha = \tilde w_\alpha$ and thus $p_\alpha = \tilde p_\alpha$. \QED
\end{proof}
 
By exploiting that $p_\alpha$ lies in the range of $B$ we can prove coercivity and subsequently existence for problem \eqref{eq:secondStepBregman}. 
In fact, we can give a more general result. 

\begin{theorem} 
\label{thm2}
Let the conditions of Theorem \ref{existencetheorem}  hold and let $p \in \partial J(0) \cap \mathcal{Z} \subset \X^*$ be such that there exists $w$ with 
$$ J^*\left(\frac{p - B w}\tau\right) = 0 $$
for some $0 <\tau < 1$.   
Then there exists a minimizer of 
\begin{align*}
\min_{u \in \X} \frac{1}{2}\| Au - f \|^2_{\Y} \text{ s.t. } J(u)-\langle p,u \rangle =0.
\end{align*}
\end{theorem}
\begin{proof}
Let $\mathcal{A} = \{ u \in \X ~|~ J(u)-\langle p,u \rangle =0 \}$ be the admissible set.
Since $0 \in \mathcal{A}$ we can look for a minimizer in the sublevel set 
\begin{align*}
  S = \big\{ u \in \mathcal{A} ~|~ \| Au -f \|_{\Y} \leq \| f \|_{\Y} \big\}.
\end{align*}
By the triangle inequality we have 
$\| Au \|_{\Y} \leq 2 \| f \|_{\Y}$ and hence $\frac{1}{2}\| Au \|^2_{\Y} \leq 2 \| f \|^2_{\Y}$ on $S$. Accordingly, $u \mapsto \frac{1}{2}\|Au\|^2_{\Y}$ is bounded on $S$. 
From the definition of the convex conjugate we know that for all $u \in \X, r \in \X^*$ we have 
\begin{align}
\langle r,u \rangle \leq J^*(r) + J(u).
\label{eq:def_convconj}
\end{align}
Hence for $u \in S$ we find 
\begin{align*}
 J(u) &= \langle p,u \rangle \\
 &= \langle p - Bw,u \rangle + \langle w, Au \rangle \\
 &\leq \langle \frac{p-B w}{\tau}, \tau u \rangle + \| w \|_{\Y} \|Au\|_{\Y}\\
 &\leq J^* \left(\frac{p-B w}{\tau} \right) + J(\tau u) + \| w \|_{\Y} \|Au\|_{\Y}
\end{align*}
which implies by the one-homogeneity of $J$ that
\begin{align*}
 J(u) \leq \frac{\| w \|_{\Y} ~ \| Au \|_{\Y}}{1-\tau}.
\end{align*}
Thus we obtain the boundedness of $u \mapsto \frac{1}{2}\| Au \|^2_{\Y} + J(u)$ on $S$. 
%
%
The remaining steps follow the proof of Theorem \ref{existencetheorem}. \QED
\end{proof}
 
Note that, provided that the operator $A$ fulfills the conditions of Theorem \ref{existencetheorem}, the assumptions of Theorem \ref{thm2} always hold for $p=\pa$ obtained from \eqref{eq:subgradient} 
with $w=\frac{1}\alpha(f-A\ua)$ and $\tau$ arbitrarily small, 
hence we conclude the existence of a minimizer $\uhat$ of \eqref{eq:secondStepBregman}. 

The situation for \eqref{eq:secondStepInvConv} is less clear, since there is no similar way to obtain coercivity. 
As we shall see in Section \ref{sec:BiasModelManifolds}, \eqref{eq:secondStepInvConv} consists in minimizing a 
quadratic functional over a linear subspace, which immediately implies the existence of $\uhat$ if $\X$ has finite dimensions.
In an infinite-dimensional setting we cannot provide an existence result in general, 
since there is neither a particular reason for the subspace to be closed nor for the quadratic functional to be coercive 
(in ill-posed problems we typically deal with an operator $A$ with nonclosed range).

\subsection{Optimal debiasing on singular vectors}

In the following we work out the behavior of the debiasing method on singular vectors \cite{benninggroundstates}, 
which represent the extension of the concept of classical singular value decomposition to nonlinear regularization functionals. 
According to \cite{benninggroundstates}, $u^\lambda \in \X$ is a singular vector if for some $\lambda > 0$ 
\begin{equation*}
	\lambda A^* A u^\lambda \in \partial J(u^\lambda)
\end{equation*}
holds. 
Without going too much into detail, singular vectors can be considered as generalized ``eigenfunctions'' of the regularization functional $J$. 
As such, they describe a class of exact solutions to problem $\eqref{eq:firstStep}$ in the following sense:

Let us consider a multiple $c u^\lambda$ of such a singular vector for $c > \lambda \alpha$. According to \cite{benninggroundstates}, the solution $\ua$
of the variational problem \eqref{eq:firstStep} for data $f=c Au^\lambda$ is given by
\begin{equation*}
	\ua = (c- \alpha \lambda) u^\lambda,
\end{equation*}
and the subgradient from the optimality condition is 
\begin{align*}
 \pa = \lambda A^* A u^\lambda \in \partial J(\ua).
\end{align*}
Hence $\ua$ recovers $c u^\lambda$ up to a (known) scalar factor $\alpha \lambda$ and shares a subgradient with $u^\lambda$. 
This means that the variational method leaves the singular vector basically untouched, which allows for its exact recovery. 
Intuitively, the quantity $-\lambda \alpha u^\lambda$ hence represents the bias of the variational method in this case, 
which should be removed by our debiasing method \eqref{eq:secondStepBregman}. 
And indeed we obtain $\uhat =c u^\lambda$ as a minimizer of \eqref{eq:secondStepBregman}, since 
\begin{align*}
 \| A\uhat - f\|_{\Y} = \|A(\uhat - cu^\lambda)\|_{\Y} = 0 
\end{align*}
and since $\uhat$ lies in the admissible set due to the shared subgradient.
If $A$ has trivial nullspace, $\uhat$ is even unique. 
Hence, the debiasing strategy leads to the exact reconstruction of the solution and corrects the bias 
$-\lambda \alpha u^\lambda$. 
Note that this is indeed an important result, since if the debiasing method failed for singular vectors it would be doubtful 
whether the method is reliable in general. 

Since the infimal convolution of Bregman distances is nonnegative and less or equal than either of the Bregman distances, 
it also vanishes at $\uhat =c u^\lambda$. 
In particular 
\begin{align*}
 \ICB(c u^\lambda, \ua) &\leq D_J^{\pa}( c u^\lambda, \ua) \\ 
 &= J(c u^\lambda) - \langle \pa, c u^\lambda \rangle = 0. 
\end{align*}
Consequently, $\uhat$ is also a solution of \eqref{eq:secondStepInvConv}.

\section{Bias and Model Manifolds}
\label{sec:BiasModelManifolds}
 
In the following we provide a more fundamental discussion of bias and decompositions obtained by debiasing methods. 
An obvious point to start is the definition of bias, which is indeed not always coherent in the imaging literature with the one in statistics. 

\subsection{Definitions of bias} \label{biasvariancesection}

We first recall the classical definition of bias in statistics.
Let $f$ be a realization of a random variable modeling a random noise perturbation of clean data $f^*=Au^*$, such that $\E [f] = f^*$.
If we consider a general canonical estimator $\Uhat(f)$, the standard definition of bias in this setup is given by 
\begin{align}
\begin{split}
  \mathbb{B}^{\text{stat}}(\Uhat) &= \E[u^*-\Uhat(f)]\\
  &= u^*-\E[\Uhat(f)].
  \label{eq:stat_bias}
\end{split}
\end{align}
Unfortunately, this bias is hard to manipulate for nonlinear estimators. 
Hence, we consider a deterministic definition of bias, which relies on the clean data $f^*$: 
\begin{align}
\begin{split}
 \mathbb{B}^*(\Uhat) &= \E[u^*-\Uhat(f^*)]= u^* - \Uhat(f^*)\\
 &=  u^*-\Uhat(\E[f]).
 \label{eq:det_bias}
\end{split}
\end{align}
We immediately note the equivalence of the two definitions in the case of linear estimators, but our computational experiments do not show a significant difference between $\mathbb{B}^{\text{stat}}$ and $\mathbb{B}^*$ even for highly nonlinear variational methods.
In general, the purpose of debiasing is to reduce the quantitative bias $B^d$, i.e. here the error between $u^*$ and $\Uhat(f^*)$ in an appropriate distance measure $d$: 
\begin{align*}
 B^d(\Uhat(f^*)) = d(\Uhat(f^*), u^*).
\end{align*}

Let us consider the specific estimator $\ua(f^*)$, 
i.e. the solution of problem \eqref{eq:firstStep} with clean data $f^*$.
As already argued in Section \ref{sec:motivation}, it suffers from a certain bias due to the chosen regularization. 
Following \cite{deledalle}, this bias can be decomposed into two parts.
The first part is related to the regularization itself, and it occurs if the assumption made by the regularization does not match the true object that we seek to recover. 
For example, trying to recover a piecewise linear object using TV regularization leads to the staircasing effect due to the assumption of a piecewise constant solution.
This part of the bias is unavoidable since it is inherent to the regularization, and it is referred to as {\em model bias}.
In particular, we cannot hope to correct it.

However, even if the regularity assumption fits, the solution still suffers from a systematic error due to the weight on the regularization.
For TV regularization for example, this is the loss of contrast observed in Section \ref{sec:motivation}.
This remaining part is referred to as {\em method bias}, and this is the part that we aim to correct.
As we shall see in the remainder of the section, the estimator $\ua(f^*)$ provides the necessary information to correct this bias.
Deledalle et al. \cite{deledalle} define an appropriate linear model subspace related to that estimator, on which the debiasing takes place. 
It allows to define the model bias as the difference between $u^*$ and its projection onto the model subspace. 
The remaining part of the difference between the reconstructed solution and $u^*$ is then the method bias. 
In the following we reintroduce the notion of model subspaces provided by \cite{deledalle} 
and further generalize it to the variational setting in infinite dimensions. 
The latter may imply the nonclosedness of the model subspace and hence nonexistence of the projection of $u^*$ onto it. 
Moreover, it seems apparent that in some nonlinear situations it might be more suitable to 
consider a model manifold instead of a linear space and we hence generalize the definition in this direction. 
We remark that the use of the term manifold is for technical reasons. 
As we shall see, the sets we consider in the course of the paper are for example (linear) subspaces or convex cones. 
The latter are not linear, but can be considered as manifolds with boundaries.
Therefore we shall use the term model manifold in general, and be more precise for particular instances of model manifolds.  

Let us first assume that we are already given an appropriate model manifold.

\begin{definition}
 Let ${\cal M}_{f^*}$ be a given model manifold and $d\colon \X \times \X \to [0,\infty)$ a distance measure. 
An estimator $\Uhat(f^*)$ of $u^*$ is a {\em debiasing} of $\ua(f^*)$ if $\Uhat(f^*) \in \mathcal{M}_{f^*}$ and 
\begin{align*}
d(\Uhat(f^*),u^*) < d(\ua(f^*),u^*). 
\end{align*}
If there exists a minimizer
\begin{align}
\label{eq:opt_debiasing}
\uhat(f^*) \in \arg \min_{v \in {\cal M}_{f^*}} d(v,u^*), 
\end{align}
we call it an {\em optimal debiasing}.
In any case, we define the {\em magnitude of the model bias} as 
\begin{align*}
 B_{\text{mod}}^d({\cal M}_{f^*}) = \inf_{v \in {\cal M}_{f^*}} d(v,u^*).
\end{align*} 
\end{definition}
Obviously the model bias only depends on the model manifold and for a given $\ua(f^*)$ it is hence, as already indicated, a fixed quantity that we cannot manipulate.
Instead we want to perform the debiasing on the manifold only, so we consider another bias for elements of ${\cal M}_{f^*}$ only.
Since according to the above definition there might exist more than one optimal debiasing, we shall from here on assume that we are given one of them. 
\begin{definition}
For a fixed optimal debiasing $\uhat(f^*)$ on ${\cal M}_{f^*}$, we define the {\em magnitude of the method bias} of $v \in {\cal M}_{f^*}$ 
 related to $\uhat(f^*)$ as 
 \begin{align*}
  B_{\text{meth}}^d(v) = d(v, \uhat(f^*)).
 \end{align*}
\end{definition}
The optimal debiasing $\uhat(f^*)$ obviously does not suffer from  method bias. 
Note that if the minimizer in \eqref{eq:opt_debiasing} does not exist, 
which can happen in particular in ill-posed problems in infinite dimensions,  
then the magnitude of the method bias is not well-defined or has to be set to $+\infty$.  

With these definitions at hand, we now aim to compute an optimal debiasing, i.e. the solution of \eqref{eq:opt_debiasing}. 
The remaining questions are how to choose an appropriate model manifold $\mathcal{M}_{f^*}$ and the distance measure $d$. 
We start with the latter.
An easy choice for the distance measure $d$ is a squared Hilbert space norm:
If the minimizer of \eqref{eq:opt_debiasing} exists, e.g. if $\mathcal{M}_{f^*}$ is nonempty, convex and closed,
the optimal debiasing $\uhat(f^*)$ is the (unique) projection of $u^*$ onto $\mathcal{M}_{f^*}$. 
We obtain a decomposition of the bias of any estimator $v \in \mathcal{M}_{f^*}$ into method and (constant) model bias: 
\begin{align*}
 v - u^* = \underbrace{ v - \uhat(f^*) }_{\text{method bias}} + \underbrace{ \uhat(f^*) - u^* }_{\text{model bias}}. 
\end{align*} 
In case $\mathcal{M}_{f^*}$ is a closed subspace of $\X$, this decomposition is even orthogonal, i.e. 
\begin{align*}
 B^d(v) &= \|v - u^*\|^2 \\
 &= \| v - \uhat(f^*) \|^2 + \| \uhat(f^*) - u^* \|^2 \\
 &= B_{\text{meth}}^d(v) + B_{\text{mod}}^d({\cal M}_{f^*}). 
\end{align*} 
Unfortunately, for general inverse problems with a nontrivial operator we do not know $u^*$ and hence cannot compute its projection onto $\mathcal{M}_{f^*}$. 
Instead we have access to the data $f^* = Au^*$ (or rather to one noisy realization $f$ of $f^*$ in practice, which we discuss later).
In order to make the bias (and the associated debiasing) accessible, we can consider bias through the operator $A$.
Hence the optimal debiasing comes down to computing the minimizer of \eqref{eq:opt_debiasing} with a distance defined over $A (\mathcal{M}_{f^*})$, i.e. 
\begin{align}
	\uhat(f^*) &= \arg \min_{v \in {\cal M}_{f^*}} \|Au^* - Av\|^2 \nonumber \\
	&= \arg \min_{v \in {\cal M}_{f^*}} \|f^* - Av\|^2. 
	\label{eq:opt_output_debiasing}
  \end{align}
Correspondingly, if such a minimizer $\uhat(f^*)$ exists, we measure the magnitude of model and method bias in the output space, rather than in image space, i.e. 
\begin{align*}
 & B_{\text{mod}}^d({\cal M}_{f^*}) = \inf_{v \in {\cal M}_{f^*}} \|Av - f^*\|^2, \\
 & B_{\text{meth}}^d(v) = \| A\uhat(f^*) - Av\|^2.
\end{align*} 
We can hence at least guarantee that the optimal debiasing has zero method bias in the output space.
For denoising problems without any operator ($A$ being the identity), or for $A$ invertible on $\mathcal{M}_{f^*}$ 
we obtain the equivalence of both approaches. 
In ill-posed inverse problems it is usually rather problematic to measure errors in the output space, since noise can also be small in that norm. 
Notice however that we do not use the output space norm on the whole space, but on the rather small model manifold, on which - if chosen appropriately - the structural components dominate. 
On the latter the output space norm is reasonable.

The main advantage of this formulation is that we are able to compute a minimizer of \eqref{eq:opt_output_debiasing}, since it is in fact a constrained least-squares problem with the data fidelity of \eqref{eq:firstStep}.
Its solution of course requires a proper choice of the underlying model manifold $\mathcal{M}_{f^*}$, which we discuss in the following. 

\subsection{Model manifolds}

In general, a model manifold can be characterized as the space of possible solutions 
for the debiasing step following the first solution $\ua(f)$ of the variational problem \eqref{eq:firstStep}.
As such it contains the properties of $\ua(f)$ that we want to carry over to the debiased solution. 
In the context of sparsity-enforcing regularization this is basically a support condition on the debiased solution. 
\subsubsection{Differential model manifolds}
Deledalle et al. \cite{deledalle} use the notion of Fr\'echet derivative to define their model subspace in a finite-dimensional setting.
We naturally generalize this concept using the directional derivative instead, and further extend it to infinite dimensions.
The following definitions can e.g. be found in \cite{shapiro}.
\begin{definition}
Let $\V$ and $\W$ be Banach spaces. 
A mapping $F \colon \V \to \W$ is called Fr\'echet differentiable at $x \in \V$ if there exists a linear and bounded operator $\Der F(x;\cdot) \colon \V \to \W$ such that 
 \begin{align*}
  \lim_{\|g\|_{\V} \to 0} \dfrac{\| F(x + g) - F(x) + \Der F(x;g) \|_\W}{\|g\|_{\V}} = 0.
 \end{align*}
\end{definition}
\begin{definition}
A mapping $F \colon \V \to \W$ is called directionally differentiable in the sense of G\^ateaux at $x \in \V$ if the limit 
\begin{align*}
 \der F(x;g) := \lim_{t \to 0^+} \dfrac{F(x + tg) - F(x)}{t} 
\end{align*}
exits for all $g \in \V$. 
\end{definition}
We can immediately deduce from the definition that, if the directional derivative $\der F(x;\cdot)$ exits, it is positively one-homogeneous in $g$, i.e. 
\begin{align*}
 \der F(x;\lambda g) = \lambda \der F(x;g) 
\end{align*}
for all $\lambda \geq 0$ and $g \in \V$. 
If it is linear in $g$, we call $F$ G\^ateaux differentiable at $x$.

Provided a unique and Fr\'echet differentiable map $f \mapsto \ua(f)$, Deledalle et al. \cite{deledalle} 
introduce the tangent affine subspace 
\begin{equation*}
\MF = \big\{\ua(f) + \Der \ua (f;g) ~|~  g \in \Y ~\big\},
\end{equation*}
where $\Der \ua (f;\cdot)\colon \Y \to \X$ is the Fr\'echet derivative of $\ua(f)$ at $f$.  
To be less restrictive, the easiest generalization of $\MF$ is to consider the directional derivative.
\begin{definition}
If the map $f \mapsto \ua(f)$ is directionally differentiable with derivative $\der u_\alpha (f; \cdot)$, we define
\begin{equation*}
	\MG =  \big\{\ua(f) + \der \ua(f;g) ~|~  g \in {\cal Y} ~ \big\}.
\end{equation*}
\end{definition}
Note that if the map is Fr\'echet differentiable, $\MG$ is a linear space and coincides with the model subspace $\MF$. 

We now derive a few illustrative examples that we use throughout the remainder of the paper.
In order to keep it as simple as possible, the easiest transition from the finite-dimensional vector space setting to infinite dimensions
are the $\ellp$-spaces of $p$-summable sequences: 
\begin{definition}
For $1 \leq p < \infty$ we define the spaces $\ellp$ of $p$-summable sequences with values in $\R^d$ by 
\begin{align*}
 \ellp(\R^d) = \big\{(x_i)_{i \in \N}, x_i \in \R^d : \sum_{i \in \N} | x_i|^p < \infty \big\}, 
\end{align*}
where $| \cdot |$ denotes the Euclidean norm on $\R^d$. 
For $p = \infty$ we define 
\begin{align*}
 \linf(\R^d) = \big\{(x_i)_{i \in \N}, x_i \in \R^d : \sup_{i \in \N} | x_i| < \infty \big\}.
\end{align*}
It is easy to show that $\ellp(\R^d) \subset \ell^q(\R^d)$ for $1 \leq p \leq q \leq \infty$.
In particular for $d=1$ we denote by $\lone$, $\ltwo$ and $\linf$ the spaces of summable, square-summable and bounded scalar-valued sequences.
\end{definition}

\begin{example}
\label{ex:aniso_shrinkage}
{\it Anisotropic shrinkage}.
Let $f \in \ltwo$ be a square-summable sequence.
The solution of 
\begin{align}
 \ua(f) \in \arg \min_{u \in \lone} \dfrac{1}{2} \|u-f\|_{\ltwo}^2 + \alpha \|u\|_{\lone}
 \label{eq:anisotropic_l1}
\end{align}
for $\alpha > 0$ is given by 
\begin{align*}
 [\ua(f)]_i = \begin{cases}
               f_i - \alpha ~ \sign(f_i), & |f_i| \geq \alpha, \\
               0, & |f_i| < \alpha.
              \end{cases}
\end{align*}
Its support is limited to where $|f_i|$ is above the threshold $\alpha$.
The directional derivative $\der \ua (f;g)$ of $\ua(f)$ into the direction $g \in \ltwo$ is given by 
\begin{align*}
 [\der \ua (f;g)]_i& \\ = &\begin{cases}
                       g_i, & |f_i| > \alpha \\
                       0, & |f_i| < \alpha \\
                       g_i, & |f_i| = \alpha, \sign(f_i) = \sign(g_i) \\
                       0, & |f_i| = \alpha, \sign(f_i) \neq \sign(g_i).
                      \end{cases}
\end{align*}
\begin{proof}
See Appendix \ref{app:ex_iso_shrinkage}.
\end{proof}
First, if we exclude the case $|f_i| = \alpha$, the directional derivative is linear, hence it is a G\^ateaux derivative. 
In fact it is even an infinite-dimensional Fr\'echet derivative, and the resulting model manifold coincides with the model subspace defined in finite dimensions in \cite{deledalle}:
\begin{equation*}
\MF = \big\{ u \in \ltwo  ~|~ \supp(u) \subset \supp(\ua(f)) \big\}.
\end{equation*}
The model manifold carries over information about the support of the first solution $\ua(f)$. 
Note that $\MF$ contains all elements of $\ltwo$ which share the same support as $\ua(f)$, 
but as well allows for zeros where $\ua(f) \neq 0$. 
In that sense $\ua(f)$ defines the maximal support of all $u \in \MF$. 
If we allow $|f_i|$ to be equal to $\alpha$, we obtain a larger set which allows for support changes in the direction of $f_i$ on the threshold: 
\begin{align*}
u \in \MG \Leftrightarrow u_i = \begin{cases}
\lambda \in \mathbb{R}, &  |f_i|> \alpha, \\
0, &  |f_i|< \alpha, \\
\lambda \geq 0 , & f_i = \alpha,   \\
\lambda \leq 0 , & f_i = - \alpha.
\end{cases} 
\end{align*}
Note that the case $|f_i| > \alpha$ reveals a remaining shortcoming of the definition via the directional derivative, 
e.g. if $f_i> \alpha$ it is counter-intuitive to allow for negative elements in $\MG$, while this is not the case for $f_i = \alpha$. 
The main reason appears to be the strong deviation of the linearization in such directions from the actual values of $\uai$, 
which is not controlled by the definition.
However, minimizing the data term over $\MG$ for the debiasing in Eq. \eqref{eq:opt_output_debiasing} forces the changes to have the right sign and the debiased solution $\uhat(f)$ corresponds to hard-thresholding:
\begin{align*}
 [\uhat(f)]_i = \begin{cases}
              f_i, & |f_i| \geq \alpha, \\
              0, & |f_i| < \alpha.
             \end{cases}
\end{align*}
Note that we as well maintain the signal directly on the threshold. 

\end{example}

We obtain analogous results for isotropic shrinkage, i.e. if $f \in \ltwo(\R^d)$ for $d > 1$. 
Since the computation of the derivative requires a little more work, we provide the results in Appendix \ref{app:ex_iso_shrinkage}.
A more interesting example is the model manifold related to anisotropic $\lone$-regularized general linear inverse problems. 
\begin{example}
\label{ex:aniso_l1}
{\it Anisotropic $\lone$-regularization}.
For $r > 1 $ let $A \colon \ell^r \to \ltwo$ be a linear and bounded operator and $f \in \ltwo$.
Consider the solution $\ua(f)$ of the $\lone$-regularized problem
\begin{align}\label{eq:aniso_l1}
 \ua(f) \in \arg \min_{u \in \lone} \frac{1}{2} \| Au - f\|_{\ltwo}^2 + \alpha \|u \|_{\lone},
\end{align}
where we assume that the solution is unique for data in a neighborhood of $f$. 
Computing the directional derivative directly is a more tedious task in this case, but computing the 
model manifold $\MG$ is actually easier via a slight detour. 

Let $\ua(f)$ be the solution for data $f$ and $\ua(\tilde{f})$ the solution for data $\tilde f$. 
First, we derive an estimate on the two subgradients from the optimality conditions
\begin{align*}
0 &= A^* ( A \ua(f) - f) + \alpha \pa, \hspace{0.5em} \pa \in \partial \| \ua(f) \|_{\lone}, \\
0 &= A^* ( A \ua(\tilde{f}) - \tilde{f}) + \alpha \tilde{p}_{\alpha}, \hspace{0.5em} \tilde{p}_{\alpha} \in \partial \| \ua(\tilde{f}) \|_{\lone}.  
\end{align*}
Following the ideas of \cite{bur07}, we subtract the two equations and multiply by $\ua(f) - \ua(\tilde{f})$ to arrive at 
\begin{align*}
 \|A &\ua(f) - A \ua(\tilde{f}) \|_{\ltwo}^2 \\ 
 &+ \alpha \langle \pa - \tilde{p}_{\alpha}, \ua(f) - \ua(\tilde{f}) \rangle \\ 
 &= \langle f - \tilde{f},A \ua(f) - A \ua(\tilde{f}) \rangle \\
 &\leq \frac{1}{2} \| f - \tilde{f} \|_{\ltwo}^2 + \frac{1}{2} \|A \ua(f) - A \ua(\tilde{f})\|_{\ltwo}^2. 
\end{align*}
The last line follows from the Fenchel-Young inequality, obtained by applying the inequality \eqref{eq:def_convconj} to $J = \frac{1}{2} \|\cdot\|_{\ltwo}^2$.
The second term on the left hand side is a symmetric Bregman distance, i.e. the sum of two Bregman distances (cf. \cite{bur07}), hence positive. 
Leaving it out and rearranging then yields
\begin{align}
 \|A \ua(f) &- A \ua(\tilde{f}) \|_{\ltwo} \leq \| f - \tilde{f} \|_{\ltwo}.
 \label{eq:standard_estimate}
\end{align}
Since $A^* \colon \ltwo \to \ell^s$, where $s^{-1} + r^{-1} = 1$, $A^*$ is also continuous to $\linf$, hence we derive the following estimate from the optimality conditions:
\begin{align*}
 &\| \pa - \tilde p_\alpha \|_{\linf} \\ 
 &= \frac{1}{\alpha} \|A^*(A \ua(f) - A \ua(\tilde{f})) - A^*(f - \tilde{f}) \|_{\linf} \\
 &\leq \frac{\|A^*\|}{\alpha} \|A \ua(f) - A \ua(\tilde{f})\|_{\ltwo} \\
 &+ \frac{\|A^*\|}{\alpha} \| f - \tilde{f} \|_{\ltwo} \\
 &\leq \frac{C}{\alpha} \| f - \tilde f \|_{\ltwo}, 
\end{align*}
where we used \eqref{eq:standard_estimate} for the last inequality and $\|\cdot \|$ denotes the operator norm.

Next, we note that since $A^*$ maps to $\ell^s$ and $\pa$ and $\tilde p_\alpha$ lie in its range, they necessarily have to converge to zero.
This implies the existence of $N \in \N$ such that for all $i \geq N$ both $|(\pa)_i|$ and $|(\tilde p_\alpha)_i|$ are strictly smaller than $1$ and hence $\ua(f)$ and $\ua(\tilde{f})$ vanish for all $i \geq N$. 
As a consequence it is sufficient to consider a finite dimensional setting for the following reasoning.

In view of the subdifferential of the $\lone$-norm,
\begin{align*}
 \partial \|u\|_{\lone} = \{ &p \in \linf : \|p\|_{\linf} \leq 1, \\
 &p_i = \sign(u_i) \text{ for } u_i \neq 0 \},
\end{align*}
we have to consider several cases. 
If $\uai = 0$ and $|(\pa)_i| < 1$, we derive from
\begin{align*}
 |(\tilde{p}_{\alpha})_i| 
 &\leq |(\tilde{p}_{\alpha})_i - (\pa)_i| + |(\pa)_i| \\
 &\leq \frac{C}{\alpha} \| f - \tilde f \|_{\ltwo} + |(\pa)_i|, 
\end{align*}
that if $\| f - \tilde f \|_{\ltwo}$ is sufficiently small, then $|(\tilde{p}_{\alpha})_i| < 1$. 
Hence $[\ua(\tilde{f})]_i = 0$, and the derivative related to the perturbed data $\tilde{f}$ vanishes.
In case $\uai = 0$ and $(\pa)_i = 1$, by a similar argument $(\tilde{p}_{\alpha})_i \neq -1$ and thus $[\ua(\tilde{f})]_i \geq 0$ and $[\der \ua(f;g)]_i \geq 0$.
Analogously, $[\der \ua(f;g)]_i \leq 0$ if $\uai = 0$ and $(\pa)_i = -1$.
If $\uai \neq 0$, the directional derivative is an arbitrary real number depending on the data perturbation. 
Summing up we now know that every directional derivative is an element $v \in \lone$ fulfilling 
\begin{align}
\label{eq:viproperties}
v_i = \begin{cases}
  0, &|(\pa)_i| < 1,\\
\lambda \geq 0, &   (\ua)_i = 0, (\pa)_i = 1 , \\
\lambda \leq 0, &  (\ua)_i = 0, (\pa)_i = -1.
\end{cases} 
\end{align}
Note again that $v$ differs from $0$ only for a finite number of indices.
Hence, for $v$ satisfying \eqref{eq:viproperties}, we can pick $t > 0$ sufficiently small such that $\pa$ is a subgradient of $\tilde u = \ua(f) + t v$.
Indeed, for example if $(\ua)_i = 0$ and $(\pa)_i = 1$, then $v_i \geq 0$, so $\tilde u \geq 0$, and hence $(\pa)_i$ fulfills the requirement of a subgradient of $\tilde u$.
The other cases follow analogously.
Then from the optimality condition of $\ua(f)$ we get:
\begin{align*}
 & \ A^*(A\ua(f) -f) )+ \alpha\pa = 0\\
 \Leftrightarrow & \ A^*(A(\underbrace{\ua(f)+tv}_{\tilde u}) - (f+tAv)) + \alpha\pa = 0.
\end{align*}
We then deduce that $\tilde u$ is a minimizer of problem \eqref{eq:aniso_l1} with data $\tilde{f} = f + tAv$.
Hence, there exists a data perturbation such that $v$ 
is the directional derivative of $\ua(f)$.
Putting these arguments together we now know that $ u \in \MG$ if and only if
\begin{align*}
u_i = \begin{cases}
\lambda \in \R, & \uai \neq 0, \\
0, & \uai = 0, |(\pa)_i| < 1, \\
\lambda \geq 0, & \uai = 0, (\pa)_i = 1, \\
\lambda \leq 0, & \uai = 0, (\pa)_i = -1.
\end{cases}
\end{align*}
It is not surprising that $\MG$ has a similar structure as the model manifold for the anisotropic shrinkage in Example \ref{ex:aniso_shrinkage}. 
It allows for arbitrary changes on the support of $\ua(f)$ and permits only zero values if $\uai = 0$ and $|(\pa)_i| < 1$. 
If we exclude the case where $|(\pa)_i| = 1$ even though $\uai$ vanishes, 
debiasing on $\MG$ effectively means solving a least-squares problem with a support constraint on the solution. 
But we again find an odd case where changes are allowed outside of the support of the initial solution $\ua(f)$. 
It occurs when $|(\pa)_i| = 1$ even though $\uai$ vanishes, which seems to be the indefinite case. 
However, it has been argued in \cite{moeller14} that a subgradient equal to $\pm 1$ is a good indicator of support, hence it is reasonable to trust the subgradient in that case.
\end{example}
\subsubsection{Variational model manifolds}

As we have shown so far, the appropriate use of a derivative can yield suitable spaces for the debiasing. 
However, for already supposedly easy problems such as the latter example the explicit computation of such spaces or of the derivatives can be difficult or impossible. 
And even if it is possible, there remains the question of how to effectively solve the debiasing on those spaces, both theoretically and numerically.

On the other hand, the latter example implies that a subgradient of the first solution rather than the solution itself can provide the necessary information for the debiasing. 
This naturally leads us to the idea of Bregman distances in order to use the subgradient in a variational debiasing method. 
And indeed we show that the associated manifolds are closely related, and that they link the concept of model manifolds to the already presented debiasing method from Section \ref{sec:Debiasing}. 
Furthermore, this does not only provide a theoretical framework, but also numerical solutions to perform debiasing in practice, even for more challenging problems.

In the following we introduce related manifolds motivated by the variational problem itself. 
The optimality condition of the variational problem \eqref{eq:firstStep} defines a unique map 
$f \mapsto \pa \in \partial J(\ua)$, which allows us to consider the following manifolds.
We drop the dependence of $\ua$ on $f$ for the sake of readability.

\begin{definition}
For $\pa \in \partial J(\ua)$ defined by \eqref{eq:subgradient} we define
\begin{align*}
\MB &= \big\{ u \in \X ~|~ D_J^{\pa} (u, \ua) = 0 \big\}, \\
\MIC &= \big\{ u \in \X ~|~ \ICB(u, \ua) = 0 \big\}.
\end{align*}
\end{definition}

In order to assess the idea of the above manifolds, we first revisit the anisotropic shrinkage problem of Example \ref{ex:aniso_shrinkage}. 

\begin{example}
 \label{ex:var_aniso_shrinkage}
 {\it Anisotropic shrinkage.}
The optimality condition of problem (\ref{eq:anisotropic_l1}) yields the subgradient 
\begin{align}
(\pa)_i = \dfrac{f_i-(\ua)_i}{\alpha} = \begin{cases}
\sign(f_i), & |f_i| \geq \alpha, \\
\frac{f_i}{\alpha}, & |f_i| < \alpha,
\end{cases}
\label{eq:subgrad_l1}
\end{align}
and for $J = \| \cdot \|_{\lone}$ the Bregman distance takes the following form:
\begin{align*}
 D_{\lone}^{\pa}(u,\ua) 
 &= \| u \|_{\lone} - \langle \pa,u \rangle \\
 &= \sum_{i \in \N} |u_i| - (\pa)_i u_i \\
 &= \sum_{i \in \N} (\sign(u_i) - (\pa)_i ) u_i. 
\end{align*}
A zero Bregman distance thus means that either $u_i = 0$ or $\sign(u_i) = (\pa)_i$. 
Having a closer look at the subgradient \eqref{eq:subgrad_l1}, we observe that if $|f_i| < \alpha$, then $|(\pa)_i| < 1$.
Hence the latter condition cannot be fulfilled, so in this case $u_i$ has to be zero.
We can thus characterize the model manifold related to a zero Bregman distance as:
\begin{align*}
u \in \MB \Leftrightarrow u_i = \begin{cases}
\lambda ~ \sign(f_i), \lambda \geq 0, & |f_i| \geq \alpha, \\
0, & |f_i| < \alpha.
\end{cases}
\end{align*}
As for $\MG$, the model manifold $\MB$ fixes the maximum support to where $|f_i| \geq \alpha$. 
However, $\MB$ only allows for values on the support sharing the same sign as $f_i$ (respectively $(\ua)_i$).

By adapting the proof of \cite{moeller14}, we obtain a similar result for the infimal convolution of Bregman distances, without the restriction on the sign: 
\begin{align*}
 \ICBl(u,\ua) 
 &= [D_{\lone}^{\pa}(\cdot,\ua) \Box D_{\lone}^{-\pa}(\cdot,-\ua)](u) \\
 &= \sum_{i \in \N} (1 - |(\pa)_i|)|u_i|.
\end{align*}
For this infimal convolution to be zero we need either $u_i=0$ or $|(\pa)_i|=1$.
By the structure of the subgradient $\pa$ we thus find
\begin{align*}
u \in \MIC \Leftrightarrow u_i = \begin{cases}
\lambda \in \R, & |f_i| \geq \alpha, \\
0, & |f_i| < \alpha.
\end{cases}
\end{align*}
Hence we observe the following connection between the manifolds: 
\begin{align*}
\MB \subset \MG \subset \MIC.
\end{align*} 

Note that the manifold $\MG$ related to the directional derivative seems to be the odd one of the three. 
While allowing for arbitrary sign for $|f| > \alpha$, 
it only allows for changes in the direction of $f$ directly on the threshold. 
In that sense, $\MB$ and $\MIC$ seem to be more suitable in order to either include or exclude the sign-constraint. 
A closer inspection at the manifolds reveals that $\MIC$ is a linear space, as we further elaborate on in the next subsection. 
In this case it is actually even the span of $\MB$, which is however not true in general. 
This can e.g. be seen from the next example of isotropic TV-type regularization.
\end{example}

\begin{example}
 \label{ex:iso_TV}
 {\it Isotropic TV-type regularization.}
 Let $A \colon \ltwo(\R^n) \to \ltwo(\R^d)$ and $\Gamma \colon \ltwo(\R^n) \to \lone(\R^m)$ be linear and bounded operators and $J(u) = \|\Gamma u \|_{\lone(\R^m)}$ for $d,m,n \in \N$.
 We aim to find the variational model manifolds for the debiasing of the solution 
 \begin{align*}
  \ua \in \argmin_{u \in \ltwo(\R^n)} \frac{1}{2} \| Au - f \|_{\ltwo(\R^d)} + \alpha \|\Gamma u \|_{\lone(\R^m)}.
 \end{align*}
Given the (unique) subgradient $\pa \in \partial J(\ua)$ from the optimality condition, the chain rule for subdifferentials \cite[p. 27]{ekelandtemam} implies the existence
of a $\qa \in \partial \| \cdot \|_{\lone(\R^m)} (\Gamma \ua)$ such that $\pa = \Gamma^* \qa$ and 
\begin{align*}
 D_J^{\pa}(u,\ua) = D_{\lone(\R^m)}^{\qa}(\Gamma u, \Gamma \ua).
\end{align*} 
If we denote the angle between $(\Gamma u)_i$ and $(\qa)_i$ by $\varphi_i$, the Bregman distance reads:
\begin{align*}
 D_J^{\pa}(u,\ua) 
 &= D_{\lone(\R^m)}^{\qa}(\Gamma u, \Gamma \ua) \\
 &= \sum_{i \in \N} |(\Gamma u)_i| - (\qa)_i \cdot (\Gamma u)_i \\
 &= \sum_{i \in \N} |(\Gamma u)_i| \big(1 - \cos(\varphi_i) |(\qa)_i| \big) 
\end{align*}
For a zero Bregman distance we can distinguish two cases: If $|(\qa)_i| < 1$, then $(\Gamma u)_i$ has to be zero.
If $|(\qa)_i| =1$, then either $(\Gamma u)_i = 0$ or $\cos(\varphi_i) = 1$, hence $(\Gamma u)_i = \lambda (\qa)_i$ for $\lambda \geq 0$.
Hence the model manifold $\MB$ is given by 
\begin{align*}
 u \in \MB &\Leftrightarrow \\
 (\Gamma u)_i &= \begin{cases}
					  \lambda (\qa)_i, \lambda \geq 0, & |(\qa)_i| = 1, \\
					  0, & |(\qa)_i| < 1.
					  \end{cases}
\end{align*}
In particular, if $(\Gamma \ua)_i \neq 0$, then by the structure of the $\lone(\R^m)$-subdifferential 
we know that $(\qa)_i = \frac{(\Gamma \ua)_i}{|(\Gamma \ua)_i|}$ and thus $(\Gamma u)_i = \mu (\Gamma \ua)_i$ for some $\mu \geq 0$. 
So provided that $|(\qa)_i| < 1$ whenever $(\Gamma \ua)_i = 0$ we find 
\begin{align*}
 u \in \MB &\Leftrightarrow \\
 (\Gamma u)_i &= \begin{cases}
					  \mu (\Gamma \ua)_i, \mu \geq 0, & (\Gamma \ua)_i \neq 0, \\
					  0, & (\Gamma \ua)_i = 0.
					  \end{cases}
\end{align*}
Performing the debiasing on the latter manifold hence means minimizing the data term with a support and direction constraint on the gradient of the solution. 
This in particular allows to restore the loss of contrast which we have observed for TV regularization in Section \ref{sec:motivation}. 
Note that the condition $|(\qa)_i| < 1 \Leftrightarrow (\Gamma \ua)_i = 0$ excludes the odd case where the subgradient seems to contain more information than the first solution, 
as already seen in Example \ref{ex:aniso_l1}.

In the above illustration of the model manifold, the debiasing seems to rely on the choice of $\qa$, which is obviously {\it not} unique. 
However, in practice we still use the unique subgradient $\pa$ from the optimality condition which avoids the issue of the choice of a ``good'' $\qa$. 

The computation of $\MIC$ is a little more difficult in this case, since we cannot access an explicit representation of the functional $\ICB(\cdot, \ua)$. 
However, since 
\begin{align*}
 \ICBlisoq(\Gamma u, \Gamma \ua) \leq \ICB(u,\ua) 
\end{align*}
(cf. Appendix \ref{app:infconv}, Thm. \ref{thm:smaller}), we can instead use the infimal convolution of two $\lone(\R^m)$-Bregman distances to illustrate the model manifold.
We have (cf. Appendix \ref{app:infconv}, Thm. \ref{thm:infconv_l1})
\begin{align*}
 \ICBlisoq(\Gamma u, \Gamma \ua) = \sum_{i \in \N} G((\Gamma u)_i, (\qa)_i)
\end{align*}
with $G \colon \R^m \times \R^m \to \R$ defined as
\begin{align*}
 &G((\Gamma u)_i, (\qa)_i) =\\
 & \begin{cases}
                             |(\Gamma u)_i| (1 - |\cos(\varphi_i)| |(\qa)_i|), \\
                             \hspace{8.5em} \text{ if } |(\qa)_i| < |\cos(\varphi_i)|, \\
                             |(\Gamma u)_i| | \sin(\varphi_i)| \sqrt{1 - |(\qa)_i|^2}, \\
                             \hfill \text{ if } |(\qa)_i| \geq |\cos(\varphi_i)|.
                            \end{cases}
\end{align*}
For $G$ to be zero we once again distinguish two situations. 
If $|(\qa)_i| < 1$, in the first case $G$ can only vanish if $(\Gamma u)_i = 0$. 
In the second case, since $1 > |(\qa)_i| \geq |\cos(\varphi_i)|$, we infer $\varphi_i \notin \{0,\pi\}$, and hence neither the sinus nor the square root can vanish. 
This means once again that $(\Gamma u)_i = 0$. 
If $|(\qa)_i| = 1$, we can only be in the second case and $G$ vanishes independently of $(\Gamma u)_i$.
Thus $(\Gamma u)_i$ can be arbitrary. 
Putting the arguments together, we find 
\begin{align*}
 u \in \MIC &\Rightarrow \ICBlisoq(\Gamma u, \Gamma \ua) = 0 \\
 &\Leftrightarrow (\Gamma u )_i = \begin{cases}
                                             \lambda \in \R^m, & |(\qa)_i| = 1, \\
                                             0, &  |(\qa)_i| < 1.
                                            \end{cases}
\end{align*}
This is indeed not the span of $\MB$, but it instead allows for arbitrary elements if $|(\qa)_i| = 1$. 
From this example, we cannot immediately state that $\MB \subset \MIC$, because so far we only know that $\MB$ as well as $\MIC$ are subsets of the set $\lbrace u \in \X ~|~ \ICBlisoq(\Gamma u, \Gamma \ua) = 0\rbrace$.
However, in the next subsection we see that $\MB \subset \MIC$ is indeed true and it is actually a general property of the variational model manifolds.

Note that we gain the same support condition on the gradient as for $\MB$, but allow for arbitrary gradient directions on the support, which intuitively does not seem restrictive enough. 
However, in practice for the debiasing the direction is not arbitrary, but the data term decides, so we can expect a similar result for debiasing in $\MB$ and $\MIC$. 
Indeed the numerical studies in Section \ref{sec:exp} confirm these expectations. 
\end{example}

\subsection{Properties of variational model manifolds}

In the following we discuss some properties of the variational manifolds $\MB$ and $\MIC$.  
All results are general and do not depend on the particular choice of a subgradient, so we drop the dependence on $f$ in the notation of the manifolds. 
Let $v \in \X$ and $p \in \partial J(v)$. 
We start with a result on the structure of $\mb$:
\begin{theorem}
The set 
\begin{equation*}
\mb = \{u\in\X~|~D_J^p(u,v) = 0 \}
\end{equation*}
is a nonempty convex cone. 
\end{theorem} 
\begin{proof}
The map $u \mapsto D_J^p(u,v)$ is convex and nonnegative, hence 
\begin{equation*}
\{u~|~D_J^p(u,v) = 0 \} = \{u~|~D_J^p(u,v) \leq 0 \} 
\end{equation*}
is convex as a sublevel set of a convex functional. Moreover, for each $c \geq 0$ we have
\begin{equation*}
D_J^p(c u,v)=c~D_J^p(u,v), 
\end{equation*}
i.e. if $u$ is an element of the set, then every positive multiple $c u$ is an element, too. Hence it is a convex cone. Since $D_J^p(v,v)=0$ it is not empty.\QED
\end{proof} 

The structure of $\mic$ is even simpler; as announced in a special example above it is indeed a linear space:
\begin{theorem}\label{subspacelemma}
The set 
\begin{equation*}
\mic = \{u\in\X|[D_J^p(\cdot,v)\Box D_J^{-p}(\cdot,-v)](u) = 0 \}
\end{equation*}
is a nonempty linear subspace of $\X$.
\end{theorem} 
\begin{proof}
By analogous arguments as above we deduce the convexity and since 
\begin{align*}
\mathrm{ICB}_J^p(0,v) &= \inf_{\phi + \psi = 0} D_J^p(\phi,v) + D_J^{-p}(\psi,-v)\\ &\leq D_J^p(v,v) + D_J^{-p}(-v,-v) = 0
\end{align*} 
the set is not empty. For arbitrary $c \in \R \setminus\{0\}$ we have
\begin{alignat*}{3}
&\mathrm{ICB}_J^p(cu,v)\\  &= \inf_z  J(cu-z) + J(z)  - \langle p, cu - 2z \rangle \\
&= \abs{c}\inf_w  J(u-w) + J(w) - \langle p, u -2w \rangle, 
\end{alignat*} 
where we use the one-to-one transform $z=cw$ for $c > 0$ and $z=c(u-w)$ for $c < 0$.
This  implies that $\mathrm{ICB}_J^p(cu,v)=0$ if $\mathrm{ICB}_J^p(u,v) = 0$. 
Now let $u_1,u_2 \in \mic$, i.e. $\mathrm{ICB}_J^p(u_1,v) = 0$ and $\mathrm{ICB}_J^p(u_2,v) = 0$.
Then by definition of the infimum there exist sequences $(z_1^n)_{n\in\N}, (z_2^n)_{n\in\N}$ such that 
\begin{alignat*}{4}
 \lim_{n \to \infty} J(u_1 - z_1^n) &+ J(z_1^n) &&- \langle p, u_1 - 2 z_1^n \rangle &&= 0, \\
 \lim_{n \to \infty} J(u_2 - z_2^n) &+ J(z_2^n) &&- \langle p, u_2 - 2 z_2^n \rangle &&= 0.
\end{alignat*}
Due to its convexity and absolute one-homogeneity $J$ is a seminorm and thus satisfies the triangle inequality:
\begin{align*}
 &\mathrm{ICB}_J^p(u_1 + u_2,v)  \\
 &= \inf_z J(u_1 + u_2 - z) + J(z) \\ 
 &\hspace{3em} - \langle p, u_1 +u_2 - 2z \rangle \\ 
 &\leq J(u_1 + u_2 - z_1^n - z_2^n) + J(z_1^n + z_2^n) \\ 
 &\hspace{3em} - \langle p, u_1 + u_2 - 2z_1^n - 2z_2^n \rangle \\
 &\leq J(u_1 - z_1^n) + J(z_1^n) - \langle p, u_1 - 2 z_1^n \rangle \\ 
 & \hspace{3em} + J(u_2 - z_2^n) + J(z_2^n) - \langle p, u_2 - 2 z_2^n \rangle \\
 & \to 0, \text{ for } n \to \infty.
\end{align*}
Hence $ u_1 + u_2 \in \mic$ and $\mic$ is a linear subspace.
\QED
\end{proof} 


As one may expect from the fact that the infimal convolution is a weaker distance than the original Bregman distance, 
we obtain an immediate inclusion between the corresponding manifolds:
\begin{lemma}
$\mb \subset \mic$.
\end{lemma}
\begin{proof}
Let $u \in \mb$, i.e. $D_J^p(u,v) = 0$. For $c \geq 0$ we have
\begin{equation*}
D_J^{-p}(-c u,-v)=c~D_J^p(u,v). 
\end{equation*}
Thus we deduce
\begin{align*}
\mathrm{ICB}_J^p(u,v)
&= \inf_{\phi + \psi = u} D_J^p(\phi,v) + D_J^{-p}(\psi,-v) \\
&\leq D_J^p(2u,v) + D_J^{-p}(-u,-v) \\
&= 2 D_J^p(u,v) + D_J^p(u,v)
= 0.
\end{align*}
The assertion follows by the nonnegativity of the maps $u \mapsto D_J^p(u,v)$ and $u \mapsto D_J^{-p}(u,-v)$. 
Note that for $p \neq 0$ the subset is proper, since e.g. $-v \in \mic$ but $-v \notin \mb$.\QED
\end{proof}

We finally elaborate on the importance of absolute one-homogeneity of $J$ for our approach 
(respectively also other debiasing approaches as in \cite{deledalle}), such that the subdifferential can be multivalued. 
Otherwise the model manifolds may just contain a single element and debiasing 
in this manifold cannot produce any other solution.
This is e.g. the case for a strictly convex functional.

\begin{lemma}
Let $J$ be strictly convex. Then $\mb$ is a singleton.
\end{lemma}
\begin{proof}
For strictly convex $J$ the mapping  $u \mapsto D_J^p(u,v)$ is strictly convex as well, hence $D_J^p(u,v) = 0$ 
if and only if $u = v$ and $\mb = \{ v \}$.
\QED
\end{proof}

However, one can easily exclude this case since our assumption of $J$ being one-homogeneous guarantees 
that it is not strictly convex.

\subsection{Bias-variance estimates} 

Another justification for the deterministic definition of bias as well as our choice for the distance measure in Section \ref{biasvariancesection} can be found in the variational model itself. 
In order to derive quantitative bounds for bias and variance in a variational model, we start with the Tikhonov regularization (Ridge regression) model related to the functional 
 $J(u) = \frac{1}2 \Vert u \Vert_{\X}^2$.
 The optimality condition for this problem is given by
 \begin{equation*} 
 	A^* (A \ua(f) -f) + \alpha \ua(f) = 0.
 \end{equation*}
 We easily see that there exists $w_\alpha = \frac{1}\alpha (f- A\ua(f))$ such that $\ua(f) = A^* w_\alpha$ and
 \begin{align*}
  A\ua(f) - Au^* + \alpha w_\alpha =  f- Au^*.
 \end{align*}
 Now let us assume that a source condition $u^* \in \mathrm{Im}[A^*]$ holds, i.e. $u^* = A^* w^*$ for some $w^*$. 
 In this case we can subtract $\alpha w^*$ on both sides and take a squared norm to arrive at 
 \begin{align*}
 &\Vert A\ua(f) - Au^* \Vert_{\Y}^2  + \alpha^2 \Vert w_\alpha - w^* \Vert_{\Y}^2\\ 
 & \hspace{5.33em} + 2 \alpha \langle  A\ua(f) - Au^* , w_\alpha - w^*  \rangle \\= \ 
 &\Vert f- Au^* \Vert_{\Y}^2 + \alpha^2 \Vert w^* \Vert_{\Y}^2 \hspace{0.09em} - \hspace{0.09em} 2\alpha \langle f - Au^*, w^* \rangle. 
 \end{align*}
 Now taking the expectation on both sides and using $\E[f] = f^* =Au^*$ we find 
 \begin{align}
 \label{eq:BiasVarianceEstimate}
 &\E[  \Vert A\ua(f) - Au^* \Vert_{\Y}^2 ] + \alpha^2 \E [\Vert w_\alpha - w^* \Vert_{\Y}^2 ] \nonumber \\
 &\hspace{8em} + 2 \alpha \E[ \Vert \ua(f) - u^* \Vert_{\X}^2] \nonumber \\ = \
 & \E[\Vert f- Au^* \Vert_{\Y}^2 ] + \alpha^2 \Vert w^* \Vert_{\Y}^2. 
 \end{align} 
 The left-hand side is the sum of three error terms for the solution measured in different norms: in the output space, 
 the space of the source element, and the original space used for regularization. 
 All of them can be decomposed in a bias and a variance term, e.g.
 \begin{align*}
 &\E[ \Vert \ua(f) - u^* \Vert_{\X}^2]\\ = &\Vert \E[\ua(f)]- u^* \Vert_{\X}^2  +  \E[ \Vert \ua(f) - \E[\ua(f)] \Vert_{\X}^2] .
 \end{align*}  
 The term $\E[\Vert f- Au^* \Vert_{\Y}^2 ]$ in \eqref{eq:BiasVarianceEstimate} is exactly the variance in the data. 
 As a consequence $\alpha \Vert w^* \Vert_{\Y}$ 
 measures the bias in this case. 
 Note that in particular for zero variance we obtain a direct estimate of the bias via $\alpha \Vert w^* \Vert_{\Y}$.
 
 \begin{figure*}[ht!]
  \centering
  \begin{tabular}{m{0.018\textwidth}>{\centering\arraybackslash}m{0.2\textwidth}>{\centering\arraybackslash}m{0.2\textwidth}>{\centering\arraybackslash}m{0.2\textwidth}>{\centering\arraybackslash}m{0.2\textwidth}}
& TV denoising & Bregman debiasing & TV denoising  & Bregman debiasing\\
& $u_{\alpha}(f)$ & $\hat{u}(f)$ & $u_{\alpha}(f)$ & $\hat{u}(f)$\\
\rotatebox{90}{Noisy data $f$} &
\includegraphics[width=0.2\textwidth]{Giraffe_TVdenoised_lambda0-3}&
\includegraphics[width=0.2\textwidth]{Giraffe_Bregmandebiased_lambda0-3}&
\includegraphics[width=0.2\textwidth]{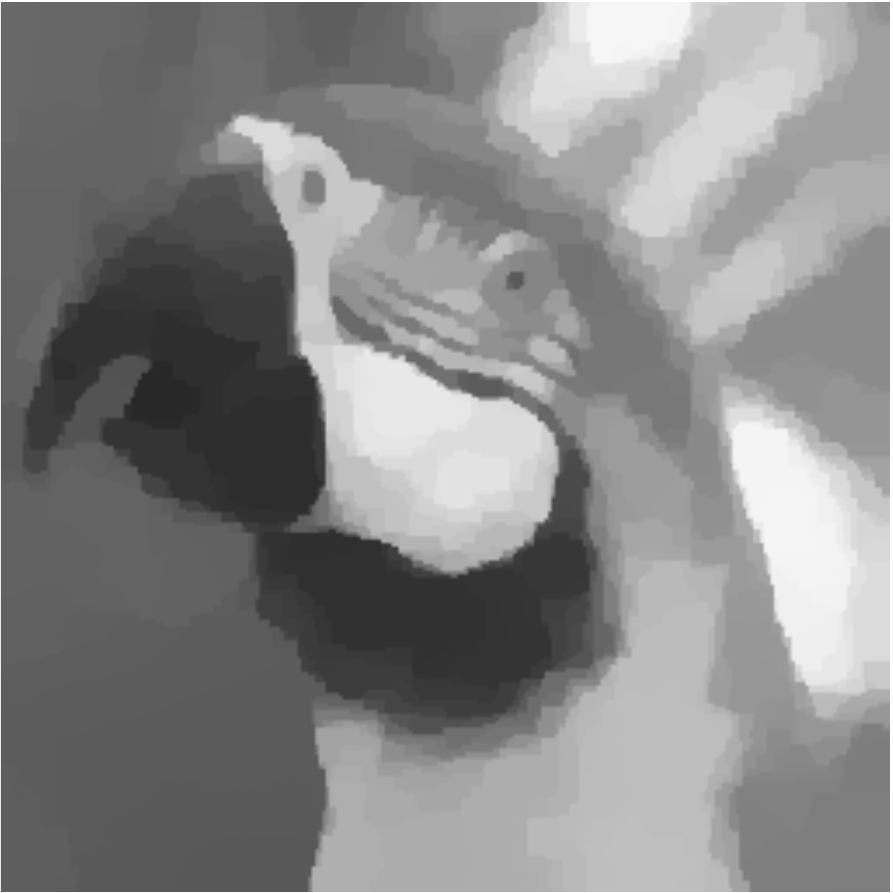}&
\includegraphics[width=0.2\textwidth]{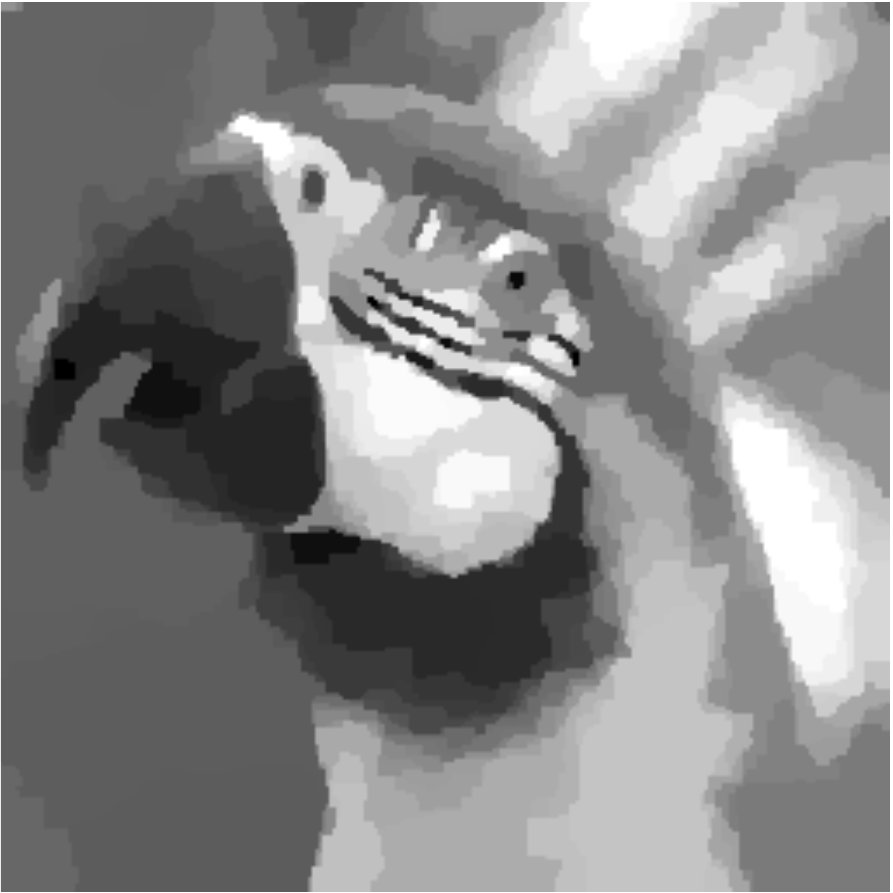}\\
& TV denoising  & Bregman debiasing & TV denoising & Bregman debiasing\\
&  $u_{\alpha}(f^*)$ & $\hat{u}(f^*)$&  $u_{\alpha}(f^*)$ &  $\hat{u}(f^*)$\\
\rotatebox{90}{Clean data $f^*$} &
\includegraphics[width=0.2\textwidth]{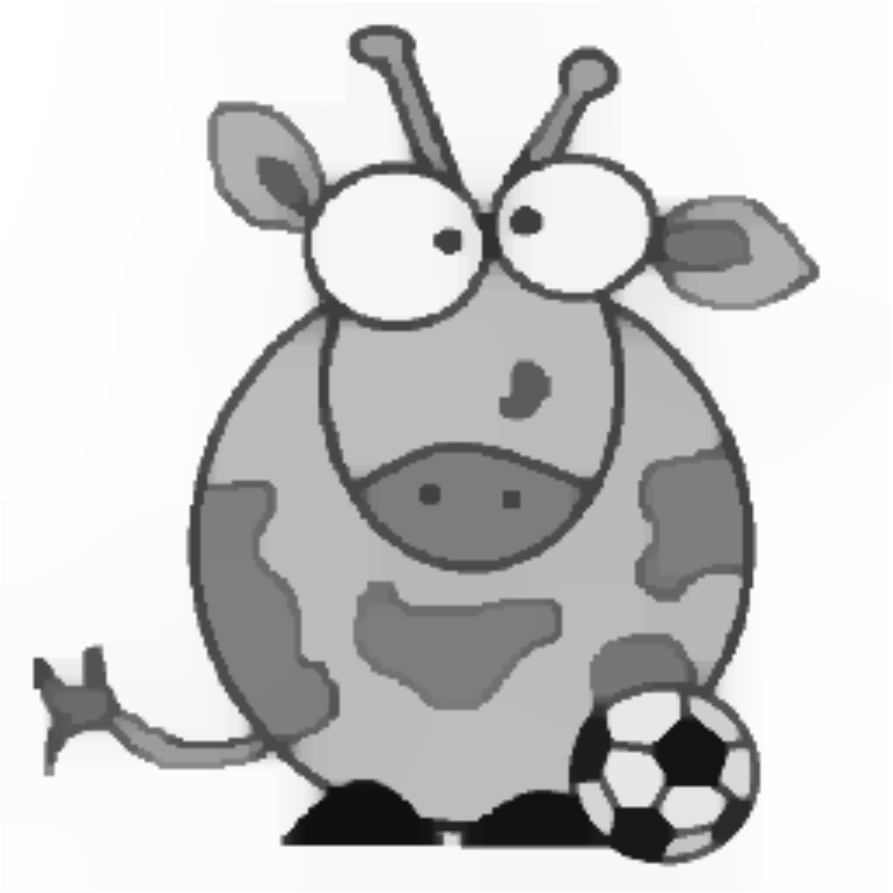}&
\includegraphics[width=0.2\textwidth]{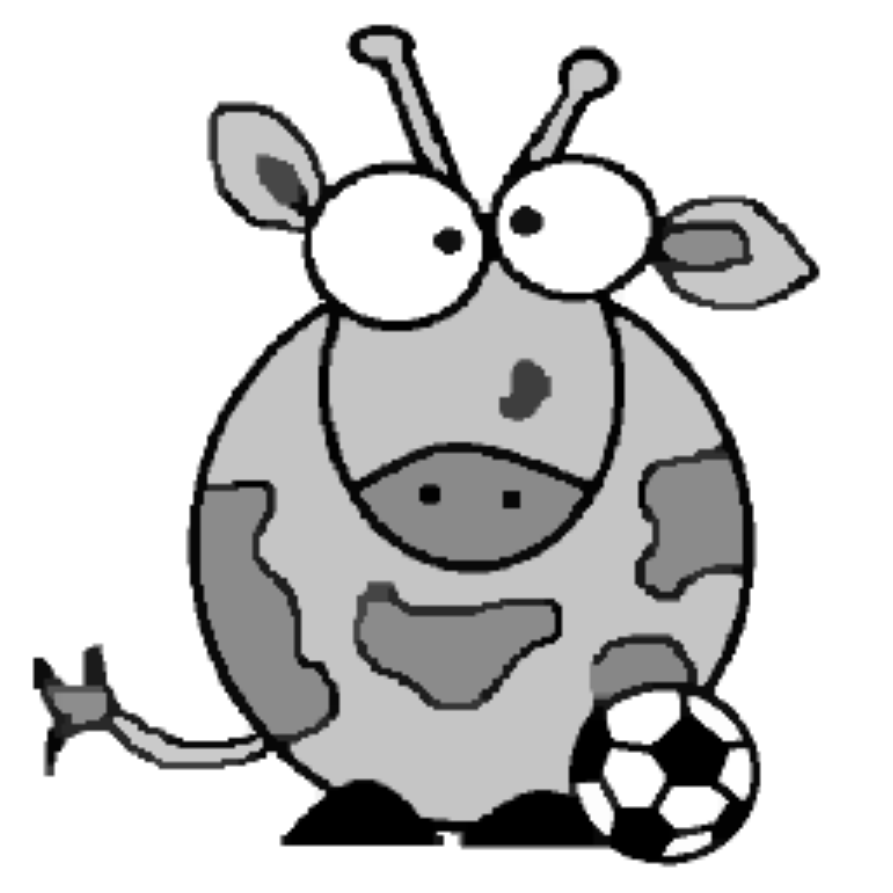}&
\includegraphics[width=0.2\textwidth]{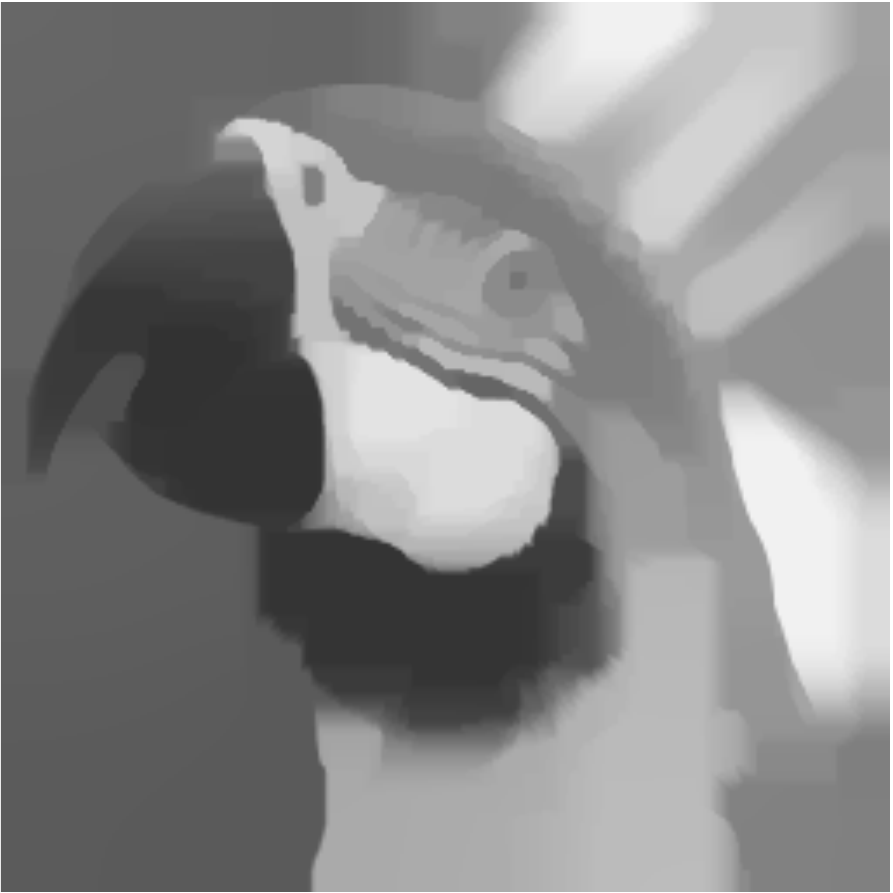}&
\includegraphics[width=0.2\textwidth]{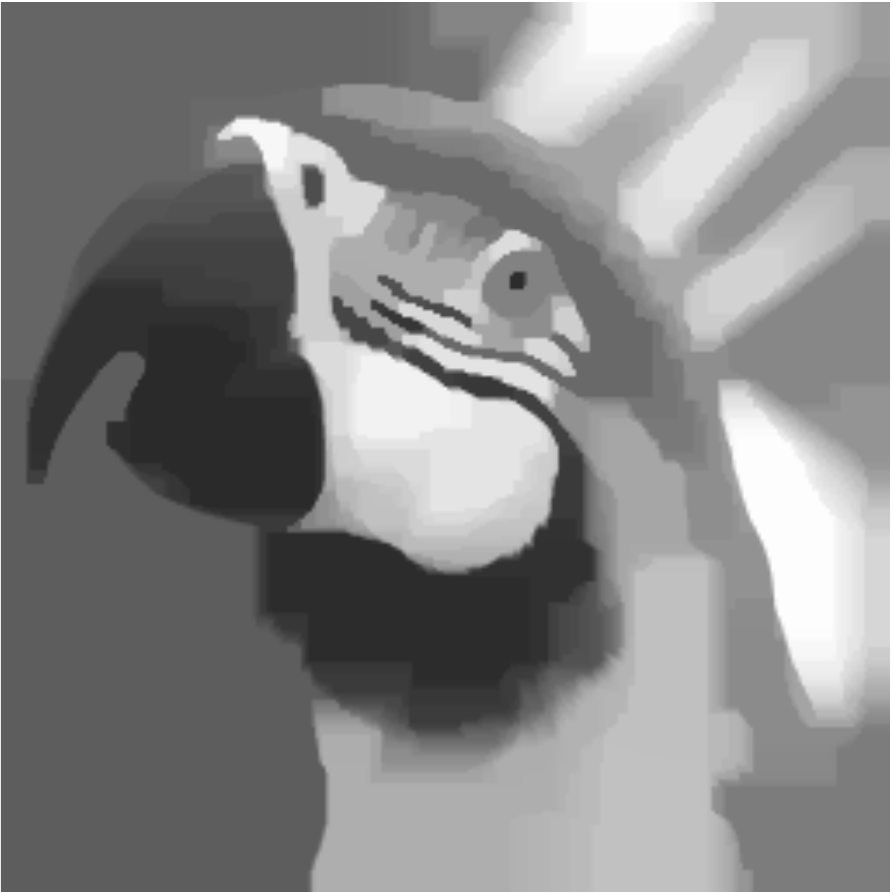}\\
  \end{tabular}
  \caption{TV denoising and debiasing of the Giraffe and the Parrot images for either noisy data $f$ or clean data $f^*$,
  with the same regularization parameter $\alpha=0.3$.\label{fig:Mf_comp}}
\end{figure*}

 In the case of the variational model \eqref{eq:firstStep}
 this can be generalized using recent approaches \cite{bregmanbuchkapitel,BuOs04,bur07,resmerita}  
 using the source condition  $A^* w^* \in \partial J(u^*)$. 
 Now completely analogous computations as above yield 
 \begin{align*}
 &\E[  \Vert A\ua(f) - Au^* \Vert_{\Y}^2 ] + \alpha^2 \E [\Vert w_\alpha - w^* \Vert_{\Y}^2 ] \\
 &\hspace{7.35em} + 2 \alpha \E[ D^{\mathrm{sym}}_J(\ua(f),u^*) ] \\
 = \ &\E[\Vert f- Au^* \Vert_{\Y}^2 ] + \alpha^2 \Vert w^* \Vert_{\Y}^2 ,
  \end{align*} 
 with the only difference that we now use the symmetric Bregman distance
  \begin{equation*}
 D^{\mathrm{sym}}_J(\ua(f),u^*) = \langle A^*w_\alpha - A^*w^* , \ua(f) - u^* \rangle,
 \end{equation*}
 with $A^*w_{\alpha} \in \partial J(\ua(f))$.
 The bias-variance decomposition on the right-hand side remains the same. 
 In the noiseless case it is then natural to consider this (here, deterministic) estimate as a measure of bias: 
 \begin{align*}
 &\Vert A\ua(f^*) - Au^* \Vert_{\Y}^2  + \alpha^2 \Vert w_\alpha - w^* \Vert_{\Y}^2  \\
 &\hspace{8.4em} + 2 \alpha D^{\mathrm{sym}}_J(\ua(f^*),u^*)  \\
 &= \ \alpha^2 \Vert w^* \Vert_{\Y}^2 ,
  \end{align*} 
 Here, as already discussed in Section \ref{biasvariancesection}, we again consider a difference between the exact solution $u^*$ 
 and the estimator for $\E[f] = f^*$, i.e. the expectation of the noise, 
 rather than the expectations of the estimators $\ua(f)$ over all realizations of $f$ (which coincide if $J$ is quadratic). 
 We observe that there are three natural distances to quantify the error and thus also the bias: 
 a quadratic one in the output space and a predual space (related to $w$),
 and the symmetric Bregman distance related to the functional $J$. The first term $\Vert A\ua(f^*) - Au^* \Vert_{\Y}^2$ is exactly the one we use as a measure of bias.
The second term $ \alpha^2 \E [\Vert w_\alpha - w^* \Vert_{\Y}^2 ]$ is constant on the model manifold $\mathcal{M}_{f^*}^{\mathrm{B}}$, 
since by definition of the manifold $p_\alpha = A^* w_\alpha$ is a subgradient of all the elements in $\mathcal{M}_{f^*}^{\mathrm{B}}$.
The third term $ D^{\mathrm{sym}}_J(\ua(f^*),u^*)$ is not easy to control; if the manifold is appropriate, meaning that $\pa \in \partial J(u^*)$, 
then the symmetric Bregman distance vanishes for every element in $\mathcal{M}_{f^*}^{\mathrm{B}}$.
In any other case, we do not have access to a subgradient $p^* \in \partial J(u^*)$, so we cannot control the Bregman distance for any element of the manifold.
Hence, with our method we minimize the part of the bias that we can actually control. 
In fact, if the model manifold is right, we even minimize the whole bias.

\subsection{Back to the proposed method}

To sum up, the debiasing method we have introduced in Equations \eqref{eq:secondStepInvConv} and \eqref{eq:secondStepBregman} comes down to debiasing over ${\cal M}_{f^*}^{\mathrm{IC}}$ and ${\cal M}_{f^*}^{\mathrm{B}}$, respectively,
while the results of Section \ref{sec:Debiasing} guarantee the existence of the optimal debiasing $\uhat(f^*)$ at least on ${\cal M}_{f^*}^{\mathrm{B}}$.

However in practice, we do not have access to the clean data $f^*$, but often only to one noisy realization $f$, 
which makes the regularization in \eqref{eq:firstStep} necessary in the first place. 
Instead of the true model manifold $\mathcal{M}_{f^*}$, we hence use an approximation $\mathcal{M}_{f}$ computed from the noisy data $f$ 
to perform the debiasing of the reconstruction $\ua(f)$ for noisy data.
The following experiments show that $\mathcal{M}_{f}$ is a good approximation of $\mathcal{M}_{f^*}$ in terms of the resulting bias and bias reduction. 
They also relate the different definitions of bias that we have considered.
In particular, we distinguish between the statistical bias of Equation \eqref{eq:stat_bias} which is the expectation over several noisy realizations $f$
and the deterministic bias that we define in Equation \eqref{eq:det_bias}, which instead considers the outcome given the noiseless data $f^*$.

Figure \ref{fig:Mf_comp} displays the TV denoising and debiasing (using the Bregman distance model manifold) results obtained with noisy data $f$ (first row)
or clean data $f^*$ (second row) with the same regularization parameter $\alpha=0.3$.
We have performed the experiments for both the cartoon {\it Giraffe} image and the natural {\it Parrot} image\footnote{\url{http://r0k.us/graphics/kodak/}}. 
First, for the Giraffe image we observe that the TV denoised solution $u_{\alpha}(f^*)$ for clean data suffers from a heavy loss of contrast, i.e. from method bias.
The debiased solution $\uhat(f^*)$ however is again close to the original data $f^*$. 
This shows that if the noiseless data is well represented by the choice of regularization (and hence $\mathcal{M}_{f^*}$),
i.e. if there is no or little model bias, the debiasing procedure allows to recover the original signal almost perfectly.
On the other hand, the same experiments on the natural {\it Parrot} image show the problem of model bias since the choice of regularization does not entirely match the data $f^*$. 
The debiasing allows to recover the lost contrast, but even the result for noiseless data still suffers from bias, i.e. the loss of small structures, which is model bias in that case.

Besides, if $\alpha$ is big enough to effectively remove noise during the denoising step, then the TV solutions $u_{\alpha}(f)$ and $u_{\alpha}(f^*)$ are close to each other.
This leads to comparable model manifolds and hence debiased solutions, which confirms that $\mathcal{M}_{f}$ is indeed a good approximation to $\mathcal{M}_{f^*}$.

\begin{figure*}[ht!]
  \centering
  \begin{tabular}{m{0.018\textwidth}>{\centering\arraybackslash}m{0.28\textwidth}>{\centering\arraybackslash}m{0.28\textwidth}>{\centering\arraybackslash}m{0.28\textwidth}}
& TV denoising, & Bregman debiasing, & \\
& $\mathbb{B}^{\text{stat}}(u_{\alpha}(f)) = u^*-\mathbb{E}[u_{\alpha}(f)]$ &  $\mathbb{B}^{\text{stat}}(\uhat(f)) = u^*-\mathbb{E}[\uhat(f)]$ &  \\
\rotatebox{90}{With noisy data $f$} &
\includegraphics[width=0.28\textwidth]{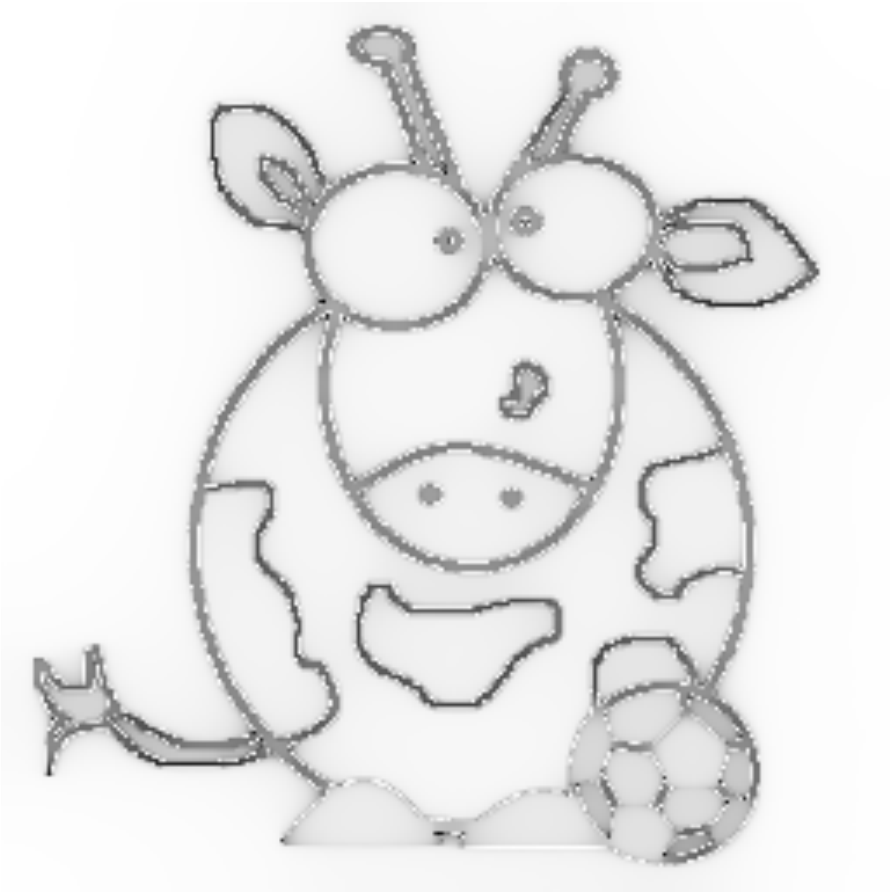}&
\includegraphics[width=0.28\textwidth]{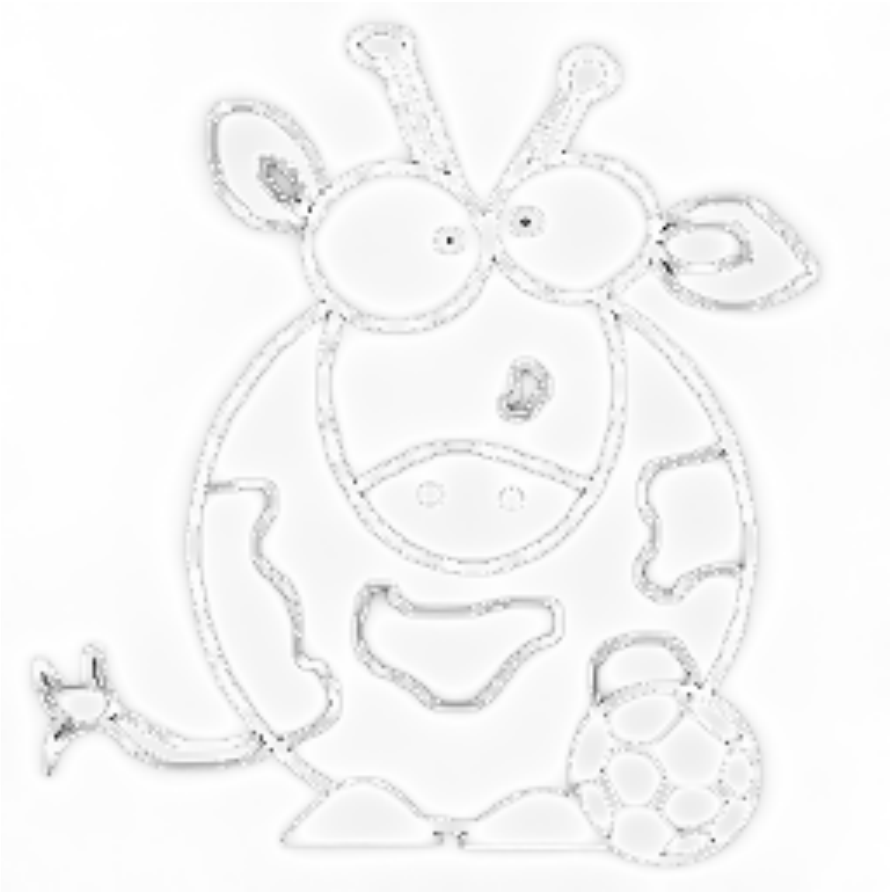}& \vspace{1em}\\
& TV denoising, & Bregman debiasing, &  \\
&  & Model bias:  & Method bias:\\
& $\mathbb{B}^*(u_{\alpha}(f^*)) = u^*-u_{\alpha}(f^*)$ & $\mathbb{B}^*(\uhat(f^*)) = u^*-\uhat(f^*)$ & $\uhat(f^*) - u_{\alpha}(f^*)$\\
\rotatebox{90}{With clean data $f^*$} &
\includegraphics[width=0.28\textwidth]{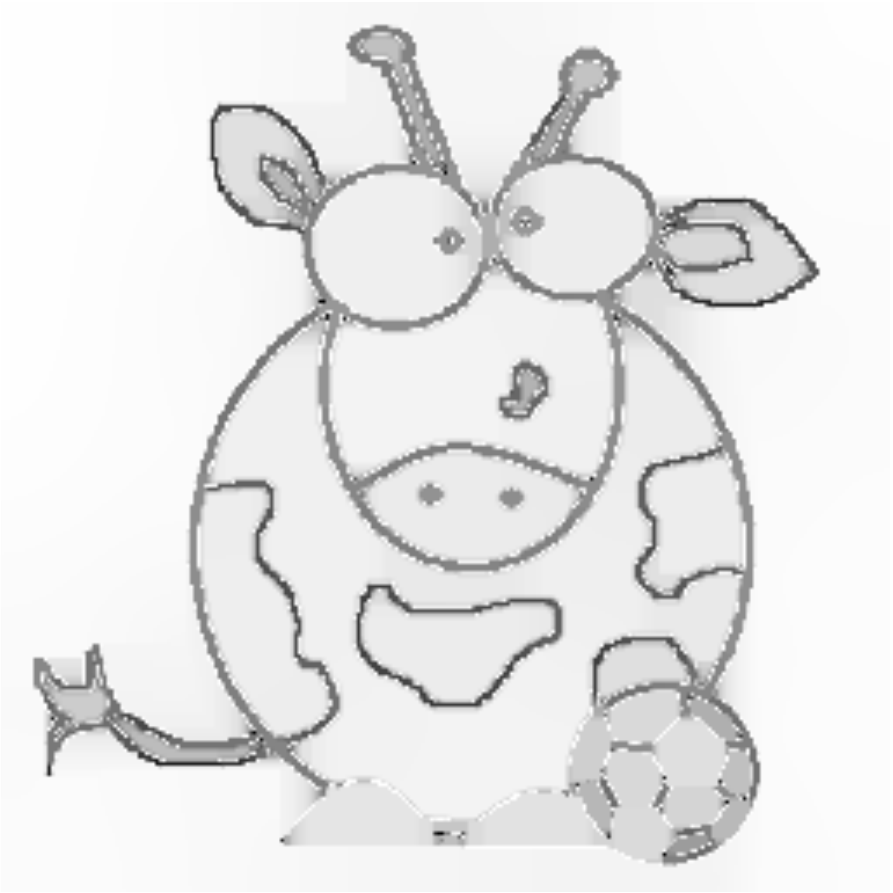}&
\includegraphics[width=0.28\textwidth]{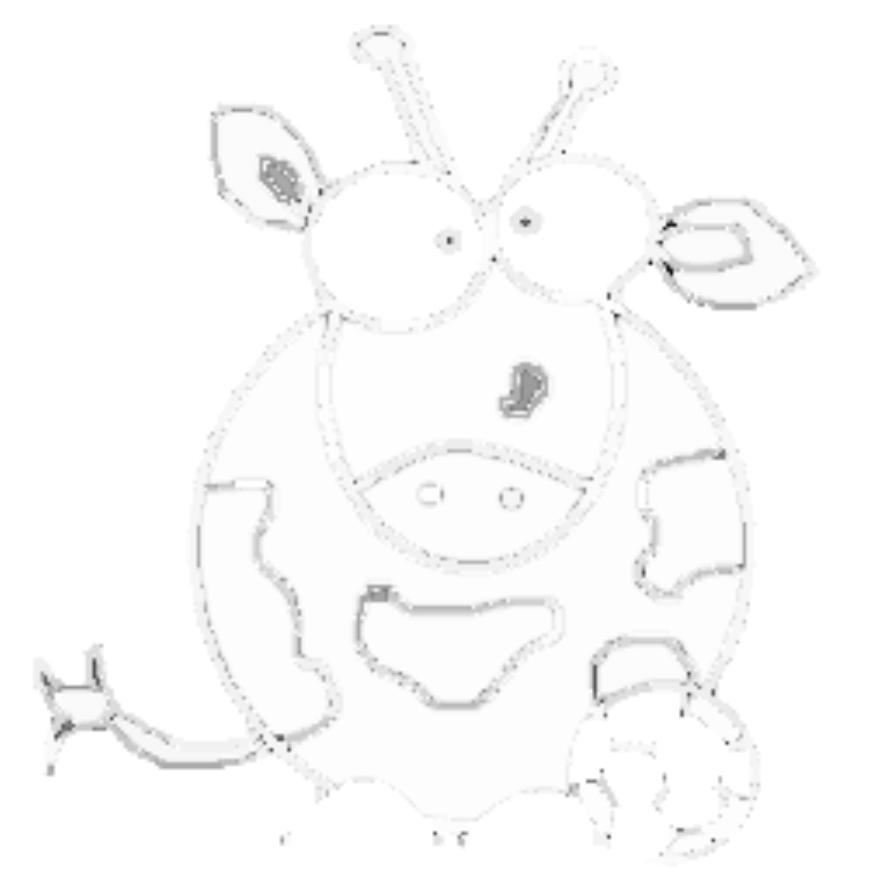}&
\includegraphics[width=0.28\textwidth]{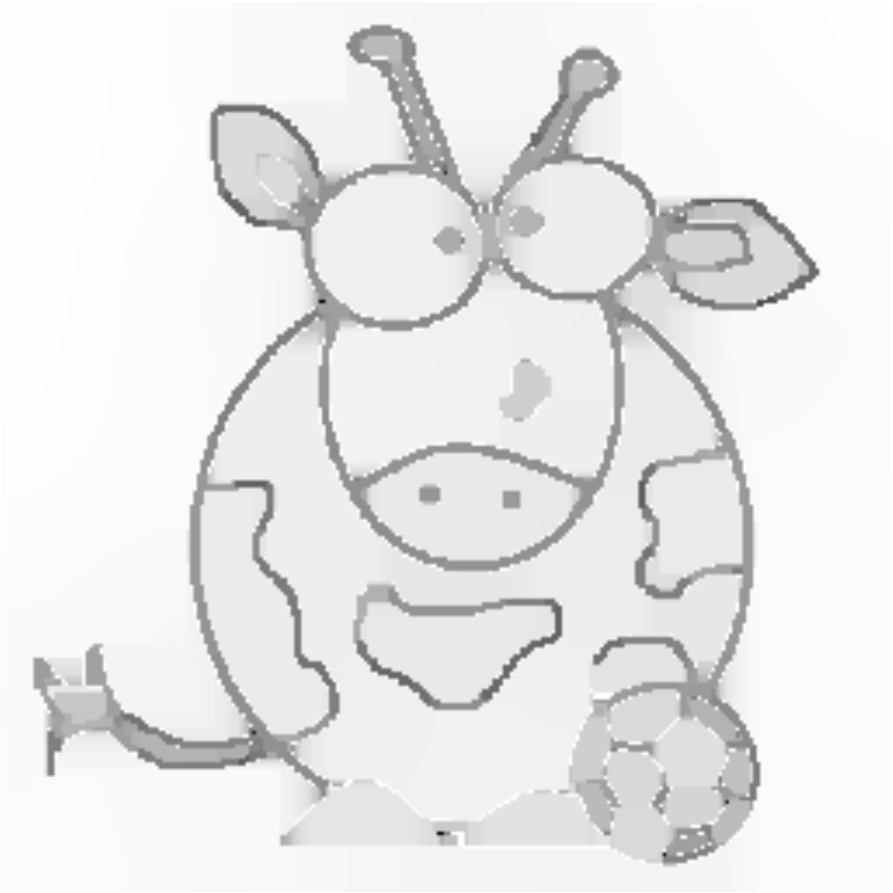}
  \end{tabular}
  \caption{Bias estimation. First row: Statistical bias computed on five hundred noisy realizations of the {\it Giraffe} cartoon image.
    Second row: Deterministic bias computed between the clean data and the recovered solution from clean data $f^*$. 
    In the first column, TV denoising leads to bias.
    In the second column, the debiasing that has been performed has reduced (or suppressed) the method bias. The remaining (small) model bias is due to the necessary regularization.
    In the third column, the difference between $\uhat(f^*)$ and $u_{\alpha}(f^*)$ shows the bias that has been reduced by the debiasing step, hence the method bias.
    \label{fig:Giraffe_bias}}
\end{figure*}

\begin{figure*}[ht!]
  \centering
  \begin{tabular}{m{0.018\textwidth}>{\centering\arraybackslash}m{0.28\textwidth}>{\centering\arraybackslash}m{0.28\textwidth}>{\centering\arraybackslash}m{0.28\textwidth}}
& TV denoising, & Bregman debiasing, & \\
& $\mathbb{B}^{\text{stat}}(u_{\alpha}(f)) = u^*-\mathbb{E}[u_{\alpha}(f)]$ &  $\mathbb{B}^{\text{stat}}(\uhat(f)) = u^*-\mathbb{E}[\uhat(f)]$ &  \\
\rotatebox{90}{With noisy data $f$} &
\includegraphics[width=0.28\textwidth]{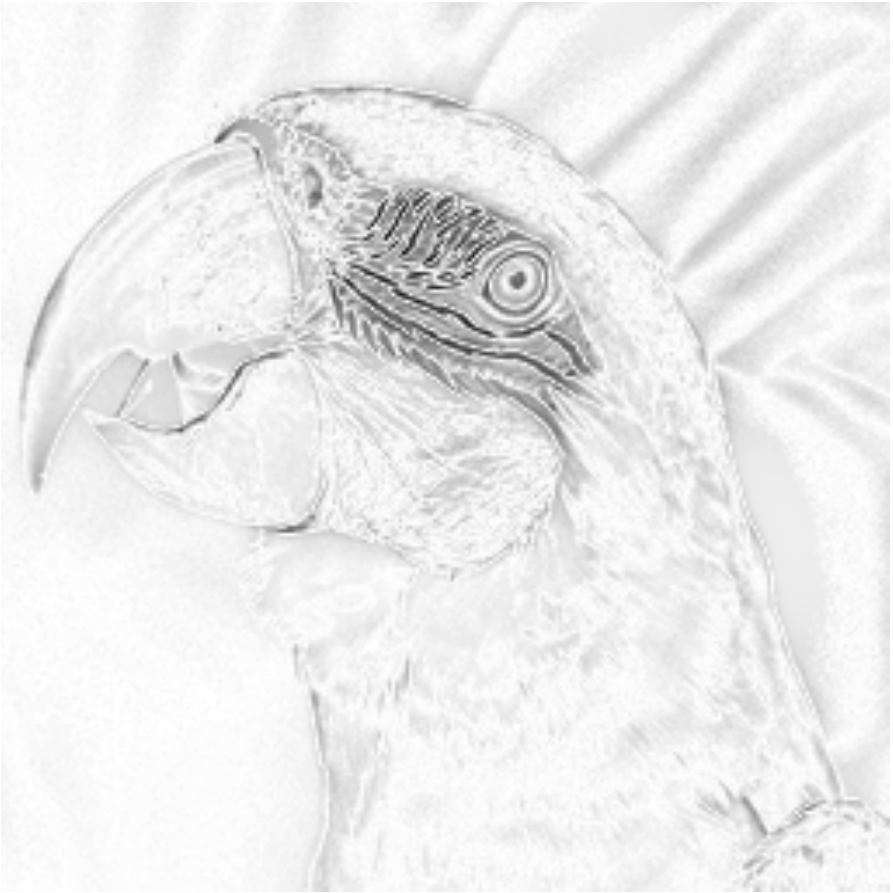}&
\includegraphics[width=0.28\textwidth]{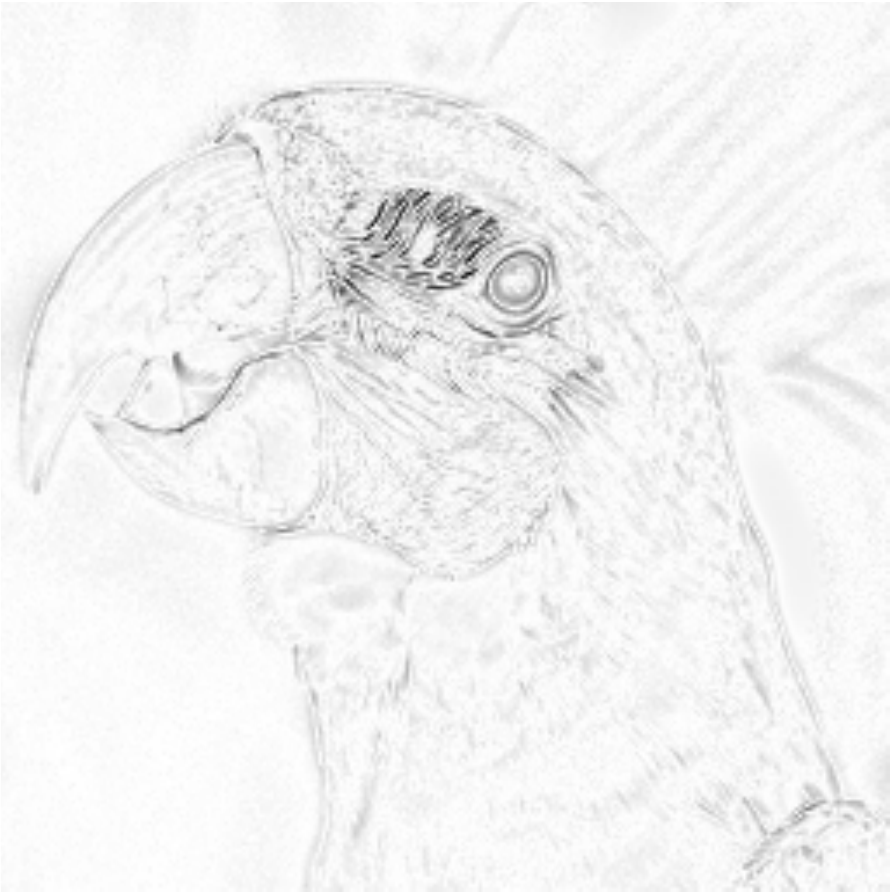}& \vspace{1em}\\
& TV denoising, & Bregman debiasing, &  \\
&  & Model bias:  & Method bias:\\
& $\mathbb{B}^*(u_{\alpha}(f^*)) = u^*-u_{\alpha}(f^*)$ & $\mathbb{B}^*(\uhat(f^*)) = u^*-\uhat(f^*)$ & $\uhat(f^*) - u_{\alpha}(f^*)$\\
\rotatebox{90}{With clean data $f^*$} &
\includegraphics[width=0.28\textwidth]{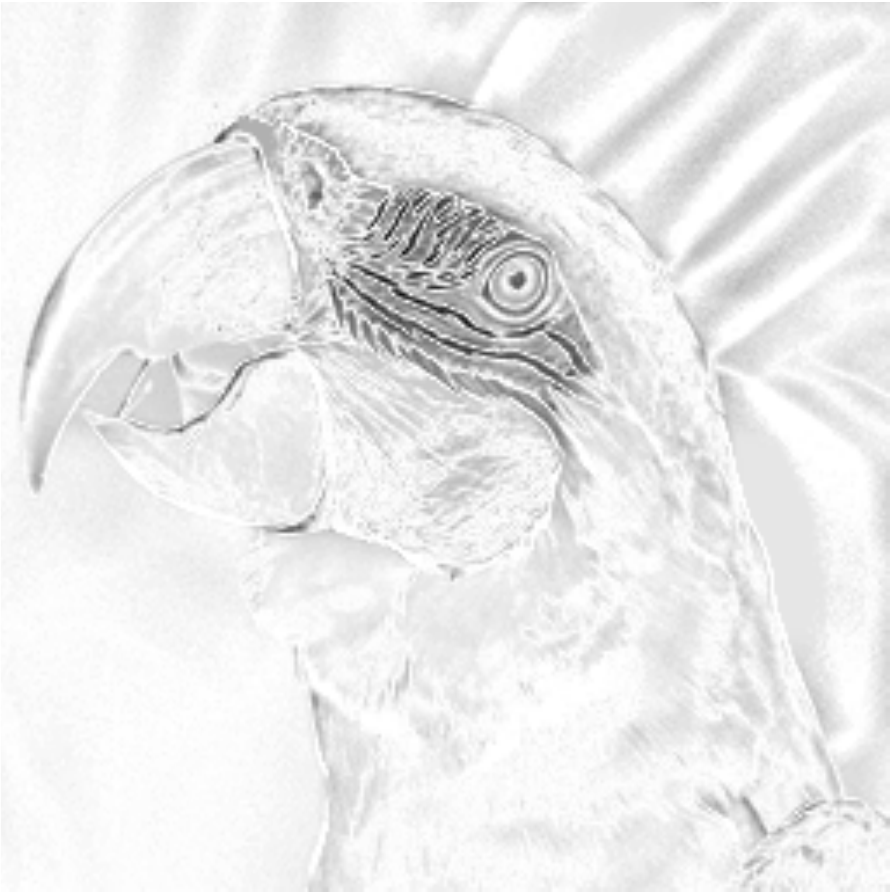}&
\includegraphics[width=0.28\textwidth]{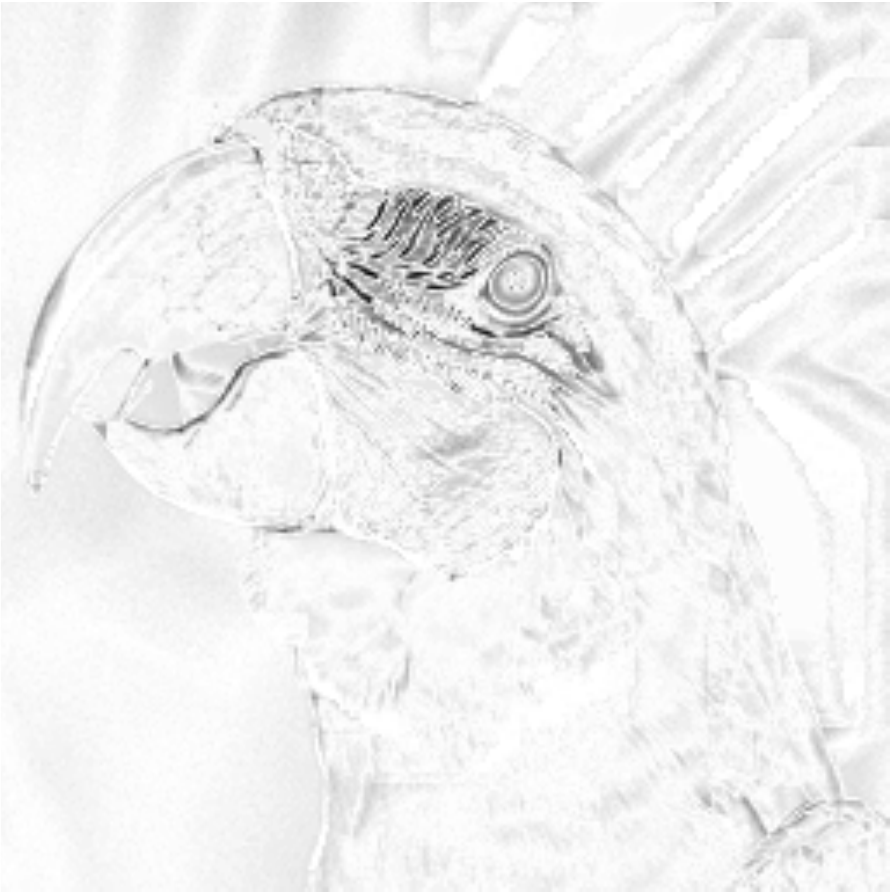}&
\includegraphics[width=0.28\textwidth]{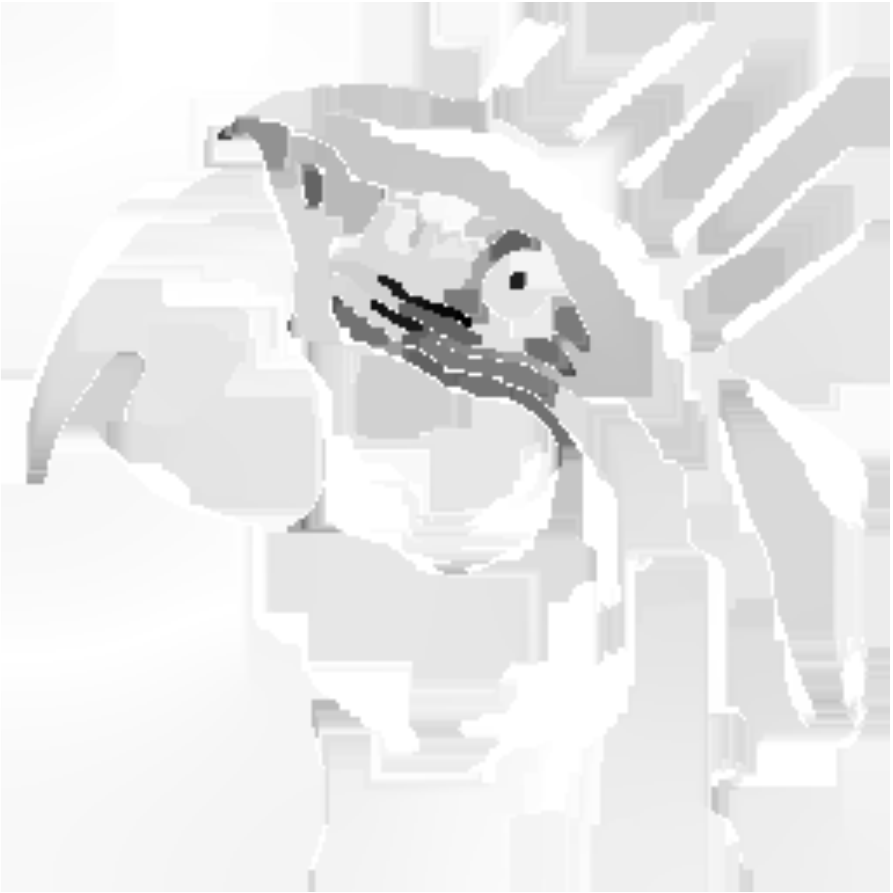}
  \end{tabular}
  \caption{Bias estimation. First row: Statistical bias computed on five hundred noisy realizations of the {\it Parrot} natural image.
    Second row: Deterministic bias computed between the clean data and the recovered solution from clean data $f^*$. 
    On the first column, TV denoising leads to both kinds of bias, model bias and method bias.
    On the second column, the debiasing that has been performed has reduced (or suppressed) the method bias, and the remaining bias is model bias.
    On the third column, the difference between $\uhat(f^*)$ and $u_{\alpha}(f^*)$ shows the bias that has been reduced by the debiasing step, hence the method bias.
    \label{fig:Parrot_bias}}
\end{figure*}

Furthermore, we can assess the bias for both manifolds.
On $\mathcal{M}_{f^*}$ we can only use the deterministic definition \eqref{eq:det_bias} of bias whereas
on $\mathcal{M}_{f}$ we use the statistical definition \eqref{eq:stat_bias}.
Figures \ref{fig:Giraffe_bias} and \ref{fig:Parrot_bias} show the bias estimation on the {\it Giraffe} cartoon image and the natural {\it Parrot} image.
The first row shows the estimations of the statistical bias $\mathbb{B}^{\text{stat}}$ for the two estimators $u_{\alpha}(f)$ and $\uhat(f)$ for noisy data $f$.
In the second row the bias $\mathbb{B}^*$ for the two estimators $u_{\alpha}(f^*)$ and $\uhat(f^*)$ for clean data $f^*$ is displayed.
This deterministic bias can also be decomposed into the associated model and method bias, whereas such a decomposition has not been defined for the statistical bias. 
The overall deterministic bias $\mathbb{B}^*(u_{\alpha}(f^*)) = u^*-u_{\alpha}(f^*)$ for TV denoising appears to be really close to the statistical bias on noisy data in the first row. 
The same applies for the bias of the debiased solutions in the second column.
This confirms that the estimation of the model manifold that we perform with noisy data is indeed a good approximation
to the ideal model manifold for clean data, and that the resulting statistical and deterministic bias are closely related.

Besides, the difference $u^*-\uhat(f^*)$ in the second row shows the remaining bias after the debiasing step, which is model bias.
For the {\it Giraffe} image, this bias is small because the cartoon image is well approximated in the model manifold associated to TV regularization.
The {\it Parrot} image however suffers from a heavier model bias, for example the loss of the small structures around the eye.
Finally, in the third column, the difference $\uhat(f^*) - u_{\alpha}(f^*)$ shows the error that has been removed by the debiasing step, which corresponds to the method bias. 
It is particularly interesting for the {\it Parrot} image. 
Here one can see the piecewise constant areas which correspond to the re-establishment of the lost contrast within the piecewise constant model provided by the model manifold.

\newpage
\subsection{Relation to inverse scale space methods}
 
We finally comment on the relation of the debiasing approaches to Bregman iterations respectively inverse scale space methods,
which are rather efficiently reducing bias as demonstrated in many examples \cite{osh-bur-gol-xu-yin,scherzer,gilboa}. 
The Bregman iteration is iteratively constructed by
\begin{align*}
	u^{k+1} &\in \text{arg}\min_{u \in \X}  \frac{1}2 \Vert A u - f \Vert_{\Y}^2 + \alpha D_J^{p^k}(u,u^k), \\
	p^{k+1} &= p^k + \frac{1}\alpha A^* (f- Au^{k+1}) \in \partial J(u^{k+1}).
\end{align*}
In the limit $\alpha \rightarrow \infty$ we obtain the time continuous inverse scale-space method, 
which is the differential inclusion
\begin{equation*}
	\partial_t p(t) = A^* (f -Au(t)), \hfill p(t) \in \partial J(u(t)),
\end{equation*}
with initial values $u(0)=0$, $p(0)=0$. 
A strong relation to our debiasing approach comes from the characterization of the primal solution given $p(t)$ \cite{adaptive1,adaptive2,osheryin}
\begin{equation*}
	u(t) \in \text{arg}\min_{u \in \X} \Vert Au -f \Vert_{\Y}^2 \quad \text{s.t. } p(t) \in \partial J(u(t)).
\end{equation*}
This reconstruction step is exactly the same as the variational debiasing step using the Bregman distance, 
however with a different preceding construction of the subgradient $p(t)$ 
(noticing that $t$ corresponds to $\frac{1}\alpha$ for the variational method). 

From the last observation it becomes apparent that the Bregman debiasing approach with \eqref{eq:subgradient} and 
\eqref{eq:secondStepBregman} is exactly equivalent if the variational method yields the same subgradient 
as the inverse scale space method, i.e. $\pa=p(\frac{1}{\alpha})$. 
This can indeed happen, as the results for singular vectors demonstrate \cite{benninggroundstates}. 
Moreover, in some cases there is full equivalence for arbitrary data, 
e.g. in a finite-dimensional denoising setting investigated in \cite{spectraltv2}. 
It has been shown that for $A$ being the identity 
and $J(u) = \Vert \Gamma u  \Vert_{1}$ with $\Gamma \Gamma^*$ being diagonally dominant the identity 
$\pa=p(\frac{1}{\alpha})$ holds, 
which implies that the Bregman debiasing approach and the inverse scale space method yield exactly the same solution. 
For other cases that do not yield a strict equivalence 
we include the Bregman iteration for comparison in numerical studies discussed below.

\section{Numerical Implementation}
\label{sec:Implementation}

\begin{algorithm*}
\caption{\textbf{Primal-Dual Algorithm for Variational Regularization (Step 1)}}
{
\begin{algorithmic}
\Require $f$, $\alpha > 0$ 
\Ensure $\sigma,\tau > 0$, $u^0 = \bar{u}^0 = 0, \; y_1^0 = y_2^0 = 0$
	\While{not converged}
		\State $y_1^{k+1} = \frac{y_1^k + \sigma Au^k - \sigma f}{1 + \sigma}$
		\State $y_2^{k+1} = \Pi_{B^\infty_{\alpha}}(y_2^k + \sigma \Gamma u^k)$
		\State $u^{k+1} = u^k - \tau (A^* y_1^{k+1} + \Gamma^* y_2^{k+1})$
		\State $\bar{u}^{k+1} = 2u^{k+1}  - u^{k}$
	\EndWhile\\
\Return $\ua = u^{k+1}$, $\pa = \frac{1}{\alpha}A^*(f - A\ua)$ (c.f. \eqref{eq:subgradient})
\end{algorithmic}
}
\label{alg:step1}
\end{algorithm*}

\begin{algorithm*}
\caption{\textbf{Primal-Dual Algorithm for Bias-Reduction with $\MB$ (Step 2 a))}}
{
\begin{algorithmic}
\Require $f$, $\gamma > 0$, $\pa$, which is obtained via Algorithm \ref{alg:step1}
\Ensure $\sigma,\tau > 0$, $u^0 = \bar{u}^0 = 0$, $y_1^0 = y_2^0 = 0$
	\While{not converged}
		\State $y_1^{k+1} = \frac{y_1^k + \sigma Au^k - \sigma f}{1 + \sigma}$
		\State $y_2^{k+1} = \Pi_{B^\infty_{\gamma}}(y_2^k + \sigma \Gamma u^k)$
		\State $u^{k+1} = u^k - \tau (A^* y_1^{k+1} + \Gamma^* y_2^{k+1} - \gamma \pa)$
		\State $\bar{u}^{k+1} = 2u^{k+1} - u^{k}$
	\EndWhile\\
\Return $\uhat = u^{k+1}$
\end{algorithmic}
}
\label{alg:case_a}
\end{algorithm*}

\begin{algorithm*}
\caption{\textbf{Primal-Dual Algorithm for Bias-Reduction with $\MIC$ (Step 2 b))}}
{
\begin{algorithmic}
\Require $f$, $\gamma > 0$ and $\pa$, which is obtained via Algorithm \ref{alg:step1}.
\Ensure $\sigma,\tau > 0$, $u^0 = z^0 = \bar{u}^0 = \bar{z}^0 = 0 , \; y_1^0 = y_2^0 = y_3^0 = 0$
	\While{not converged}
		\State $y_1^{k+1} = \frac{y_1^k + \sigma Au^k - \sigma f}{1 + \sigma}$
		\State $y_2^{k+1} = \Pi_{B^\infty_{\gamma}}(y_2^k + \sigma \Gamma(u^k - z^k))$
		\State $y_3^{k+1} = \Pi_{B^\infty_{\gamma}}(y_3^k + \sigma \Gamma z^k)$
		\State $u^{k+1} = u^k - \tau (A^* y_1^{k+1} + \Gamma^* y_2^{k+1} - \gamma \pa)$
		\State $z^{k+1} = z^k - \tau (-\Gamma^* y_2^{k+1} + \Gamma^* y_3^{k+1} + 2 \gamma \pa)$
		\State $\bar{u}^{k+1} = 2u^{k+1} - u^{k}$
		\State $\bar{z}^{k+1} = 2z^{k+1} - z^{k}$
	\EndWhile\\
\Return $\uhat= u^{k+1}$
\end{algorithmic}
}
\label{alg:case_b}
\end{algorithm*}

In Section \ref{sec:Debiasing} we have introduced a two-step-method (cf. Eq. \eqref{eq:firstStep} -- \eqref{eq:secondStepBregman}) 
in order to compute a variationally regularized reconstruction with reduced method bias in the sense discussed in Section \ref{sec:BiasModelManifolds}.
Its solution requires the minimization of the data fidelity over the model manifold defined by a zero Bregman distance or a zero infimal convolution thereof, respectively. 

This constraint is difficult to realize numerically, but can be approximated by a rather standard variational problem. 
We can translate the hard constraint into a soft constraint such that for $\gamma > 0$ the reformulated problems read:
\begin{alignat*}{2}
&\text{a) } &&\uhat\in \arg \min_{u \in \X} \frac{1}{2} \Vert A u - f \Vert_\Y^2 + \gamma {D}^{\pa}_J(u,\ua),\\
&\text{b) } &&\uhat\in \arg \min_{u \in \X} \frac{1}{2} \Vert A u - f \Vert_\Y^2 + \gamma \ICB(u,u_\alpha).
\end{alignat*}
For $\gamma \to \infty$ we obtain the equivalence of the hard and soft constrained formulations. 
However, for the numerical realization already a moderately large $\gamma$ is enough to enforce the constraint up to a satisfactory level. 
For our simulations we chose $\gamma = 1000$, but our tests showed that already for $\gamma \geq 500$ the value of the Bregman distance
or its infimal convolution stays numerically zero. 
Of course the choice of the parameter $\gamma$ depends on the specific problem we aim to solve and probably has to be adjusted slightly for different image sizes or 
involved operators. 
\subsubsection*{Discretization}
For our numerical experiments we choose the setting $\X = \mathbb{R}^n, \; \Y = \mathbb{R}^d$ and $J(u) = \Vert \Gamma u \Vert_{1}$.
In general $\Gamma \in \mathbb{R}^{n \times m}$ denotes a discrete linear operator,  
for the experiments with total variation regularization we choose a discretization of the gradient with forward finite differences.
For a general linear forward operator $A \in \mathbb{R}^{n \times d}$ we hence end up with the following discrete optimization problems:
\begin{equation*}
\begin{alignedat}{5}
&\text{1.} \; && \quad && \ua &&\in \arg \min_{u \in \R^n} \frac{1}{2}\Vert Au - f \Vert_2^2 + \alpha \Vert \Gamma u \Vert_1, 
\end{alignedat}
\end{equation*}
\begin{equation*}
\begin{alignedat}{5}
&\text{2.} \; && \text{a)} \quad && \uhat&&\in \arg \min_{u \in \R^n} \frac{1}{2} \Vert Au - f \Vert_2^2\\
& \; && \quad &&  && + \gamma \left( \Vert \Gamma u \Vert_1 - \langle \pa,u \rangle \right),\\
& \; && \text{b)} \quad && \uhat&&\in \arg \min_{u \in \R^n} \frac{1}{2} \Vert Au - f \Vert_2^2\\
& \; && \quad &&  && + \gamma \min_{z \in \R^n} \Big\{ \Vert \Gamma (u - z) \Vert_1 - \langle \pa,u-z \rangle\\
& \; && \quad &&  && +\Vert \Gamma z \Vert_1 + \langle \pa,z \rangle \Big\},
\end{alignedat}
\end{equation*}
where we leave out the particular spaces for the primal (and dual) variables for the sake of simplicity in the following. 
Taking a closer look at these minimization problems, we observe that we can exactly recover the optimization problem 
in the first step by means of problem $2. \; \text{a)}$ if we choose $\gamma = \alpha$ and $\pa = 0$. 
We therefore concentrate on the minimization problems in the second step.
\subsubsection*{Primal-dual and dual formulation}
Using the notion of convex conjugates \cite{Rockafellar}, the corresponding primal-dual and dual formulations of 
our problems are given by 
\begin{align*}
\text{a)} \quad &\min_u \max_{y_1,y_2} ~ \langle y_1,Au \rangle + \langle  y_2,\Gamma u \rangle - \gamma ~ \langle \pa, u \rangle\\
& \hspace{3em} - \frac{1}{2}\Vert y_1 \Vert_2^2 - \langle y_1,f \rangle - \iota_{B^{\infty}_\gamma}(y_2) \\
= &\max_{y_1,y_2} ~ - \frac{1}{2}\Vert y_1 \Vert_2^2 - \langle y_1,f \rangle - \iota_{B^{\infty}_\gamma}(y_2) \\ 
& \hspace{3em}- \iota_{\gamma \pa}(A^* y_1 + \Gamma ^* y_2),\\
\text{b)} \quad &\min_{u,z} \max_{y_1,y_2,y_3} ~ \langle y_1,Au \rangle + \langle y_2,\Gamma u - \Gamma z \rangle \\
& \hspace{3em} + \langle  y_3,\Gamma z \rangle - \gamma \; \langle \pa, u\rangle + 2\gamma \; \langle \pa, z \rangle\\
& \hspace{3em} - \frac{1}{2}\Vert y_1 \Vert_2^2 - \langle y_1,f \rangle\\
& \hspace{3em} - \iota_{B^{\infty}_\gamma}(y_2)- \iota_{B^{\infty}_\gamma}(y_3) \\
= &\max_{y_1,y_2,y_3} ~ - \frac{1}{2}\Vert y_1 \Vert_2^2 - \langle y_1,f \rangle \\
& \hspace{3em} - \iota_{B^{\infty}_\gamma}(y_2)- \iota_{B^{\infty}_\gamma}(y_3) \\
& \hspace{3em}- \iota_{\gamma \pa}(A^* y_1 + \Gamma ^* y_2) \\
& \hspace{3em}- \iota_{-2 \gamma \pa}(-\Gamma^*y_2 + \Gamma^* y_3),
\end{align*}
\subsubsection*{Solution with a primal-dual algorithm}
In order to find a saddle point of the primal-dual formulations, we apply a version of the 
popular first-order primal-dual algorithms \cite{pock-cre-bisch-cha,Esser_et_al,Chambolle_and_Pock}. 
The basic idea is to perform gradient descent on the primal and gradient ascent on the dual variables. 
Whenever the involved functionals are not differentiable, here the $\ell^1$-norm, this comes down to computing the corresponding proximal mappings. 
The specific updates needed for our method are summarized in Algorithm \ref{alg:step1} for the first regularization problem, 
and Algorithm \ref{alg:case_a} and Algorithm \ref{alg:case_b} for the two different debiasing steps. 

We comment on our choice of the stopping criterion. 
We consider the primal-dual gap of our saddle point problem, which is defined as the difference between the primal and the dual problem for the current values of variables.
As in the course of iterations the algorithm is approaching the saddle point, this gap converges to zero. 
Hence we consider our algorithm converged if this gap is below a certain threshold $\epsilon_1 > 0$. 
We point out that the indicator functions regarding the $\ell^\infty$-balls are always zero due to the projection of the dual variables in every update. 
Since the constraints with respect to the other indicator functions, for example
\begin{align*}
 A^* y_1 + \Gamma^* y_2 - \gamma \pa = 0 
\end{align*}
in case a), are hard to satisfy exactly numerically, we instead control that 
the norm of the left-hand side is smaller than a certain threshold $\epsilon_2$ (respectively $\epsilon_3$ for case b)). 
All in all we stop the algorithm if the current iterates satisfy: 
\begin{align*}
\text{a)} \quad  &  PD(u,y_1,y_2) = \big(-\gamma \langle \pa,u \rangle\\
& \qquad + \frac{1}{2} \Vert Au - f \Vert_2^2 + \gamma \Vert \Gamma u \Vert_1\\
& \qquad + \frac{1}{2}\Vert y_1 \Vert_2^2 + \langle y_1,f \rangle \big) / n < \epsilon_1
\end{align*}
and
\begin{align*}
&\Vert A^* y_1 + \Gamma ^* y_2 - \gamma \pa \Vert_1 /n < \epsilon_2
\end{align*}
\begin{align*}
\text{b)} \quad  & PD(u,z,y_1,y_2) = \big(-\gamma \langle \pa,u \rangle + 2\gamma\langle \pa,z \rangle \\
& \qquad + \frac{1}{2} \Vert Au - f \Vert_2^2\\
& \qquad + \gamma \Vert \Gamma u - \Gamma z\Vert_1 + \gamma \Vert \Gamma z \Vert_1\\
& \qquad + \frac{1}{2}\Vert y_1 \Vert_2^2 + \langle y_1,f \rangle \big)/n < \epsilon_1
\end{align*}
and 
\begin{align*}
&\Vert A^* y_1 + \Gamma ^* y_2 - \gamma \pa \Vert_1/n < \epsilon_2, \\
&\Vert -\Gamma^* y_2 + \Gamma ^* y_3 + 2\gamma \pa \Vert_1 /n < \epsilon_3.
\end{align*}
Note that we normalize the primal-dual gap and the constraints by the number of primal pixels $n$ in order to keep the 
thresholds $\epsilon_1, \epsilon_2$ and $\epsilon_3$ independent of varying image resolutions. 
We give an example for the specific choice of parameters for our total variation denoising problems in Table \ref{tab:params}.

\begin{table}[t!]
\centering
 \begin{tabular}{|lr|}
\hline
\textbf{Parameters}& \\
 \hline
 $\alpha$ & $0.3$ \\
 $\gamma$ & $1000$\\
 $\sigma = \tau$ & $\frac{1}{\sqrt{8}}$ \\
 $\epsilon_1$ & $10^{-5}$ \\
 $\epsilon_2$ & $10^{-6}$ \\
 $\epsilon_3$ & $10^{-6}$ \\
 \hline
\end{tabular}
\label{tab:params}
\caption{Choice of parameters for a total variation denoising problem of an image 
of size 256x256 with values in $[0,1]$, corrupted by Gaussian noise with variance $0.05$.}
\end{table}

\begin{figure*}[t!]
\centering
\includegraphics[width=0.7\textwidth]{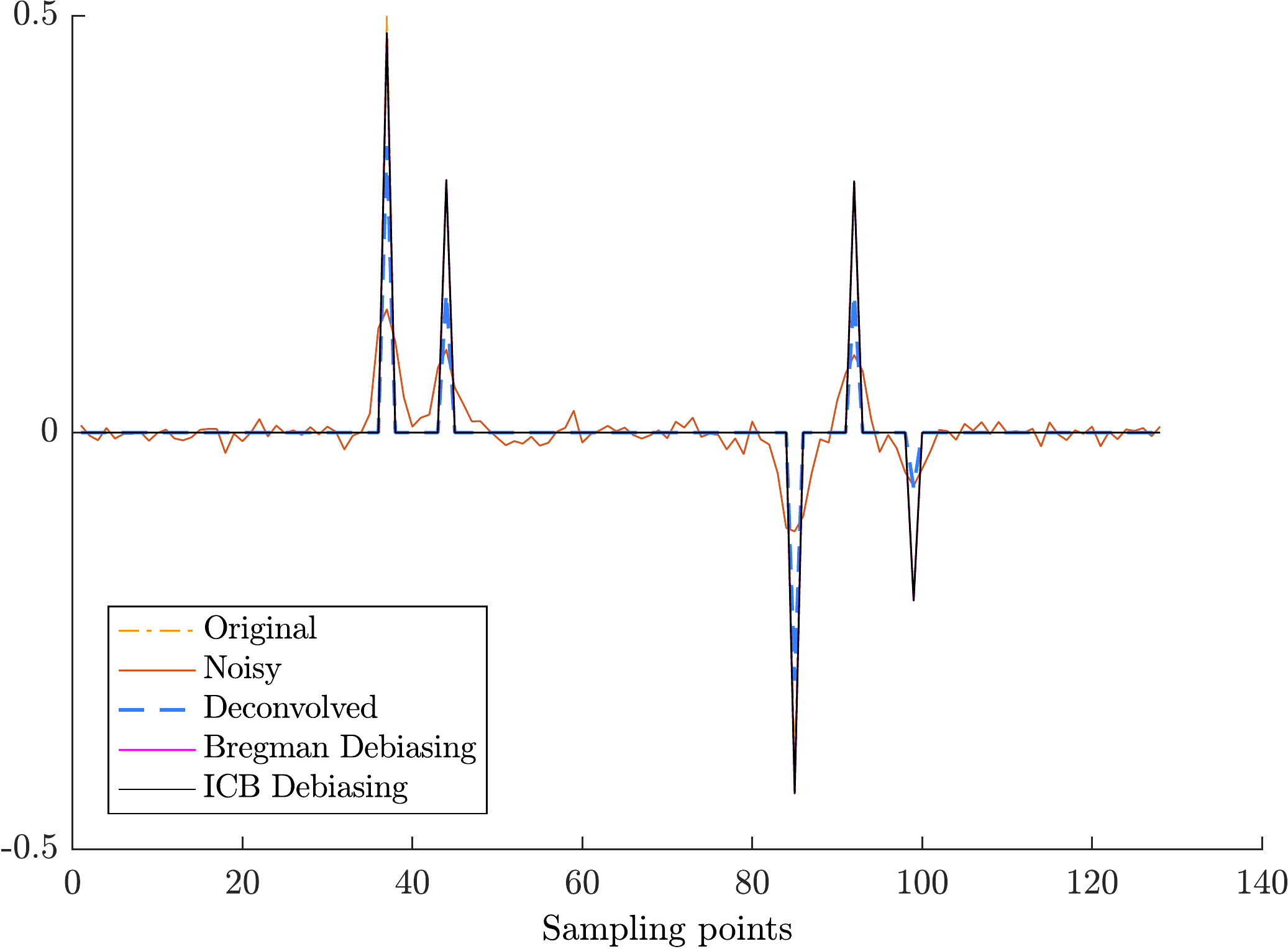}
\caption{$\ell^1$-deconvolution of a 1D signal. 
Original and noisy convolved signals, and $\ell^1$-reconstruction, Bregman debiasing and Infimal convolution debiasing.
\label{fig:L1deconv1D_a}}
\end{figure*}

\begin{figure}[t!]
\centering
\includegraphics[width=0.45\textwidth]{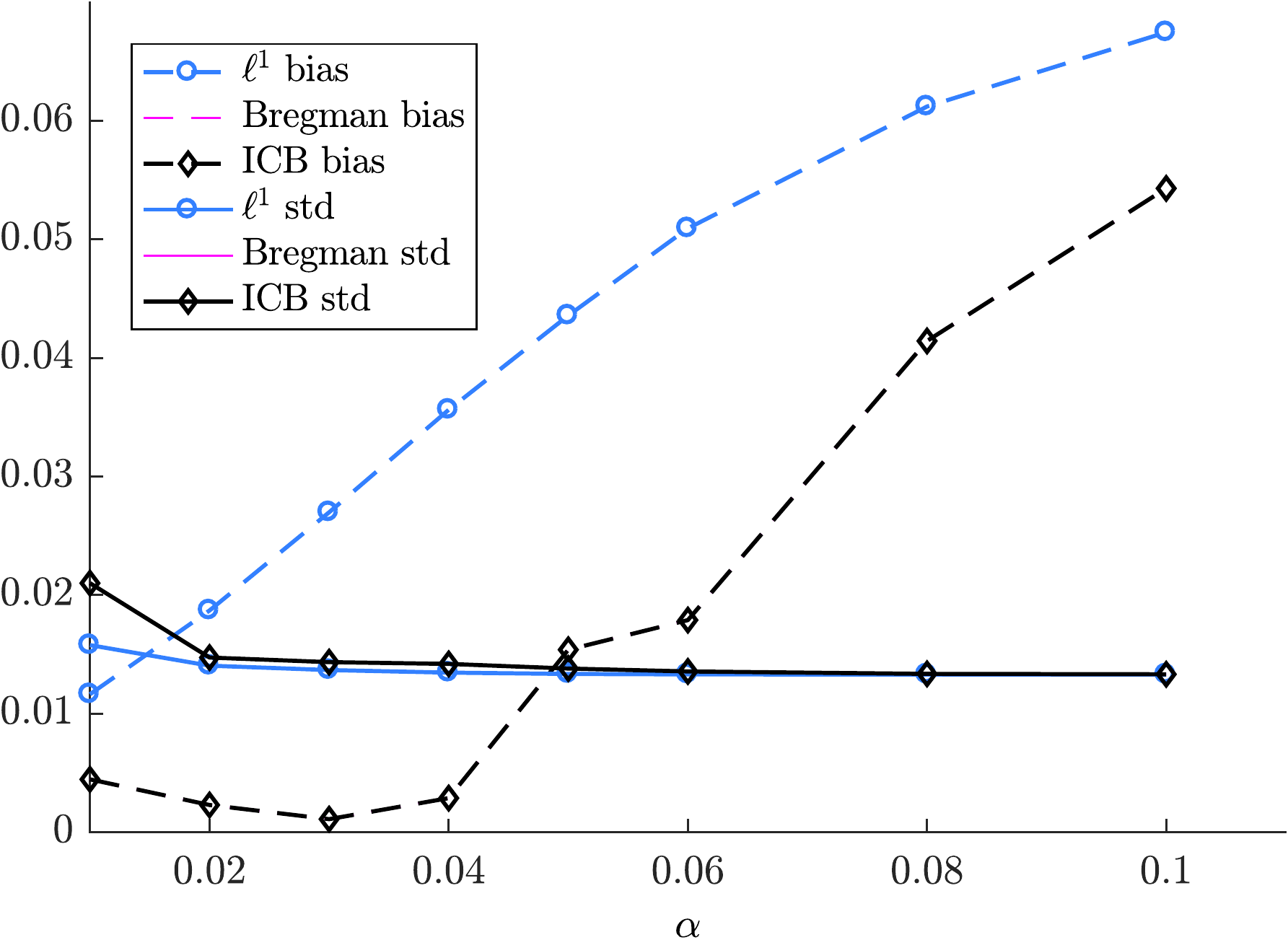}
\caption{$\ell^1$-deconvolution of a 1D signal. 
Average bias and variance computed over one thousand realizations of the noisy signal for
the noisy, restored and debiased signals.\label{fig:L1deconv1D_b}}
\end{figure}

\section{Numerical Results}
\label{sec:exp}

This section provides further experiments and numerical results that illustrate the proposed debiasing method.

\subsection{$\ell^1$-deconvolution}

The first application that we illustrate is the deconvolution of a one-dimensional signal using 
anisotropic shrinkage \eqref{eq:aniso_l1}.
Figure \ref{fig:L1deconv1D_a} displays the original signal, the blurry signal corrupted by additive Gaussian noise with standard deviation $\sigma=0.05$, 
the $\ell^1$-reconstructed signal and the debiased signals computed over the Bregman manifold $\MB$ 
and the infimal convolution subspace $\MIC$.
The last two completely overlap on these two plots.
One can see that provided that the $\ell^1$-reconstruction finds the right peak locations, 
the debiasing method is able to restore the amplitude of the original signal.

Figure \ref{fig:L1deconv1D_b} displays the evolution of the average bias of the estimated signals as well as the standard deviation of the error. 
They were computed over one thousand noisy realizations for the noisy, $\ell^1$-reconstructed and debiased signals, 
as a function of the regularization parameter $\alpha$.
These curves illustrate several behaviors: 
As expected, the residual variance decreases when the regularization parameter increases.
For a very low value of $\alpha$, the debiasing reintroduces some noise so the average variance is higher 
than for the $\ell^1$-reconstructed signal, revealing the bias-variance trade-off that has to be settled.
As $\alpha$ increases, the gap between the variance of the $\ell^1$-reconstructed and debiased signal vanishes.
On the other hand, the average bias is indeed smaller for the debiased signal than for the $\ell^1$-reconstructed signal.
Besides, for small values of the regularization parameter the average bias for the debiased signal is stable and close to zero,
showing the effective reduction of the method bias.
Then it increases by steps which correspond to the progressive vanishing of the peaks, related to model bias.
All in all, these plots show the ability of the proposed approach to reduce the method bias 
(here, the loss of intensity on the peaks), hence allowing for more efficient noise reduction and reconstruction  
for a wider range of regularization parameters.

\subsection{Anisotropic TV denoising}
In this subsection we study debiasing by means of the discrete ROF-model \cite{rof} given by:
\begin{equation} \label{ROFmodel}
	\ua(f) \in \arg\min_{u\in\R^n}  \frac{1}{2}\|u-f\|_{2}^2 + \alpha \Vert \Gamma u \Vert_1,
\end{equation}
where the $1$-norm is anisotropic, i.e.
\begin{align*}
\Vert \Gamma u \Vert_1 = \sum_{i=1}^{m/2} \vert (\Gamma u)_{1,i} \vert + \vert (\Gamma u)_{2,i} \vert,
\end{align*}
with $(\Gamma u)_1$ and $(\Gamma u)_2$ denoting the discrete gradient images in horizontal and vertical direction, respectively. 
We compare the original denoising result of Problem \eqref{ROFmodel} to the proposed debiased solutions 
obtained with the Bregman manifold $\MB$ or the infimal convolution subspace $\MIC$.

\begin{figure*}[ht!]
\center
\begin{tabular}{ccc}
(a) TV / residual vs. $\alpha$ & (b) PSNR vs. $\alpha$, {\it Giraffe} & (c) PSNR vs. $\alpha$, {\it Parrot}\\
\includegraphics[width=0.3\textwidth]{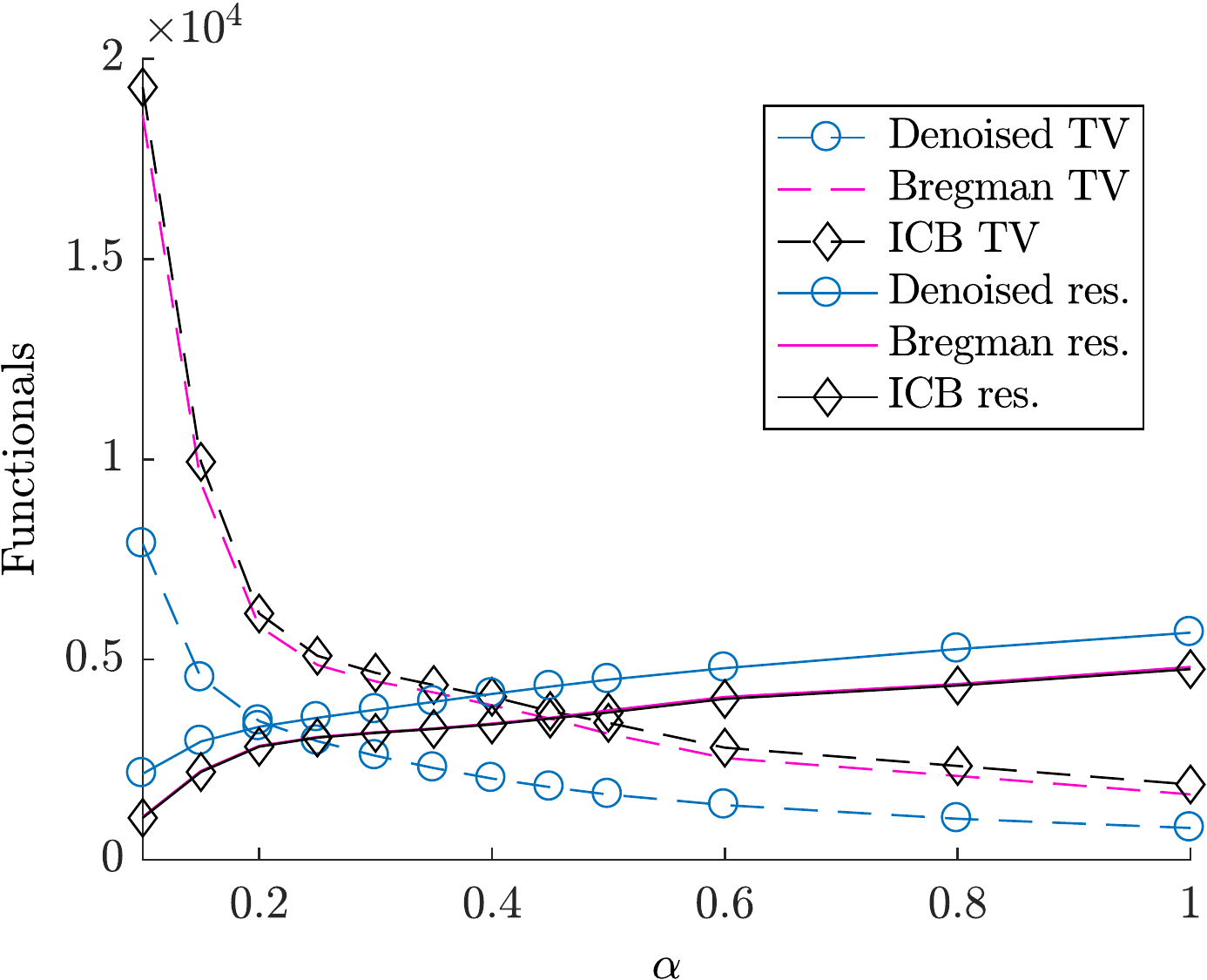}&
\includegraphics[width=0.3\textwidth]{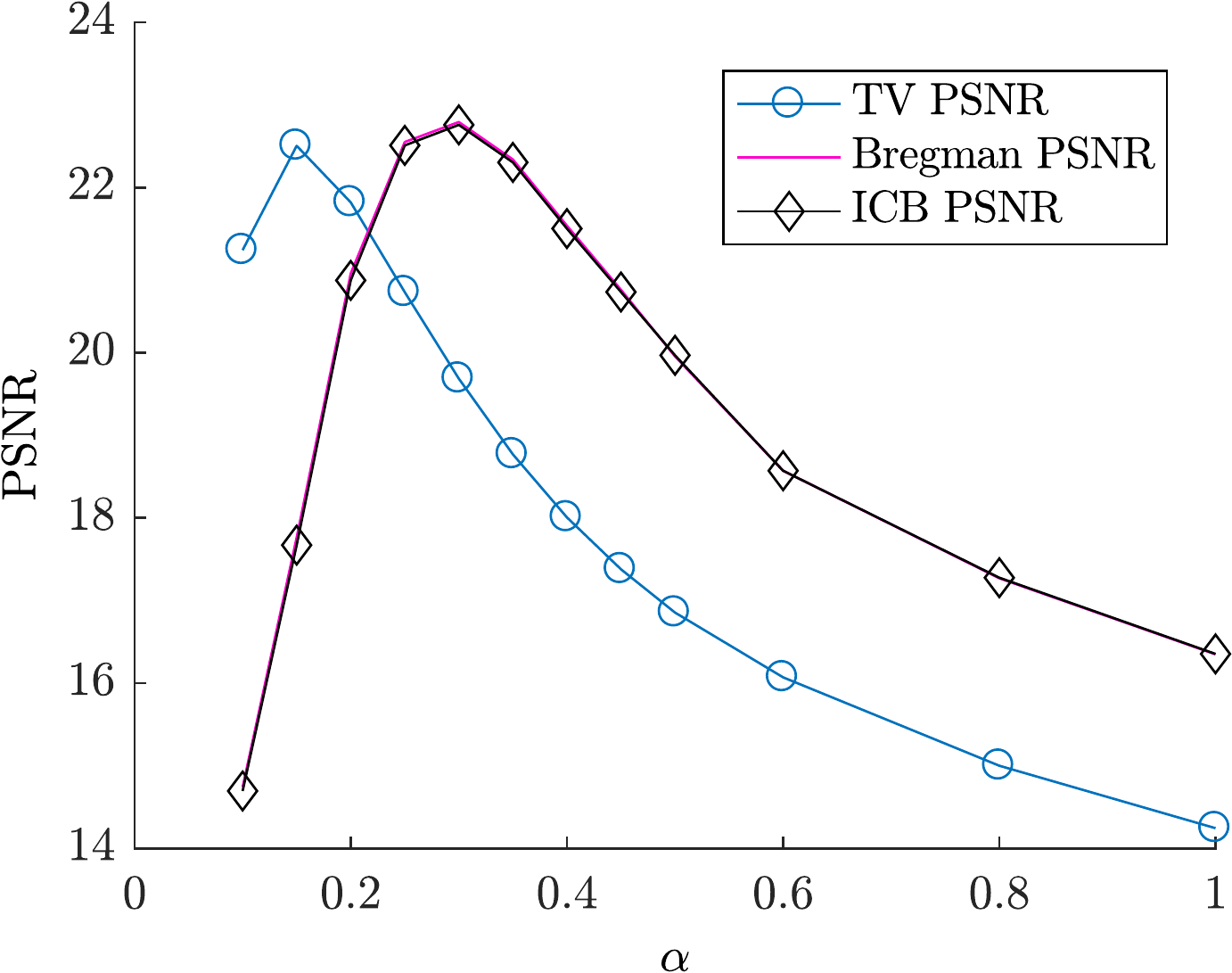}&
\includegraphics[width=0.3\textwidth]{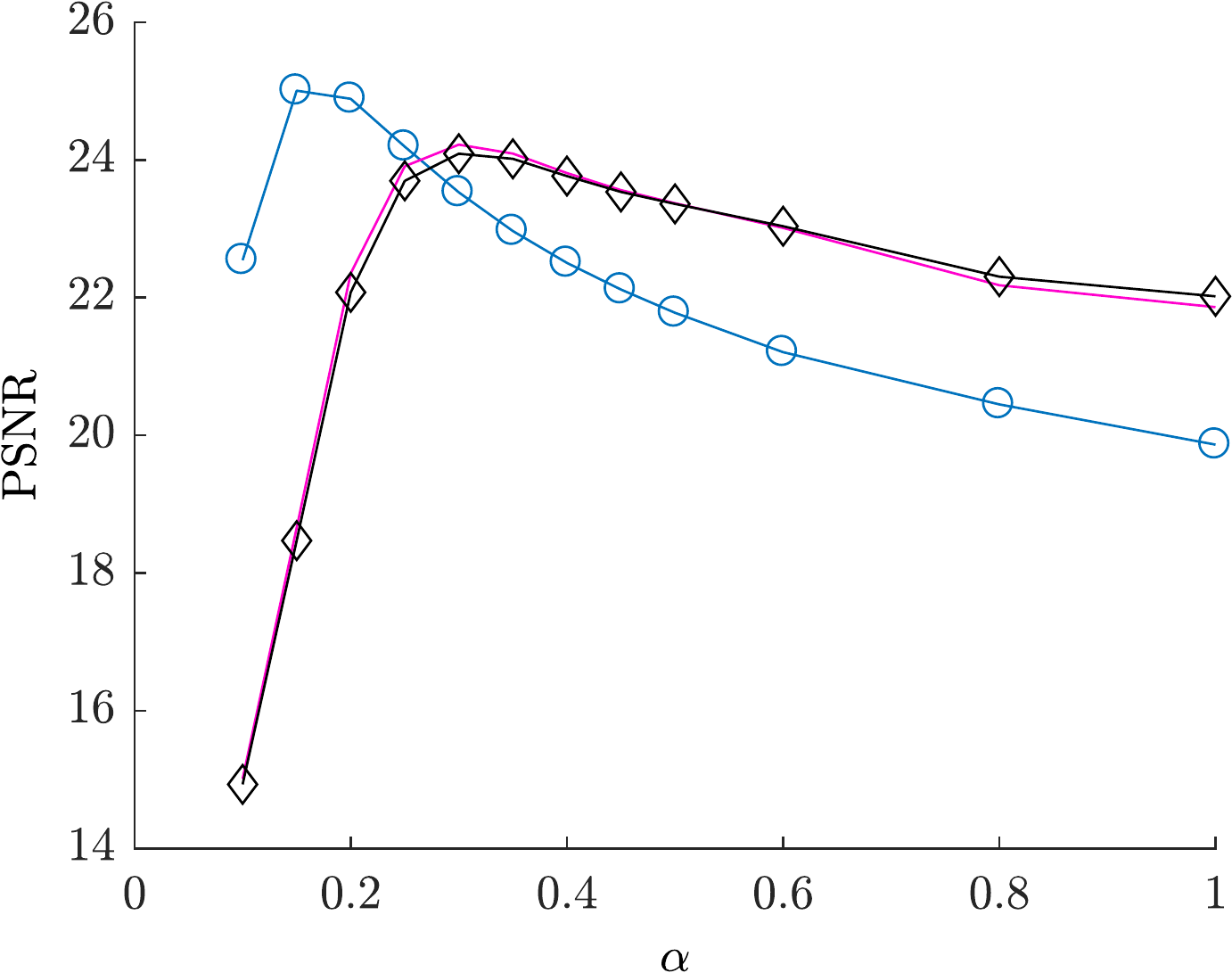}
\end{tabular}
\caption{Evolution of (a) The total variation and the residual for the cartoon {\it Giraffe} image, 
(b) The average PSNR for TV denoising, Bregman debiasing and infimal convolution debiasing for the cartoon {\it Giraffe} image
and (c) The average PSNR for TV denoising, Bregman debiasing and infimal convolution debiasing for the natural {\it Parrot} 
image as a function of the regularization parameter $\alpha$.\label{fig:TVres_psnr}}
\end{figure*}

\subsubsection{Cartoon image}

The {\it Giraffe} cartoon image has been designed not to have model bias; 
it is piecewise constant, which makes it suitable for TV denoising and allows us to study the reduction 
of the method bias only.
It takes values in $[0,1]$ and 
has been artificially corrupted with additive Gaussian noise with zero mean and variance $\sigma^2=0.05$, 
reaching an initial PSNR of about 13dB. The original image and a noisy realization are already 
displayed on the first line of Fig.~\ref{fig:firstResults} in Section \ref{sec:Debiasing}.

Figure \ref{fig:giraffe} displays the TV denoising result as well as 
the debiased solutions computed on the Bregman manifold $\MB$
or the infimal convolution subspace $\MIC$ for different values of the regularization parameter $\alpha$.
On the first line, $\alpha=0.15$ is the optimal 
regularization parameter for TV denoising (in terms of PSNR, see Fig.~\ref{fig:TVres_psnr}-(b)).
However, when performing the debiasing, noise is strongly amplified.
On the second line, $\alpha=0.3$ is the optimal  
regularization parameter for debiasing, and overall, (in terms of PSNR, see Fig.~\ref{fig:TVres_psnr}-(b)).
On the third line $\alpha=0.6$ leads to an oversmoothed solution, but the debiasing step still allows to recover 
a lot of the lost contrast. 

Since we expect the variational method to systematically underestimate the value of the regularization functional and 
overestimates the residual (see \cite{benninggroundstates} for a precise computation on singular values), 
we compare the corresponding quantities when varying $\alpha$ in Figure \ref{fig:TVres_psnr}-(a). 
We observe that for a very large range of values of $\alpha$ there appears to be an almost constant offset between the values 
for the solution $\ua(f)$ and the debiased solution $\uhat(f)$ (except for very small values of $\alpha$, when noise dominates). 
This seems to be due to the fact that the debiasing step can correct the bias in the regularization functional 
(here total variation) 
and residual to a certain extent. This corresponds well to the plot of PSNR vs. $\alpha$ in Fig.~\ref{fig:TVres_psnr}-(b), 
which confirms that the PSNR after the debiasing step is significantly larger than the one in $\ua(f)$ 
for a large range of values of $\alpha$, which contains the ones relevant in practice. 
The fact that the PSNR is decreased by the debiasing step for very small $\alpha$ corresponds to the fact that indeed 
the noise is amplified in such a case, visible also in the plots for the smallest value of $\alpha$ in Figure \ref{fig:giraffe}. 

Altogether, these results show that the proposed debiasing approach improves the denoising of the
cartoon image both visually and quantitatively.

\begin{figure*}[th!]
\center
\begin{tabular}{m{0.02\textwidth}>{\centering\arraybackslash}m{0.28\textwidth}>{\centering\arraybackslash}m{0.28\textwidth}>{\centering\arraybackslash}m{0.28\textwidth}}
& {\bf \large TV denoising} & {\bf \large Bregman debiasing} & {\bf \large ICB debiasing}\\
\rotatebox{90}{$\alpha=0.15$} & 
\includegraphics[width=0.28\textwidth]{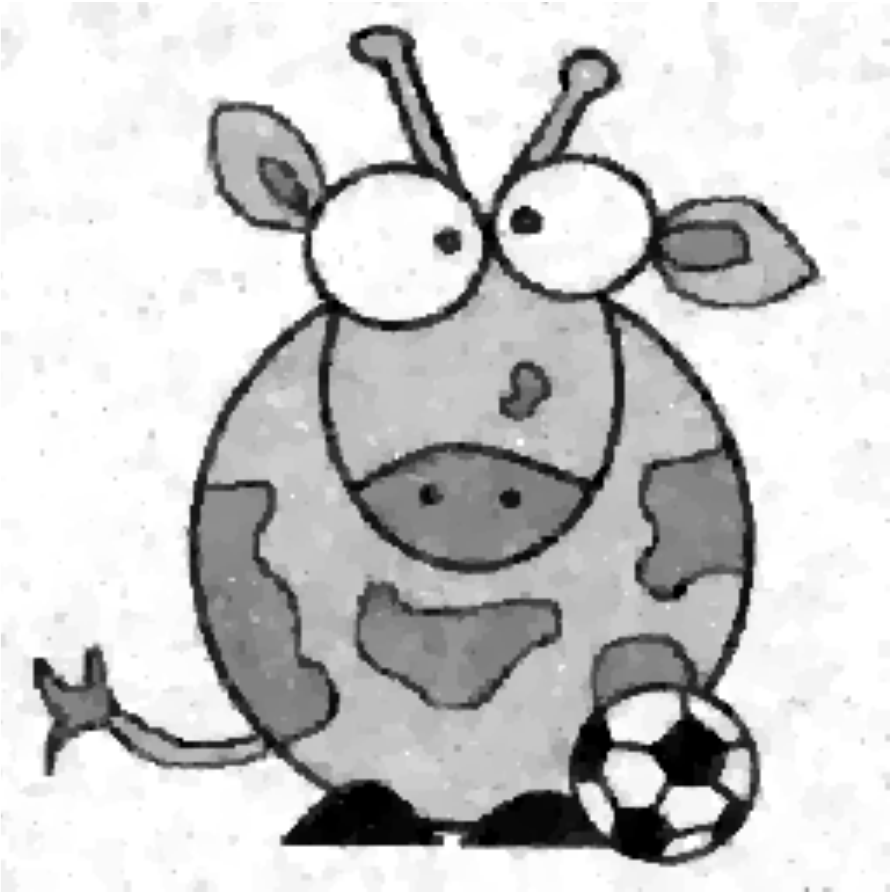}&
\includegraphics[width=0.28\textwidth]{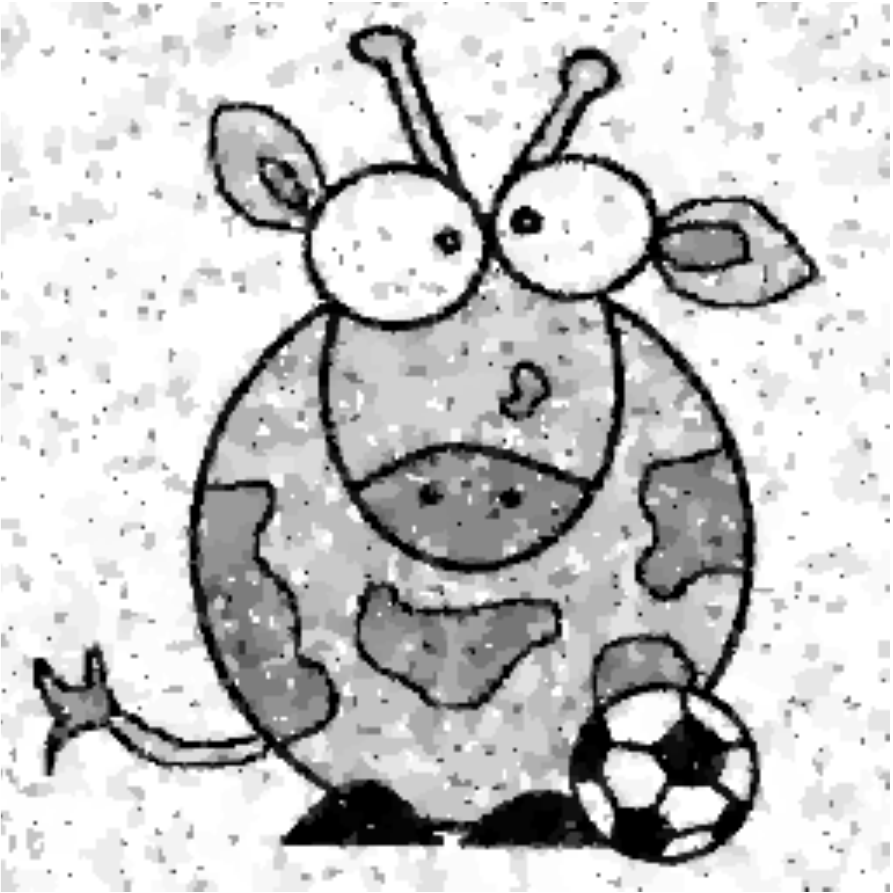}&
\includegraphics[width=0.28\textwidth]{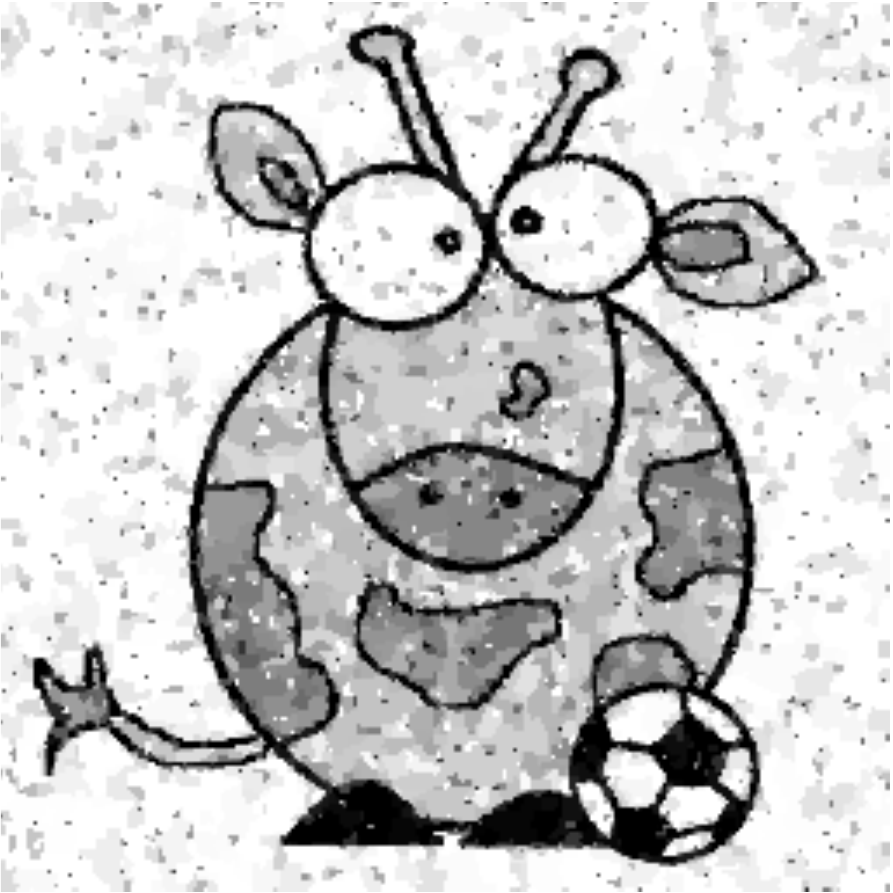}\\
& $PSNR = 22.43$ & $PSNR = 17.82$ & $PSNR = 17.69$\\
\rotatebox{90}{$\alpha=0.3$ (optimal)} & 
\includegraphics[width=0.28\textwidth]{Giraffe_TVdenoised_lambda0-3}&
\includegraphics[width=0.28\textwidth]{Giraffe_Bregmandebiased_lambda0-3}&
\includegraphics[width=0.28\textwidth]{Giraffe_ICdebiased_lambda0-3}\\
& $PSNR = 19.63$ & $PSNR = 22.75$ & $PSNR = 22.70$\\
\rotatebox{90}{$\alpha=0.6$} & 
\includegraphics[width=0.28\textwidth]{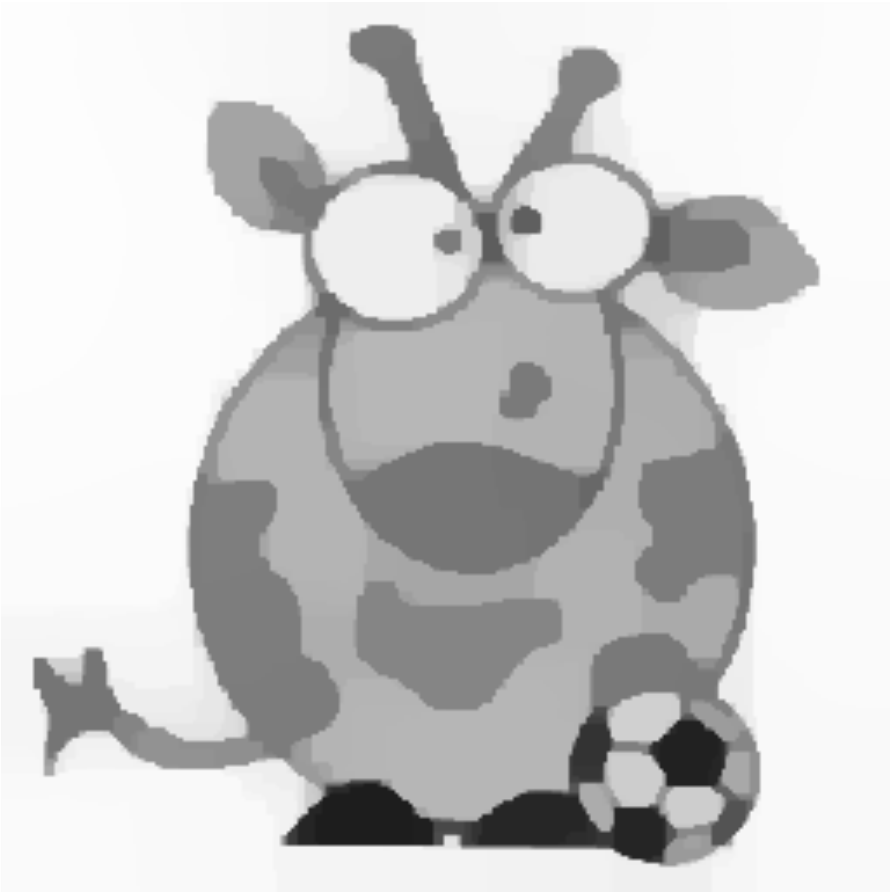}&
\includegraphics[width=0.28\textwidth]{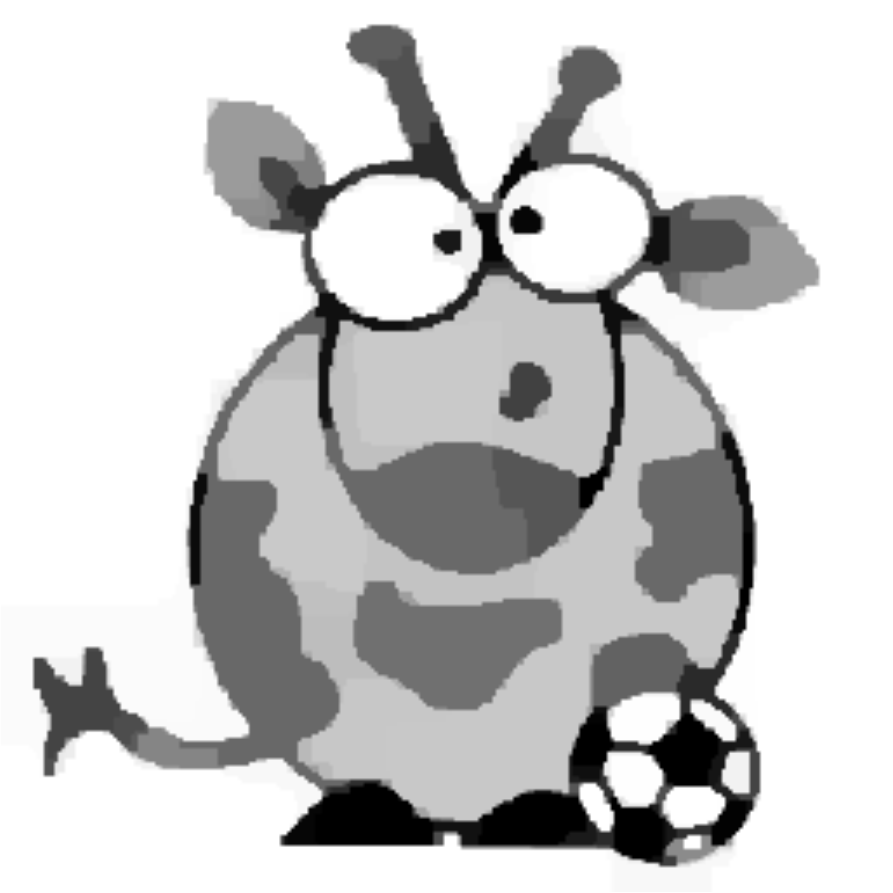}&
\includegraphics[width=0.28\textwidth]{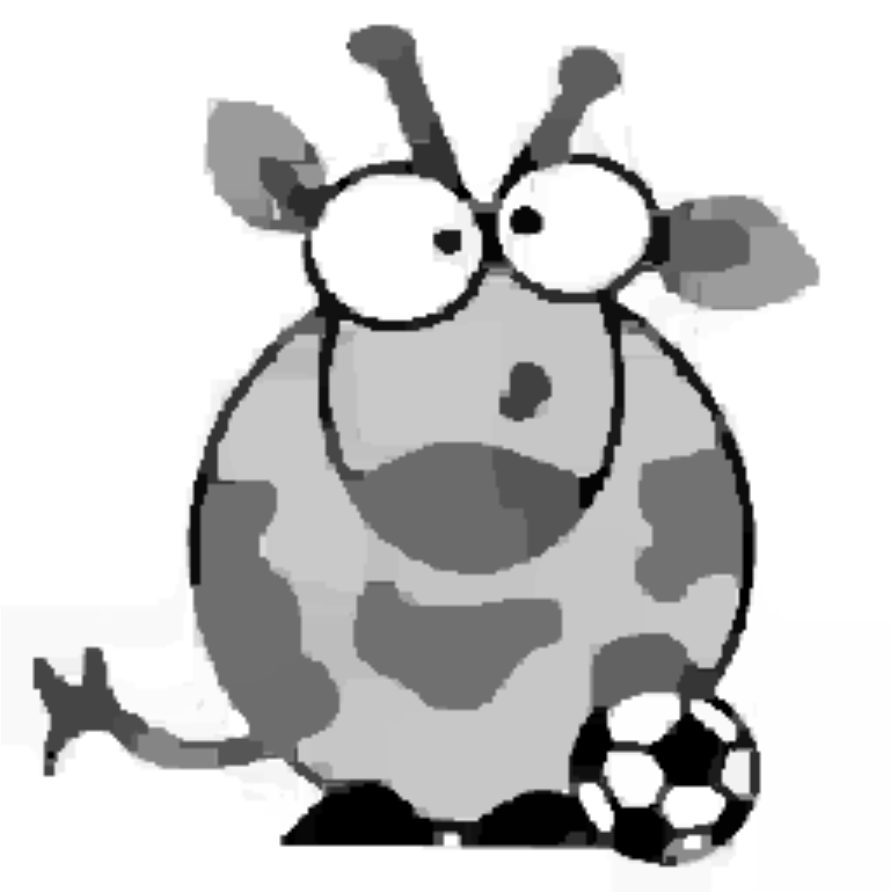}\\
& $PSNR = 16.05$ & $PSNR = 18.19$ & $PSNR = 18.34$\\
\end{tabular}
\caption{Denoising of the {\it Giraffe} cartoon image for different values of the regularization parameter $\alpha$.
First column: TV denoising. Second column: Debiasing on the Bregman manifold.
Third column: Debiasing on the infimal convolution subspace. 
\label{fig:giraffe}}
\end{figure*}

\subsubsection{Natural image}

The debiasing can also be evaluated on natural images such as the {\it Parrot} 
picture.
TV denoising on such images leads to both method bias and model bias.
We expect to reduce the former with the proposed method, while the latter is due to the piecewise constant approximation 
associated with the ROF-model.
The {\it Parrot} image takes values in $[0,1]$ and 
has been artificially corrupted with additive Gaussian noise with zero mean and variance $\sigma=0.05$, 
reaching an initial PSNR of about 13dB. The original image and a noisy realization are 
displayed on the first line of Figure \ref{fig:parrot}.

Analogously to Figure \ref{fig:giraffe}, Figure \ref{fig:parrot} also displays the TV denoising result as well as 
the debiased solutions computed on the Bregman subspace 
or the infimal convolution subspace for different values of the regularization parameter $\alpha$.
On the second line, $\alpha=0.15$ is the optimal 
regularization parameter for TV denoising (in terms of PSNR, see Fig.~\ref{fig:TVres_psnr}-(c)).
However, when performing the debiasing, the remaining noise is strongly amplified.
On the third line, $\alpha=0.3$ is the optimal
regularization parameter for debiasing (in terms of PSNR, see Fig.~\ref{fig:TVres_psnr}-(c)).
On the fourth line $\alpha=0.6$ leads to an oversmoothed solution but the debiasing step still allows to recover the lost contrast.

Note that in the {\it Parrot} case, the optimal result in terms of PSNR is obtained for the TV denoising,
for $\alpha=0.15$.
However, the debiasing obtained with $\alpha=0.3$ visually provides a smoother result on the background, 
while preserving the fine structures such as the stripes around the eye.

Note also that in each case the artifacts of TV denoising such as staircasing remain and even become more apparent. 
This however seems natural as the contrast is increased. 
Since these issues are in fact model bias they are not dealt with by the debiasing method we perform here, 
but could be reduced by an appropriate choice of regularization such as total generalized variation \cite{TGV}.

\subsubsection{Statistical behavior}

For both images, the statistical behavior of the proposed debiasing methods can be evaluated by computing 
the statistical bias $ \E[u^* - \Uhat]$ as well as the variance $\var[u^* - \Uhat]$ between the true image $u^*$ and an estimator $\Uhat$.
In our case this is either the solution of the ROF-model \eqref{ROFmodel} or the corresponding debiased result. 
Figure \ref{fig:_curve} displays the evolution of the estimated statistical bias and standard deviation of the 
TV, Bregman debiased and infimal convolution debiased estimators for the cartoon {\it Giraffe} and natural {\it Parrot} images,
as a function of the regularization parameter $\alpha$.
These curves reflect some interesting behaviors: 
As expected, the residual variance decreases as the regularization parameter increases.
Besides, the variance is always slightly higher for the debiased solutions, 
which reflects the bias-variance compromise that has to be settled.
However, as the regularization parameter increases, the gap between the denoised and debiased variance decreases.
On the other hand, as the regularization parameter grows, the bias increases for each method, 
and it always remains higher for the denoised solutions than for the debiased solutions.
One interesting fact is the behavior of the bias curve for the cartoon {\it Giraffe} image:
for low values of the regularization parameter (up to $\alpha \approx 0.3$), 
the evolution of the bias for the debiased solutions is relatively stable. 
This means that for those values, one can increase the regularization parameter in order to reduce the variance
without introducing too much (at this point, method) bias.
Then, for higher regularization parameters the bias increases in a steeper way, parallel to the evolution of the original bias 
for the TV denoised image. This reflects the evolution of the model bias from this point on, 
when the high regularization parameter provides a model subspace whose elements are too smooth compared to the true image.
For the natural {\it Parrot} image, the model bias occurs even for small values of the regularization parameter, 
because the model manifold provided by the TV regularization does not properly fit the image prior.

\begin{figure*}[th!]
\center
\begin{tabular}{cc}
(a) {\it Giraffe} & (b) {\it Parrot}\\
\includegraphics[width=0.45\textwidth]{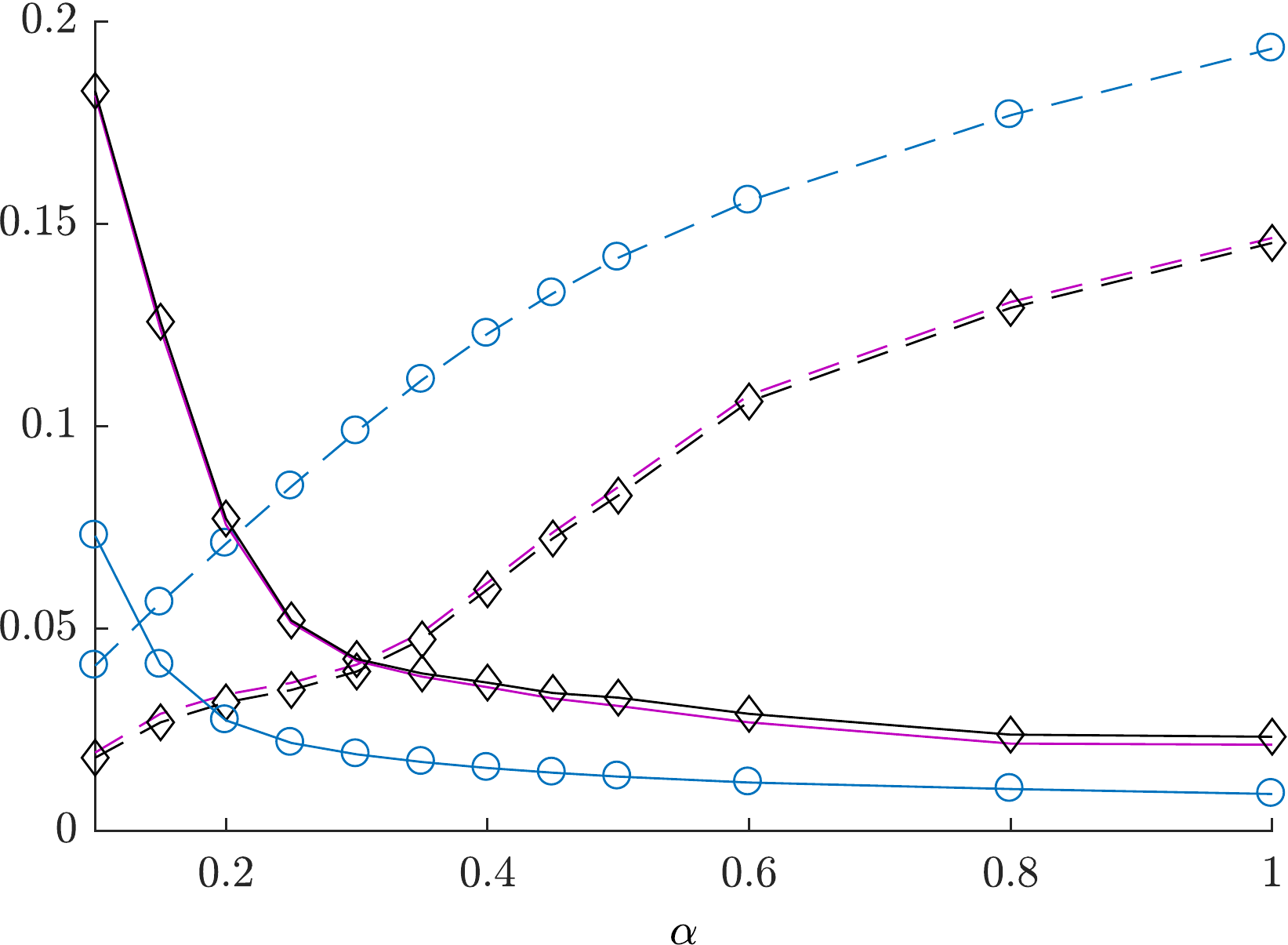}&
\includegraphics[width=0.45\textwidth]{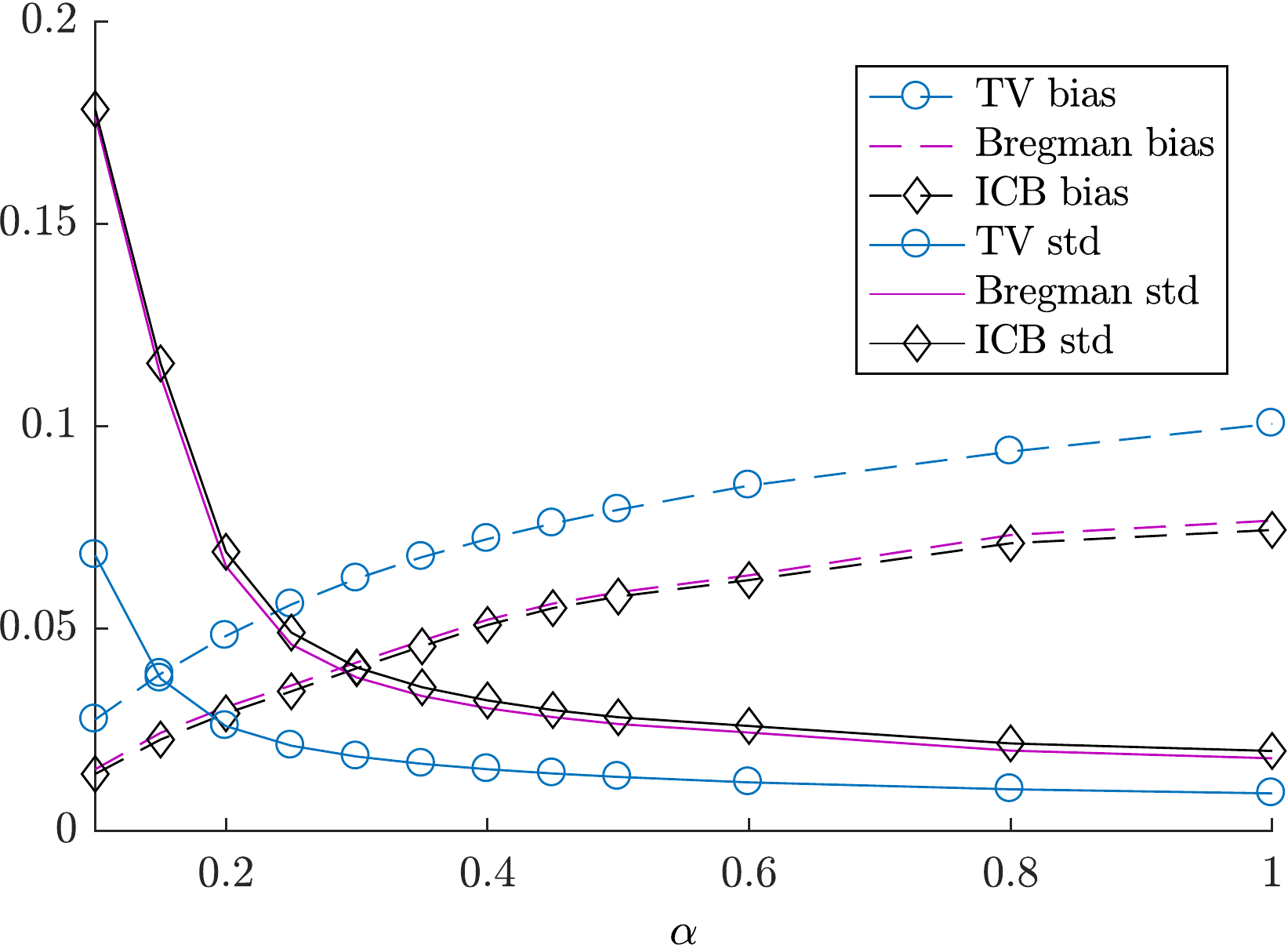}
\end{tabular}
\caption{Evolution of the average residual bias and standard deviation 
computed over 500 noisy realizations of (a) {\it Giraffe} and
(b) {\it Parrot} for TV denoising, Bregman debiasing and infimal convolution debiasing. \label{fig:_curve}}
\end{figure*}

These curves also illustrate the optimal bias-variance balance that can be achieved with or without the debiasing procedure.
Intuitively, one would expect the optimal bias-variance trade-off to be reached 
when the bias and the standard deviation curves intersect each other.
This is indeed confirmed by the PSNR curves from Fig.~\ref{fig:TVres_psnr}-(b) and \ref{fig:TVres_psnr}-(c).
Looking at those intersection points on both curves for the TV denoised solution on the one hand 
and for the debiased solutions on the other hand, 
one can see that the optimal compromise for the debiasing is reached for a higher regularization parameter than for the 
denoising. This offers more denoising performance,
and it leads to a smaller (for the {\it Giraffe} image) or equal (for the {\it Parrot} image) 
average bias and standard deviation.

\subsection{Isotropic TV denoising}

\begin{figure*}[ht!]
\center
\begin{tabular}{cccc}
{\bf Isotropic TV} & {\bf Bregman debiasing} & {\bf ICB debiasing} & {\bf Bregman iterations}\\
\includegraphics[width=0.22\textwidth]{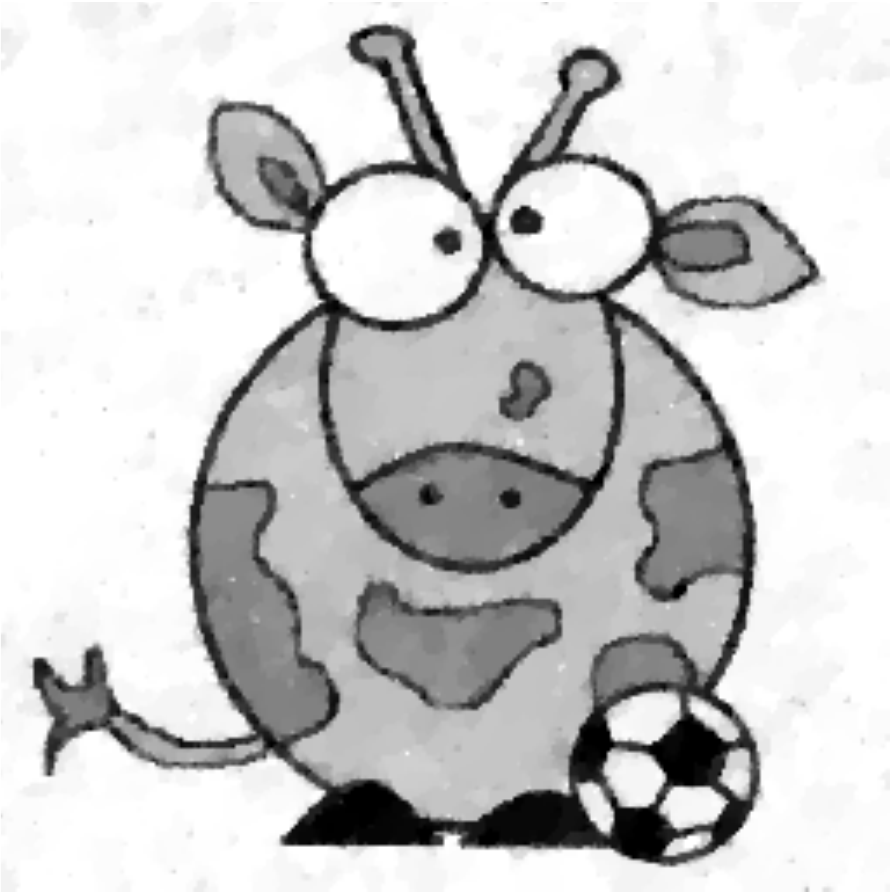}&
\includegraphics[width=0.22\textwidth]{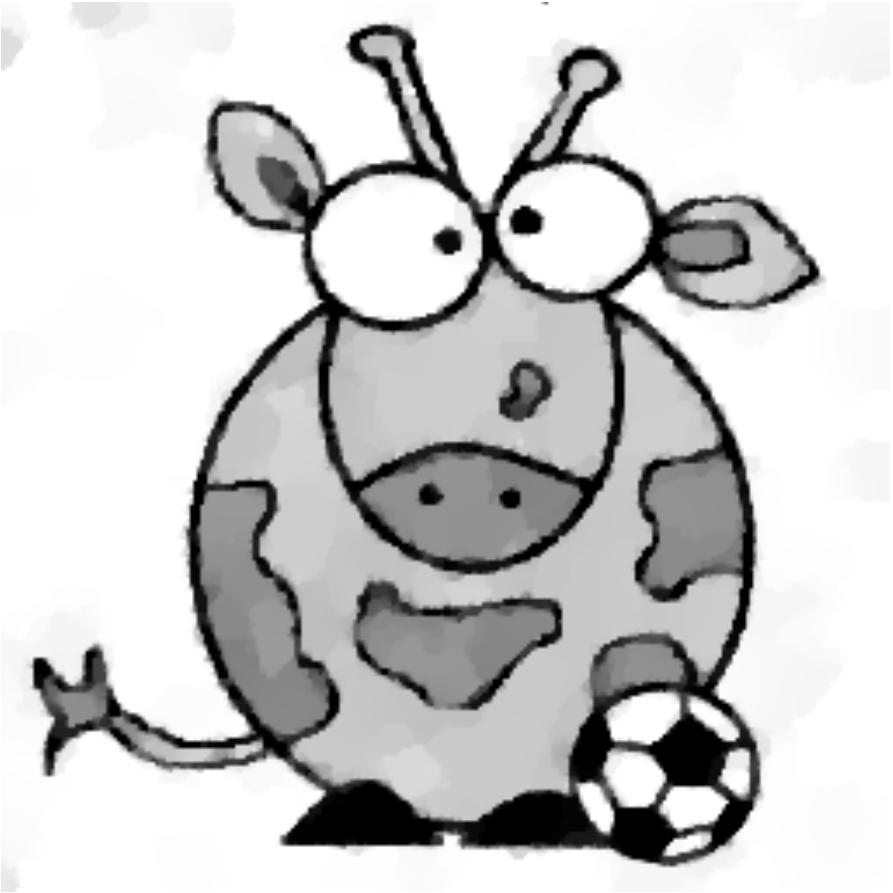}&
\includegraphics[width=0.22\textwidth]{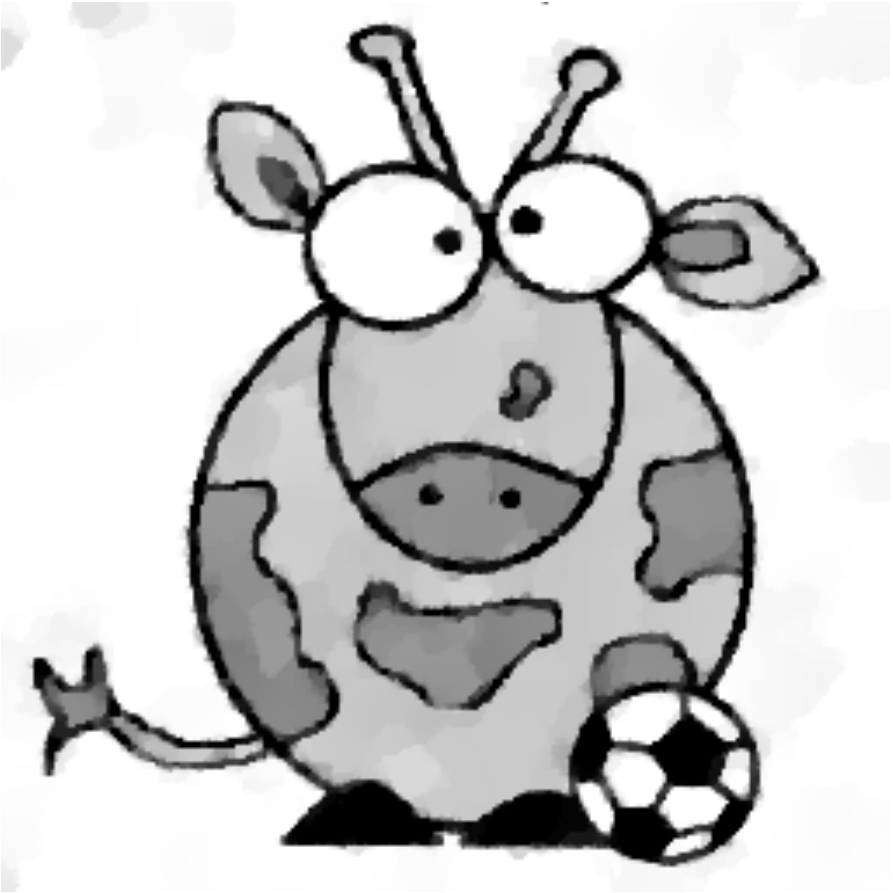}&
\includegraphics[width=0.22\textwidth]{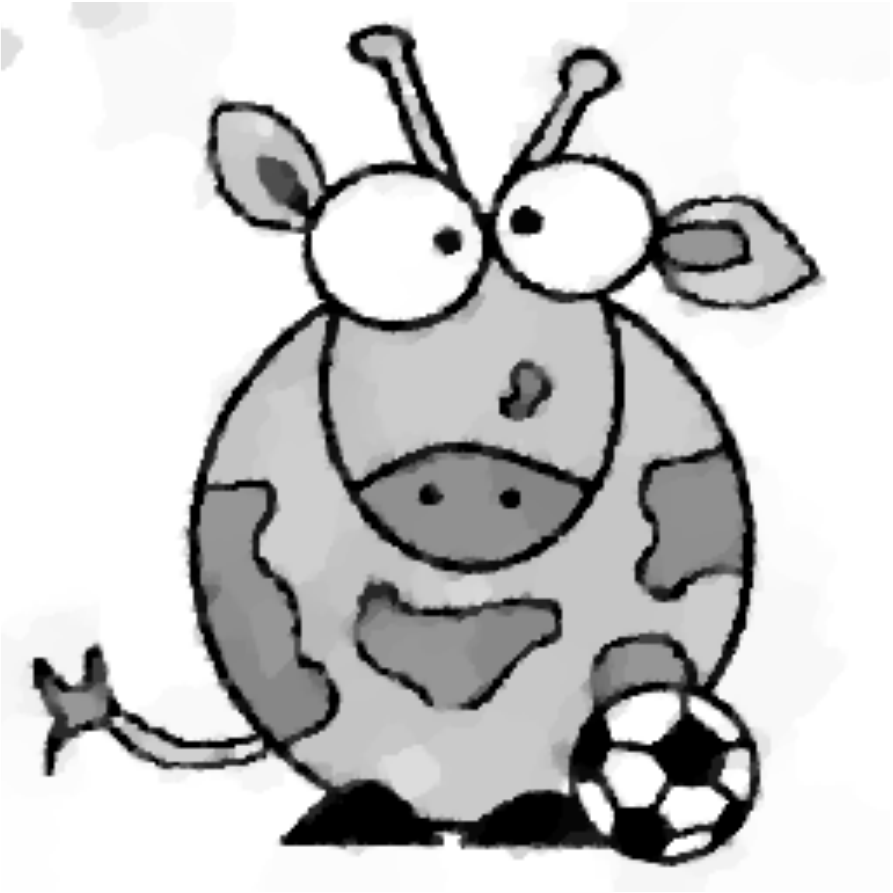}\\
$PSNR = 22.14$ & $PSNR = 22.49$ & $PSNR = 22.58$ & $PSNR = 22.97$\\
\includegraphics[width=0.22\textwidth]{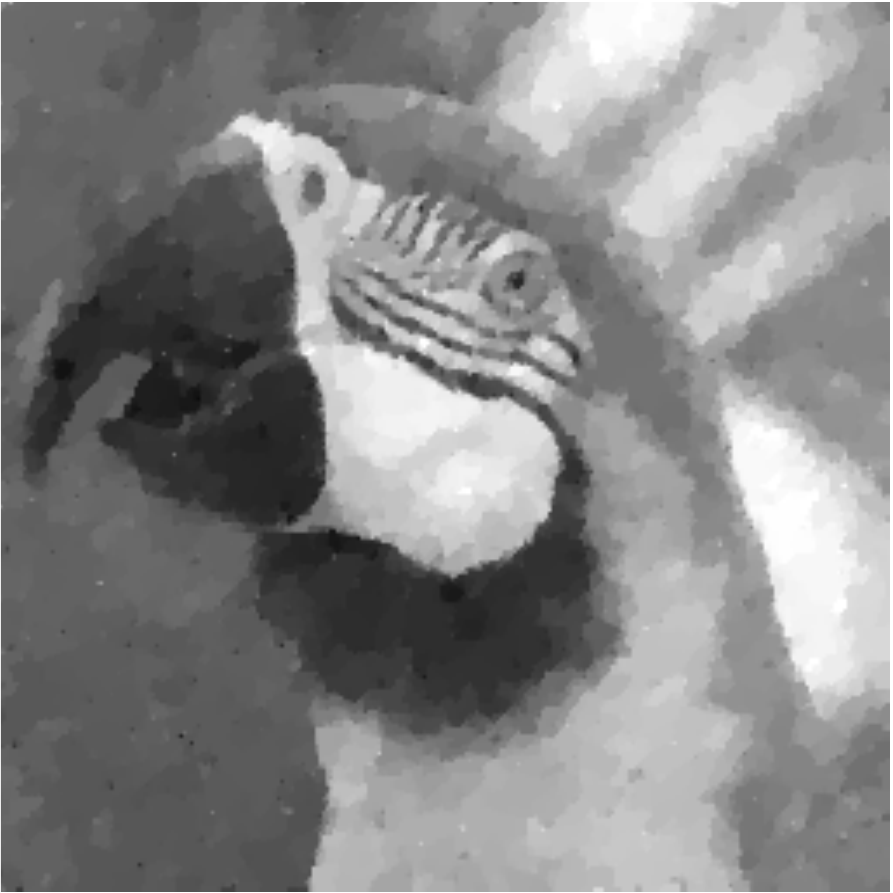}&
\includegraphics[width=0.22\textwidth]{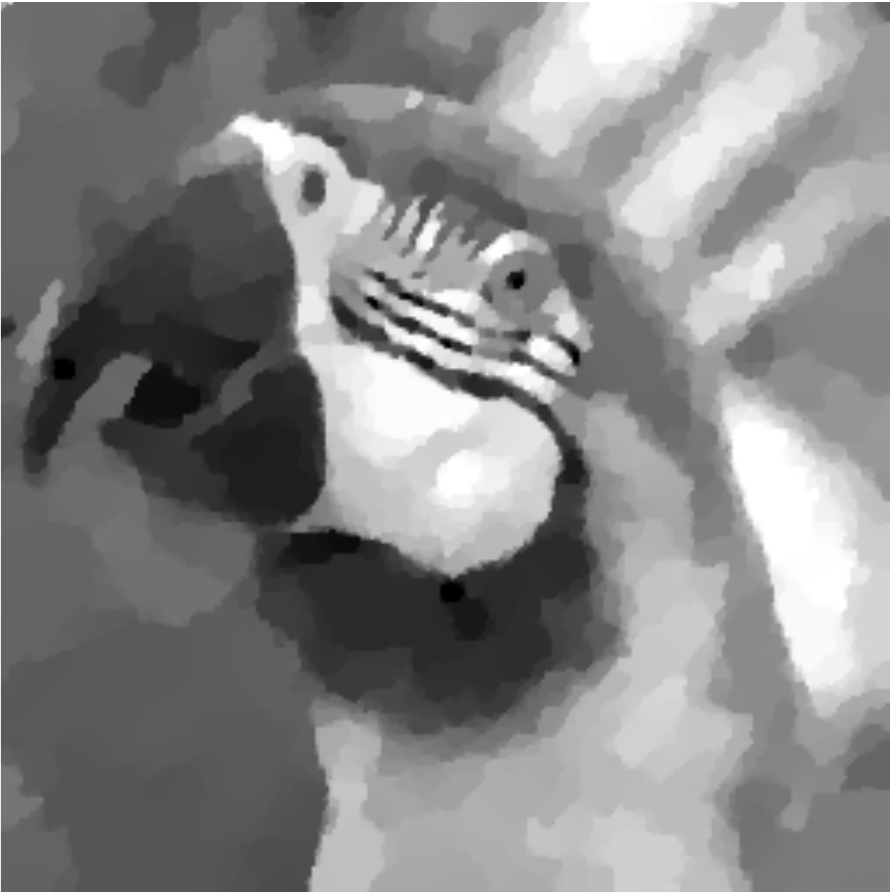}&
\includegraphics[width=0.22\textwidth]{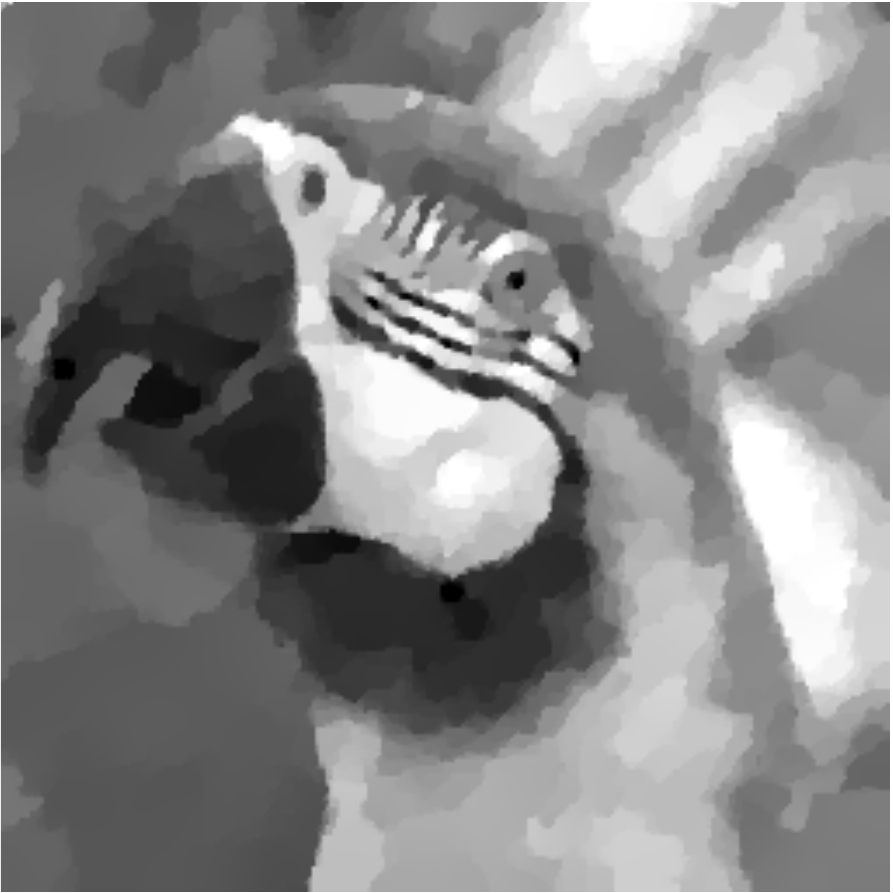}&
\includegraphics[width=0.22\textwidth]{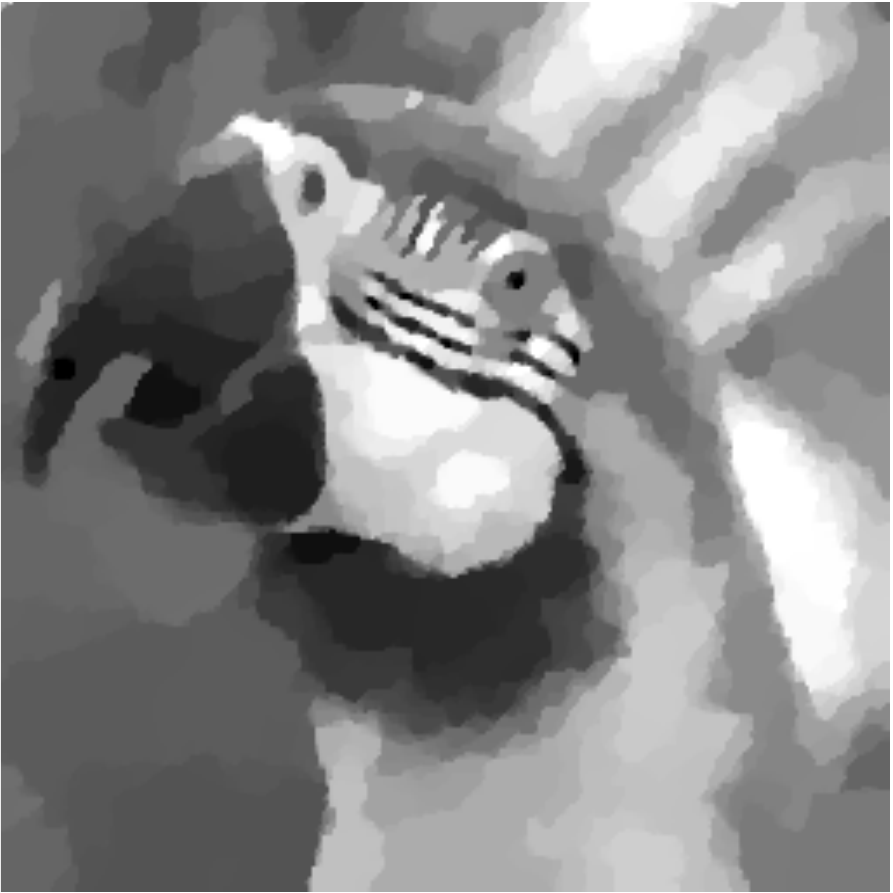}\\
$PSNR = 25.38$ & $PSNR = 24.69$ & $PSNR = 24.76$ & $PSNR = 24.60$\\
\end{tabular}
\caption{Isotropic TV denoising and debiasing of the cartoon {\it Giraffe} and natural {\it Parrot} images,
and comparison to Bregman iterations.
\label{fig:iso}}
\end{figure*}

Finally, we extend the examples presented in \cite{deledalle} with a few numerical results for isotropic TV denoising: 
\begin{align*}
\Vert \Gamma u \Vert_1 = \sum_{i=1}^{m/2} \sqrt{\vert (\Gamma u)_{1,i} \vert^2 + \vert (\Gamma u)_{2,i} \vert^2}.
\end{align*}
We then compare the denoising result to the solutions provided by the two alternative second steps of our debiasing method. Moreover, we also compare them to the result obtained from Bregman iterations.
Figure \ref{fig:iso} displays the optimal (in terms of PSNR) denoising and debiasing 
for the {\it Giraffe} and {\it Parrot} images.
The regularization parameter has been set to $\alpha=0.2$ for the denoising result and to $\alpha=0.3$ for the debiasing.
Similarly to the anisotropic case, the debiasing both visually and quantitatively improves the quality of the 
cartoon {\it Giraffe} image.
For the natural {\it Parrot} image, even though the PSNR is not improved by the debiasing process, 
one can still observe that the higher regularization parameter offers a better denoising of the background, 
while the debiasing guarantees that the fine structures around the eye are preserved with a good contrast.
Besides, the proposed debiasing approach offers similar results to Bregman iterations, displayed in the fourth column. 
However, the interesting aspect of our debiasing approach is that we only apply a two-step procedure, while Bregman iterations
have to be performed iteratively with a sufficiently high number of steps.
Note that our numerical approach to debiasing (see Section \ref{sec:Implementation}) is actually equivalent to performing one 
Bregman iteration with zero initialization of the subgradient, then updating the subgradient
and solving a second Bregman step with a sufficiently high regularization parameter. 

\begin{figure*}[hp!]
\center
\begin{tabular}{m{0.02\textwidth}>{\centering\arraybackslash}m{0.28\textwidth}>{\centering\arraybackslash}
m{0.28\textwidth}>{\centering\arraybackslash}m{0.28\textwidth}}
& & Original image & Noisy image\\
& & \includegraphics[width=0.28\textwidth]{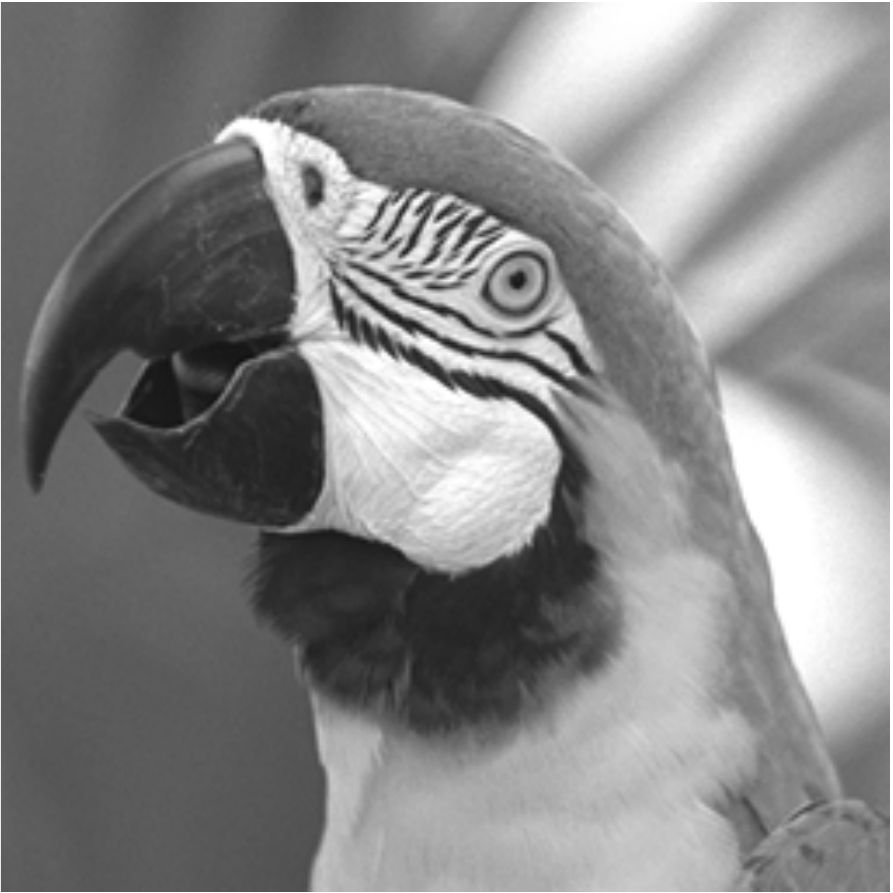}&
\includegraphics[width=0.28\textwidth]{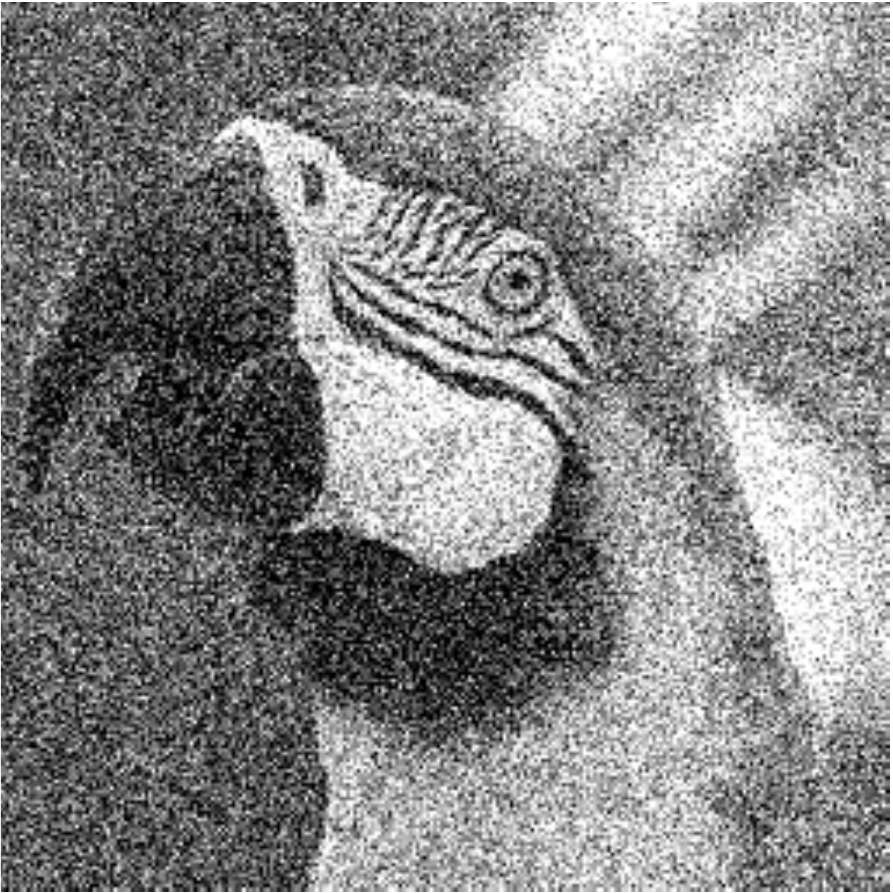}\vspace{1em}\\
& {\bf \large TV denoising} & {\bf \large Bregman debiasing} & {\bf \large ICB debiasing}\\
\rotatebox{90}{$\alpha=0.15$} & 
\includegraphics[width=0.28\textwidth]{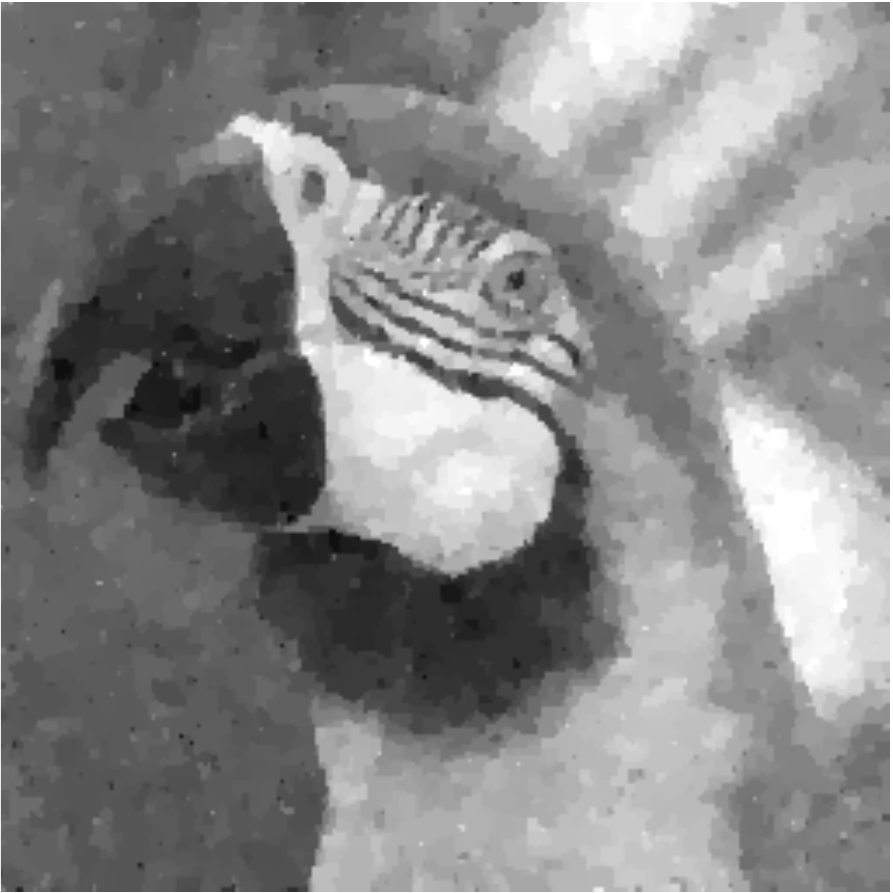}&
\includegraphics[width=0.28\textwidth]{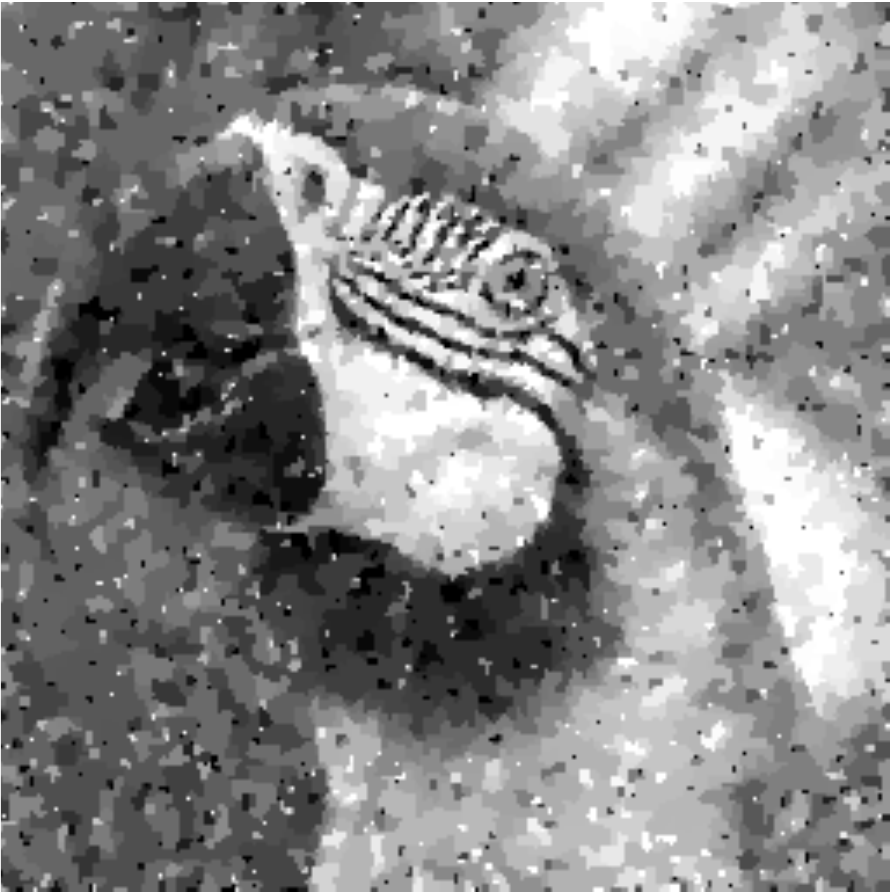}&
\includegraphics[width=0.28\textwidth]{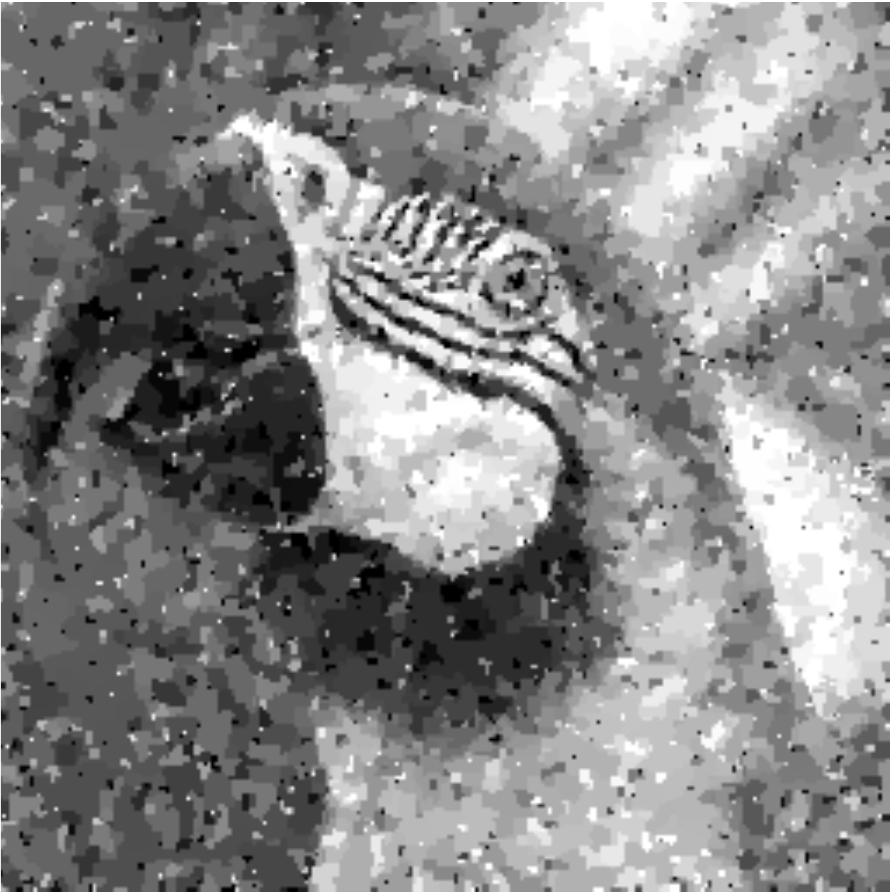}\\
& $PSNR = 25.07$ & $PSNR = 18.82$ & $PSNR = 18.44$\\
\rotatebox{90}{$\alpha=0.3$ (optimal)} & 
\includegraphics[width=0.28\textwidth]{Parrot_TVdenoised_lambda0-3}&
\includegraphics[width=0.28\textwidth]{Parrot_Bregmandebiased_lambda0-3}&
\includegraphics[width=0.28\textwidth]{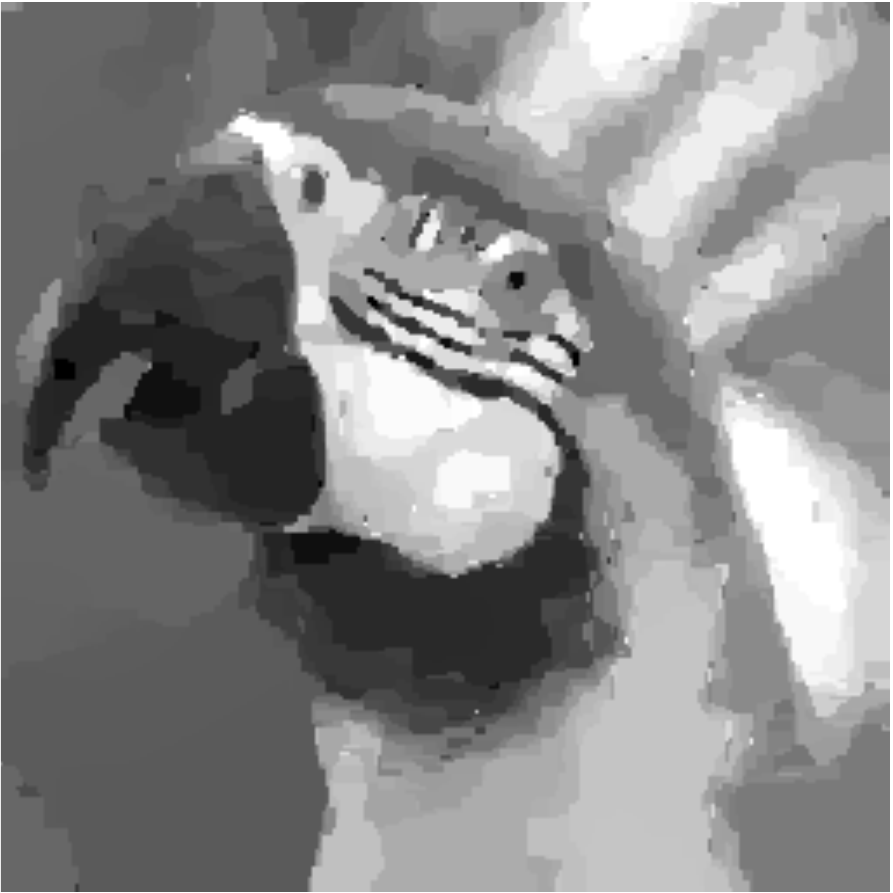}\\
& $PSNR = 23.57$ & $PSNR = 24.19$ & $PSNR = 23.95$\\
\rotatebox{90}{$\alpha=0.6$} & 
\includegraphics[width=0.28\textwidth]{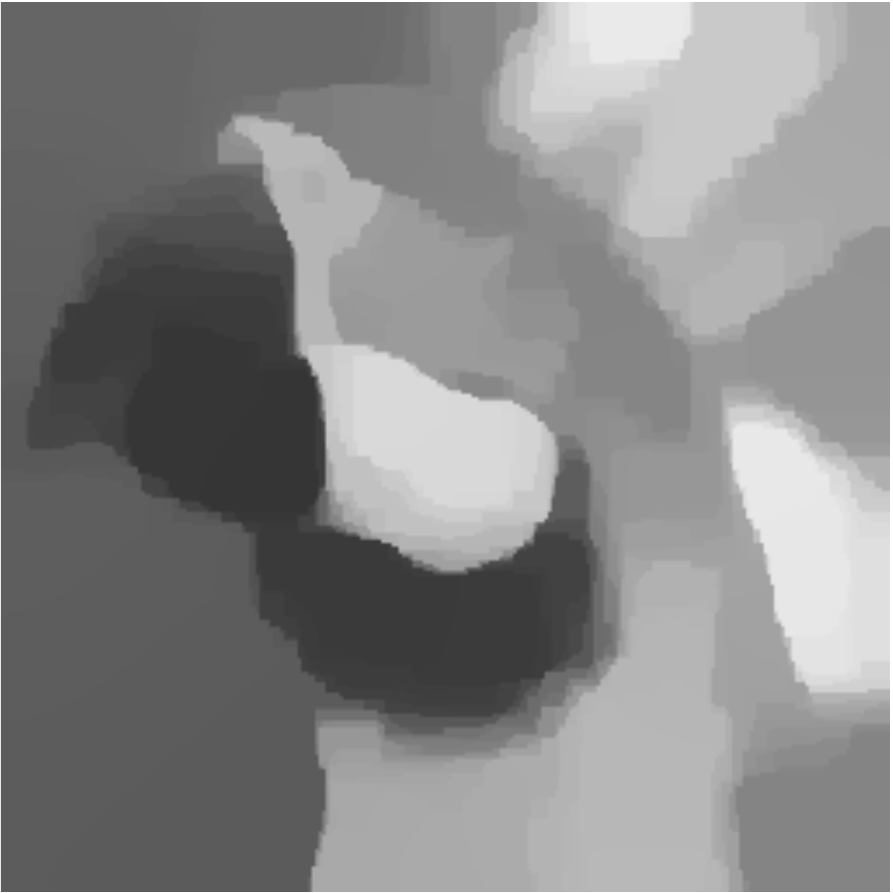}&
\includegraphics[width=0.28\textwidth]{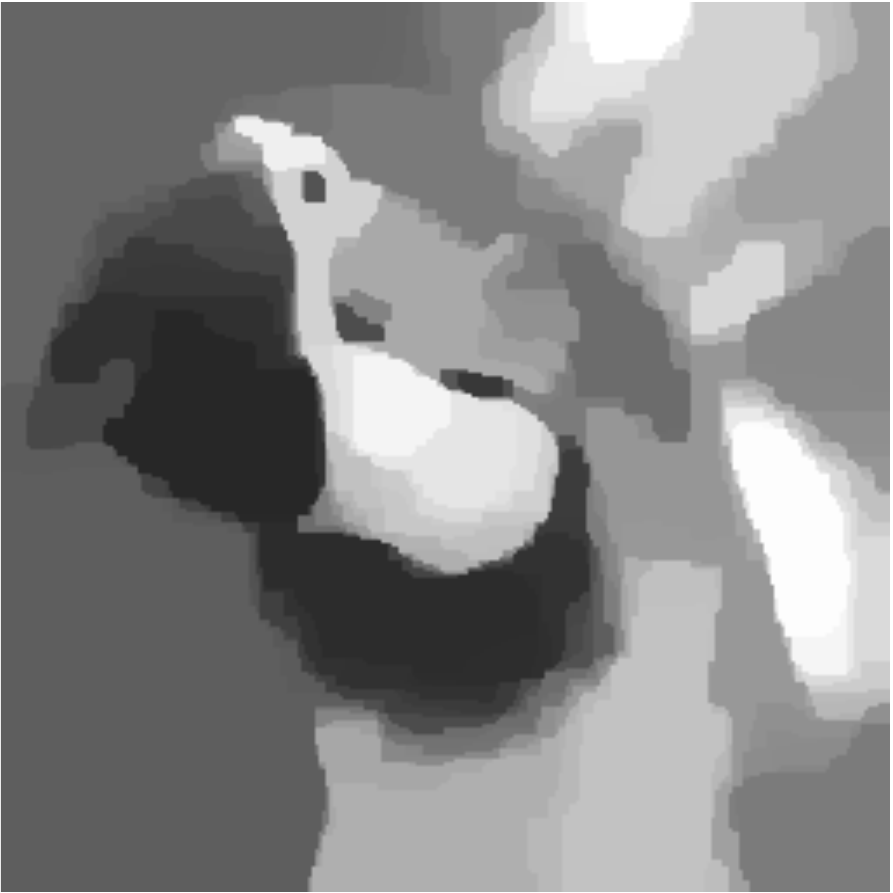}&
\includegraphics[width=0.28\textwidth]{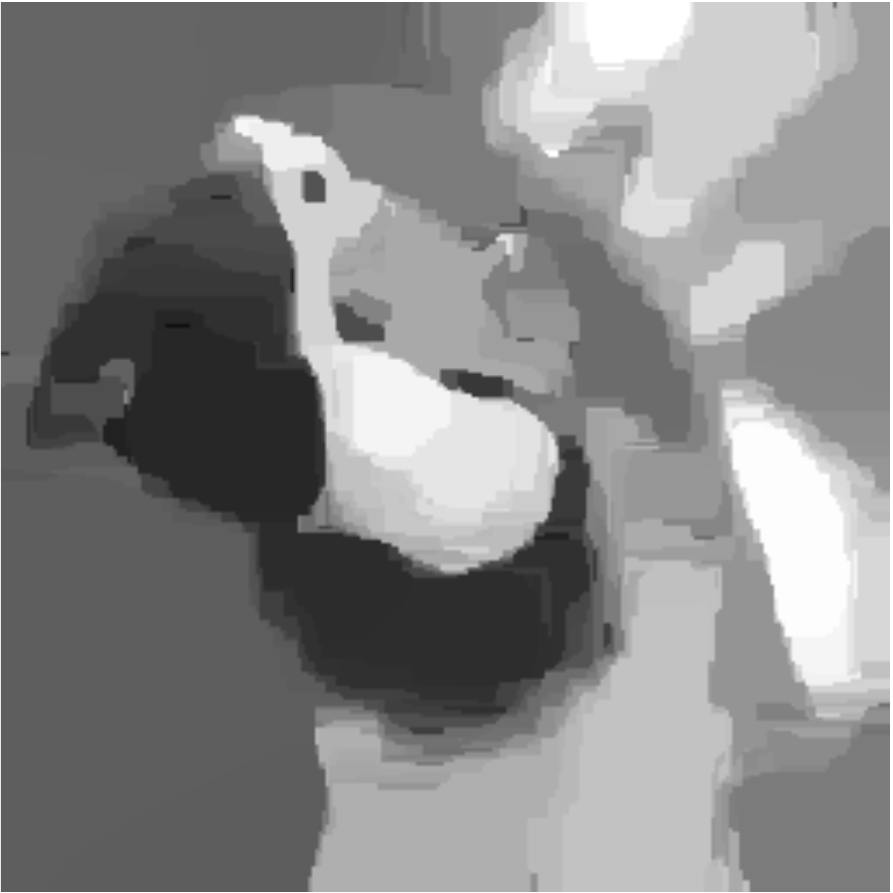}\\
& $PSNR = 21.20$ & $PSNR = 22.29$ & $PSNR = 22.29$
\end{tabular}
\caption{Denoising of the {\it Parrot} image for different values of the regularization parameter $\alpha$.
First column: TV denoising. Second column: Debiasing on the Bregman manifold.
Third column: Debiasing on the infimal convolution subspace. \label{fig:parrot}}
\end{figure*}

\section{Conclusion}

We have introduced two variational debiasing schemes based on Bregman distances and their infimal convolution, 
which are applicable for nonsmooth convex regularizations and generalize known debiasing approaches 
for $\ell^1$ and TV-type regularization. 
Based on a recent axiomatic approach to debiasing by Deledalle and coworkers \cite{deledalle}, which we further generalized towards 
infinite-dimensional problems, we were able to provide a theoretical basis of our debiasing approach and 
work out meaningful model manifolds for variational methods.
Moreover, we were able to relate the approach to Bregman iterations and inverse scale space methods.

From the numerical experiments we observe that the debiasing scheme improves the results for a wide range of 
regularization parameters, which includes the ones providing optimal results. 
Surprisingly, we often find visually optimal choices of the regularization parameters in the range 
where bias and standard deviation of the debiased solution are approximately of the same size. 

Various questions remain open for future studies: 
one might study the generalization to other regularization schemes such as total generalized variation \cite{TGV}, 
spatially adaptive methods that would further reduce the model bias \cite{hintermuller2016analytical}
or nonlocal methods for improved results on natural images. 
As already indicated in the introduction, the method is theoretically not restricted to squared Hilbert-space norms. 
Instead, it can be carried out for any suitable data fidelity $H$ and  we expect it to improve the results.   
From a theoretical, and in particular from a statistical viewpoint, the question is then how to relate the method to actual bias reduction, 
and how to properly motivate and define bias in this setting. 

Another further improvement might be achieved by only approximating the model manifold by tuning the parameter $\gamma$ 
without letting it tend to infinity. 

We acknowledge a very recent and related work on the topic from another perspective, which has been developed in parallel to this work \cite{deledalle2016clear}.
It will be interesting to investigate the connections in future work.

\section{Appendix}
We have included some examples and proofs in the Appendix in order not to interrupt the flow of the paper. 
These are in particular the proof for shrinkage and the calculation of the corresponding derivatives for isotropic and anisotropic shrinkage in Example \ref{ex:aniso_shrinkage},
and the calculation of the infimal convolution of two $\lone$-Bregman distances in Example \ref{ex:iso_TV}.

\subsection{Shrinkage}
\label{app:ex_iso_shrinkage}

Let $f \in \ltwo(\R^d)$ be a vector-valued signal for $d \in \N$.
 Then the solution of the isotropic shrinkage problem 
 \begin{align*}
  \ua(f) \in \argmin_{u \in \lone(\R^d)} \dfrac{1}{2} \|u-f\|_{\ltwo(\R^d)}^2 + \alpha \|u\|_{\lone(\R^d)}
 \end{align*}
 is given by the isotropic soft-thresholding
 \begin{align*}
  [\ua(f)]_i = \begin{cases}
            (1 - \frac{\alpha}{|f_i|}) f_i, & |f_i| > \alpha, \\
            0, & |f_i| \leq \alpha.
           \end{cases}
 \end{align*}
 \begin{proof}
   We first point out, that the objective allows to exploit strong duality.
   Following \cite[Theorem 4.4.3 and Lemma 4.3.1]{Borwein}, strong duality holds if 
   \begin{align*}
    \dom ( \|\cdot \|_{\lone(\R^d)} ) \cap \cont \left(\dfrac{1}{2} \|\cdot-f\|_{\ltwo(\R^d)}^2 \right) \neq \emptyset.
   \end{align*}
   Since the $\lone(\R^d)$-norm has full domain and the $\ltwo(\R^d)$-norm is continuous everywhere, this is trivially fulfilled.
   Hence, by the dual definition of the $\lone(\R^d)$-norm we find 
   \begin{align*}
   &\quad \min_{u \in \lone(\R^d)} \dfrac{1}{2} \|u-f\|_{\ltwo(\R^d)}^2 + \alpha \|u\|_{\lone(\R^d)} \\ 
   &= \min_{u \in \lone(\R^d)} \sup_{\substack{r \in \linf(\R^d) \\ \|r\|_{\linf(\R^d)} \leq \alpha }} \dfrac{1}{2} \|u-f\|_{\ltwo(\R^d)}^2 + \langle r,u \rangle \\
   &= \sup_{\|r\|_{\linf(\R^d)} \leq \alpha } \min_{u \in \lone(\R^d)} \dfrac{1}{2} \|u-f\|_{\ltwo(\R^d)}^2 + \langle r,u \rangle,
   \end{align*}
   where we used strong duality to interchange the infimum and the supremum.
   We can explicitely compute the minimizer for $u$ as $u = f-r$ and hence 
   \begin{align*}
   &\quad \sup_{\|r\|_{\linf(\R^d)} \leq \alpha } \min_{u \in \lone(\R^d)} \dfrac{1}{2} \|u-f\|_{\ltwo(\R^d)}^2 + \langle r,u \rangle \\
   &= \sup_{\|r\|_{\linf(\R^d)} \leq \alpha } - \dfrac{1}{2} \|r\|_{\ltwo(\R^d)}^2 + \langle r,f \rangle. 
   \end{align*}
   This supremum can be computed explicitely pointwise with the corresponding Lagrangian 
   \begin{align*}
   \mathcal{L}(r_i, \lambda) = - \dfrac{1}{2} |r_i|^2 + r_i \cdot f_i + \lambda (|r_i|^2 - \alpha^2)
   \end{align*}
   with $\lambda \leq 0$. 
   Note that both the objective function and the constraints are continuously differentiable and that Slater's condition holds. 
   Optimality with respect to $r_i$ yields
   \begin{align*}
   f_i - r_i + 2 \lambda r_i = 0
   \end{align*}
   and hence 
   \begin{align*}
   r_i = \dfrac{f_i}{1- 2 \lambda}.
   \end{align*}
   We distinguish two cases:\\ 
   If $|r_i| = \alpha$, then $\alpha (1- 2\lambda) = |f_i|$ and
   \begin{align*}
     u_i = f_i - r_i = f_i - \frac{f_i}{1-2\lambda} = (1 - \frac{\alpha}{|f_i|}) f_i.
   \end{align*}
   The nonpositivity of $\lambda$ implies that $|f_i| \geq \alpha$.
   In case $|r_i| < \alpha$, we obtain that $\lambda = 0$ and hence $r_i = f_i$ and $u_i = 0$ when $|f_i| < \alpha$. 
   Note that since $f \in \ltwo(\R^d)$ there exists a finite $N$ such that $|f_i| \leq \alpha$ for all $i > N$. 
   Hence trivially $\ua(f) \in \lone(\R^d)$ as $\sum_{i \in \mathbb{N}} \vert \left[\ua(f)\right]_i \vert$ is a finite sum. 
   This yields the assertion. \QED
 \end{proof}
 
 {\em Remark:}
   For $d = 1$ and a square-summable sequence $f \in \ltwo$ we immediately obtain the anisotropic case: 
   The solution to
   \begin{align}
   \ua \in \arg \min_{u \in \lone} \dfrac{1}{2} \|u-f\|_{\ltwo}^2 + \alpha \|u\|_{\lone}
   \end{align}
   for $\alpha > 0$ is given by 
   \begin{align*}
   [\ua(f)]_i = \begin{cases}
		  f_i - \alpha ~ \sign(f_i), & |f_i| \geq \alpha \\
		  0, & |f_i| < \alpha.
	       \end{cases}
   \end{align*}
 
 {\em Directional derivative:}
 The computation of the directional derivative requires a little more work. 
 At first, let us compute the directional derivative of the function $F \colon \R^d\backslash\{0\} \to \R$, $x \mapsto \frac{1}{|x|}$ into the direction $g \in \R^d$. 
 We define $G \colon \R^d\backslash\{0\} \to \R$, $x \mapsto \frac{1}{|x|^2}$ and calculate
 \begin{align*}
  \der G(x;g) 
  &= \lim_{t \to 0^+} \frac{G(x+tg) - G(x)}{t} \\
  &= \lim_{t \to 0^+} \frac{1}{t} \left( \frac{1}{|x+tg|^2} - \frac{1}{|x|^2} \right) \\
  &= \lim_{t \to 0^+} \frac{1}{t} \left( \frac{|x|^2 - |x + tg|^2}{|x|^2 |x + tg|^2} \right) \\
  &= \lim_{t \to 0^+} \frac{1}{t} \left( \frac{- 2 t x \cdot g - t^2 |g|^2}{|x|^2 |x + tg|^2} \right) \\
  &= - 2 \frac{x \cdot g}{|x|^4}.
 \end{align*}
 Then by the chain rule we obtain
 \begin{align*}
  \der F(x;g) 
  &= \der \sqrt{G}(x;g) = \frac{\der G(x;g)}{2 \sqrt{G(x)}} \\
  &= -2 \frac{x \cdot g}{|x|^4} \frac{|x|}{2} = - \frac{x \cdot g}{|x|^3}.
 \end{align*}
 Let us further define the projection of a vector $x \in \R^d$ onto another vector $y \in \R^d\backslash\{0\}$ as 
 \begin{align*}
  \Pi_y (x) = \frac{y \cdot x}{|y|^2} y.
 \end{align*}

 We now have to compute 
 \begin{align*}
  [\der \ua (f;g)]_i = \lim_{t \to 0^+} \dfrac{1}{t} \big( [\ua (f + tg) ]_i - [\ua (f) ]_i \big)
 \end{align*}
 and we can distinguish four cases:\\ 
 Let at first $|f_i| > \alpha$. 
 Then for $t$ small enough we have $|f_i + t g_i| > \alpha$ and hence 
 \begin{align*}
  \ &\lim_{t \to 0^+} \dfrac{1}{t} \big( [\ua (f + tg) ]_i - [\ua (f) ]_i \big) \\
  = &\lim_{t \to 0^+} \dfrac{1}{t} \left( \left( 1 - \frac{\alpha}{|f_i + t g_i|} \right) (f_i + t g_i)\right.\\
   & \hspace{11.7em} \left.- \left( 1 - \frac{\alpha}{|f_i|} \right) f_i \right) \\
  = &\lim_{t \to 0^+} \dfrac{1}{t} \left( f_i + t g_i - \alpha \frac{f_i + t g_i}{|f_i + t g_i|} - f_i + \alpha \frac{f_i}{|f_i|} \right) \\
  = &\lim_{t \to 0^+} \dfrac{1}{t} \left( t g_i - \frac{\alpha t g_i}{|f_i + t g_i|} \right.\\
  & \hspace{7.6em} \left. - \alpha f_i \left( \frac{1}{|f_i + t g_i|} - \frac{1}{|f_i|} \right) \right) \\
  = &\ g_i - \alpha \frac{g_i}{|f_i|} + \alpha f_i \frac{f_i \cdot g_i}{|f_i|^3} \\
  = &\ g_i + \frac{\alpha}{|f_i|} \left( \Pi_{f_i}(g_i) - g_i \right).
 \end{align*}
 For $|f_i| < \alpha$ and $t$ small enough we easily find $|f_i + t g_i| < \alpha$ and hence
 \begin{align*}
  [\der \ua (f;g)]_i = 0.
 \end{align*}
 In case $|f_i| = \alpha$ we need to distinguish whether $|f_i + t g_i| > \alpha$ or $|f_i + t g_i| \leq \alpha$ for arbitrarily small $t$. 
 We hence compute 
 \begin{alignat*}{3}
  &\quad && |f_i + t g_i| &&> \alpha \\
  &\Leftrightarrow \ &&|f_i + t g_i|^2 &&> \alpha^2 \\
  &\Leftrightarrow \ &&|f_i|^2 + 2 t f_i \cdot g_i + t^2 |g_i|^2 &&> \alpha^2 \\
  &\Leftrightarrow \ && 2 f_i \cdot g_i + t |g_i|^2 &&> 0, 
 \end{alignat*}
 which for arbitrarily small $t$ is true only if $f_i \cdot g_i \geq 0$. 
 Analogously we find that $|f_i + t g_i| < \alpha$ for small $t$ is only true if $f_i \cdot g_i < 0$.\\
 Hence let now $|f_i| = \alpha$ and $f_i \cdot g_i \geq 0$. 
 Then we obtain 
 \begin{align*}
  [\der \ua &(f;g)]_i 
  = \lim_{t \to 0^+} \dfrac{1}{t} \big( [\ua (f + tg) ]_i \big) \\
  &= \lim_{t \to 0^+} \dfrac{1}{t} \left( \left( 1- \frac{\alpha}{|f_i + t g_i|} \right) (f_i + t g_i) \right).
 \end{align*}
 Using $\alpha = |f_i|$, we find
 \begin{align*}
  & \lim_{t \to 0^+} \dfrac{|f_i|f_i}{t} \left( \frac{1}{|f_i|} \text{ - } \frac{1}{|f_i + t g_i|} \right) + g_i \text{ - } \frac{|f_i| g_i}{|f_i + t g_i|} \\
  &= |f_i|f_i \frac{f_i \cdot g_i}{|f_i|^3} \\
  &= \Pi_{f_i}(g_i).
 \end{align*}
 In the last case $|f_i| = \alpha$ and $f_i \cdot g_i < 0$, we find 
 \begin{align*}
  [\der \ua (f;g)]_i = \lim_{t \to 0^+} \dfrac{1}{t} \big( [\ua (f + tg) ]_i \big) = 0.
 \end{align*}
 Summing up we have
\begin{align*}
[\der \ua(f;&g)]_i \\ 
&= \left\{
\begin{array}{lcr}
\multicolumn{2}{l}{g_i + \frac{\alpha}{|f_i|} \left( \Pi_{f_i}(g_i) - g_i \right),\qquad } & |f_i|> \alpha,\\
0, & & |f_i|< \alpha,\\
\Pi_{f_i}(g_i),\qquad & \multicolumn{2}{r}{\quad |f_i|= \alpha,\ f_i \cdot g_i > 0,} \\
0, & \multicolumn{2}{r}{|f_i|= \alpha,\ f_i \cdot g_i \leq 0.} 
\end{array}\right.
\end{align*}
It remains to show that 
\begin{align*}
\Big\| \frac{ \ua (f + tg) - \ua (f) }{t} - \der \ua (f;g) \Big\|_{\lone(\R^d)} \to 0
\end{align*}
for $t \to 0^+$.
Again, since $f \in \ltwo(\R^d)$, there exists $N \in \N$ such that $|f_i| < \alpha$ and hence $[\der \ua (f;g)]_i = 0$  for all $i > N$. 
The difference quotient as well vanishes for all $i > N$, hence  the above $\lone$ norm is a finite sum and thus we  trivially obtain convergence in $\lone(\R^d)$. 

{\em Remark:}
For $d=1$ and $f \in \ltwo$ we obtain the anisotropic result: 
\begin{align*}
 [\der \ua (f;g)]_i& \\ = &\begin{cases}
                       g_i, & |f_i| > \alpha \\
                       0, & |f_i| < \alpha \\
                       g_i, & |f_i| = \alpha, \sign(f_i) = \sign(g_i) \\
                       0, & |f_i| = \alpha, \sign(f_i) \neq \sign(g_i),
                      \end{cases}
\end{align*}
where we mention that here $\Pi_{f_i}(g_i) = g_i$. 

{\em Model manifold:}
The corresponding (isotropic) model manifold is given by 
\begin{align*}
 u \in \MG \Leftrightarrow u_i = \begin{cases}
                                  v \in \R^d, & |f_i| > \alpha, \\
                                  0, & |f_i| < \alpha, \\
                                  \lambda f_i,\ \lambda \geq 0, & |f_i| = \alpha.
                                 \end{cases}
\end{align*}
Analogously to the anisotropic case discussed in Example \ref{ex:aniso_shrinkage}, the model manifold allows for arbitrary elements, 
here even including the direction, if the magnitude $|f_i|$ of the signal is strictly above the threshold parameter $\alpha$. 
As already discussed in Example \ref{ex:aniso_shrinkage}, $|f_i| = \alpha$ is the odd case of the three, 
since in contrast to $|f_i| > \alpha$ it only allows for changes into the direction of the signal $f_i$. 
If we exclude that case, we again find a linear derivative, hence a G\^ateaux derivative and even a Fr\'echet derivative.
Accordingly the isotropic shrinkage is the immediate generalization of the anisotropic shrinkage, which we can find as a special case for $d = 1$. 

Summing up, the debiasing procedure on this manifold again yields the solution of hard thresholding: 
\begin{align*}
 [\hat{u}(f)]_i = \begin{cases}
                f_i, & |f_i| \geq \alpha, \\
                0, & |f_i|< \alpha.
               \end{cases}
\end{align*}
Note that we again maintain the signal directly on the threshold.

\subsection{Infimal convolution of $\lone$ Bregman distances}
\label{app:infconv}
\begin{theorem}
\label{thm:smaller}
 Let $\Gamma \colon \ltwo(\R^n) \to \lone(\R^m)$ be linear and bounded and $J(u) = \|\Gamma u \|_{\lone(\R^m)}$ for $m,n \in \N$.
 Let further $\qa \in \partial \| \cdot \|_{\lone(\R^m)} (\Gamma \ua)$ such that $\pa = \Gamma^* \qa$.
 Then 
 \begin{align*}
  \ICBlisoq(\Gamma u, \Gamma \ua) \leq \ICB(u, \ua).
 \end{align*}
\end{theorem}

\begin{proof}
\begin{align*}
 & \quad \ICB(u, \ua) \\
 &= \inf_{z \in \ltwo(\R^n)} ~ D_J^{\pa}(u-z,\ua) + D_J^{-\pa}(z,-\ua) \\
 &= \inf_{z \in \ltwo(\R^n)} ~ \| \Gamma (u-z) \|_{\lone(\R^m)} - \langle \pa,u-z \rangle \\ 
   &\quad + \|\Gamma z\|_{\lone(\R^m)} + \langle \pa,z \rangle \\
 &= \inf_{z \in \ltwo(\R^n)} ~ \| \Gamma (u-z) \|_{\lone(\R^m)} - \langle \qa,\Gamma(u-z) \rangle \\
   &\quad + \|\Gamma z\|_{\lone(\R^m)} + \langle \qa,\Gamma z \rangle \\
 &= \inf_{\Gamma z \in \lone(\R^m)} \| \Gamma (u-z) \|_{\lone(\R^m)} - \langle \qa,\Gamma(u-z) \rangle \\
   &\quad + \|\Gamma z\|_{\lone(\R^m)} + \langle \qa,\Gamma z \rangle \\
 &\geq \inf_{w \in \lone(\R^m)} ~ \| \Gamma u - w \|_{\lone(\R^m)} - \langle \qa,\Gamma u -w \rangle \\
   &\quad + \|w\|_{\lone(\R^m)} + \langle \qa,w \rangle \\
 &= \inf_{w \in \lone(\R^m)} ~ D_{\lone(\R^m)}^{\qa}(\Gamma u - w, \Gamma \ua) \\ 
   &\quad + D_{\lone(\R^m)}^{-\qa}(w, -\Gamma \ua) \\
 &= \ICBlisoq(\Gamma u, \Gamma \ua).
\end{align*}
\QED
\end{proof}
Note that we get equality for surjective $\Gamma$ in Theorem \ref{thm:smaller}.

\begin{theorem}
\label{thm:infconv_l1}
 Let $v,u \in \ell^1(\R^m)$ and $q \in \partial \|v\|_{\lone(\R^m)}$. Then
\begin{align*}
 \textnormal{ICB}_{\lone(\R^m)}^q(u,v) = \sum_{i \in \N} G(u_i, q_i)
\end{align*}
with $G \colon \R^m \times \R^m \to \R$ defined as
\begin{align*}
 &G(u_i, q_i) \\
 &= \begin{cases}
   |u_i| (1 - |\cos(\varphi_i)| |q_i|), & |q_i| < |\cos(\varphi_i)|, \\
   |u_i| | \sin(\varphi_i)| \sqrt{1 - |q_i|^2}, & |q_i| \geq |\cos(\varphi_i)|.
   \end{cases}
\end{align*}
 where $\varphi_i$ denotes the angle between $u_i$ and $q_i$, i.e. 
 $\cos(\varphi_i) |u_i| |q_i| = u_i \cdot q_i$
 with $\varphi_i := 0$ for $q_i = 0$ or $u_i = 0$.
\end{theorem}
\begin{proof}
Let 
\begin{alignat*}{2}
f_1(u) &= D_{\lone(\R^m)}^q(u,v) &= \|u\|_{\lone(\R^m)} - \langle q,u \rangle, \\
f_2(u) &= D_{\lone(\R^m)}^{-q}(u,-v) &= \|u \|_{\lone(\R^m)} + \langle q, u \rangle.
\end{alignat*}
Since $(f_1 \Box f_2)^* = f_1^* + f_2^*$ and by the definition of the biconjugate, we know that
\begin{align*}
f_1 \Box f_2 \geq (f_1^* + f_2^*)^*.
\end{align*}

(1) We shall first compute the right-hand side.
We have
\begin{align*}
f_1^*(w) &= \iota_{B^{\infty}(1)}(w +q), \\
f_2^*(w) &= \iota_{B^{\infty}(1)}(w -q),
\end{align*}
where $\iota_{B^{\infty}(1)}$ denotes the characteristic function of the $\linf(\R^m)$-ball
\begin{align*}
 B^{\infty}(1) = \big\{ w \in \linf(\R^m) ~|~ \| w \|_{\linf(\R^m)} \leq 1 \big\}.
\end{align*}
Thus 
\begin{align*}
&(f_1^* + f_2^*)^*(u) = \sup_{w \in \linf(\R^m) } \langle u, w \rangle \\
&\text{ s.t.} ~ \|w+q\|_{\linf(\R^m)} \leq 1, \|w-q\|_{\linf(\R^m)} \leq 1.
\end{align*}
Taking into account the specific form of these constraints, we can carry out the computation pointwise, i.e.
\begin{align*}
\sup_{w_i \in \R^m } u_i \cdot w_i ~ \text{ s.t.} ~ |w_i + q_i| \leq 1, |w_i - q_i| \leq 1.
\end{align*}
From now on we drop the dependence on $i$ for simplicity.

$\bullet$ Let us first consider the case $|q| = 1$.
We immediately deduce that $w = 0$ and $u \cdot w = 0$. 

$\bullet$ Hence we assume $|q| < 1$ from now on, and set up the corresponding Lagrangian
\begin{align}
\mathcal{L}(w, \lambda, \mu) = - w \cdot u &+ \lambda (|w - q|^2 -1) \nonumber \\ 
&+ \mu (|w + q|^2 -1).
\label{eq:Lagrange}
\end{align}
Both the objective functional and the constraints are differentiable, 
so every optimal point of (\ref{eq:Lagrange}) has to fulfill the four Karush-Kuhn-Tucker conditions, namely
\begin{align*}
&\dfrac{\partial}{\partial w} \mathcal{L}(w, \lambda, \mu) = 0, \quad &\lambda (|w -q|^2-1) = 0, \\ 
&\lambda, \mu \geq 0, & \mu (|w + q|^2 -1) = 0,
\end{align*}
Slater's condition implies the existence of Lagrange multipliers for a KKT-point of (\ref{eq:Lagrange}). 
The first KKT-condition yields
\begin{align}
-u + 2 \lambda (w-q) + 2 \mu (w + q) = 0.
\label{eq:FirstOrderOptimality}
\end{align}

$\ast$ Let us first remark that the case $u=0$ causes the objective function to vanish anyway, hence in the following $u\neq 0$.

$\ast$ Then let us address the case $q=0$ in which \eqref{eq:FirstOrderOptimality} yields 
\begin{align*}
 u = 2 (\lambda + \mu) w.
\end{align*}
In case $|w| = 1$ we find that $2 (\lambda + \mu) = |u|$, hence $w = \frac{u}{|u|}$. 
We infer 
\begin{align*}
 w \cdot u = \frac{u \cdot u}{|u|} = |u|. 
\end{align*}
Note that for $|w| < 1$, we find that $\lambda = \mu = 0$ and hence $u = 0$.\\

$\ast$ If $q \neq 0$, we can distinguish four cases:\\
\textbf{1st case:} $|w -q|^2 <1, |w +q|^2 = 1$. \\
Thus $\lambda = 0$ and (\ref{eq:FirstOrderOptimality}) yields
\begin{align*}
u = 2 \mu (w +q). 
\end{align*}
Since $|w + q|^2 = 1$, we deduce $\mu = |u|/2$, so
\begin{align*}
w = \dfrac{u}{|u|} - q
\end{align*}
and finally for the value of the objective function
\begin{align*}
w \cdot u = \left( \dfrac{u}{|u|} - q \right) \cdot u = |u| - q \cdot u.
\end{align*}
\textbf{2nd case:} $|w +q|^2 <1, |w -q|^2 = 1$. \\ 
We analogously find 
\begin{align*}
w \cdot u = |u| + q\cdot u.
\end{align*}
The first two cases thus occur whenever (insert $w$ into the conditions)
\begin{align*}
\left| \dfrac{u}{|u|} - 2q \right| < 1 \text{ or } \left| \dfrac{u}{|u|} + 2q \right| < 1.
\end{align*} 
We calculate
\begin{alignat*}{3}
& \quad &&\left| \dfrac{u}{|u|} - 2q \right|^2 &&< 1 \\ 
&\Leftrightarrow  && \hspace{3.35em} |q|^2 &&< q \cdot \dfrac{u}{|u|} \\
&\Leftrightarrow  &&\hspace{3.35em} |q| &&< \cos(\varphi).
\end{alignat*}
Hence $q\cdot u >0$ and 
\begin{align*}
|u| - q \cdot u = |u| -|q \cdot u|. 
\end{align*}
In the second case we analogously find 
\begin{align*}
 |q| < -\cos(\varphi),
\end{align*}
hence
$q \cdot u <0$ and
\begin{align*}
|u| + q \cdot u = |u| -|q \cdot u|, 
\end{align*}
so we may summarize the first two cases as 
\begin{align*}
w \cdot u = |u| - |q \cdot u| =  |u| (1 - |\cos(\varphi)| |q|),
\end{align*}
whenever $|q| < |\cos(\varphi)|$. \\

\noindent
\textbf{3rd case:} $|w -q|^2 =1, |w +q|^2 = 1$. \\
At first we observe that from 
\begin{align*}
|w +q |^2 = |w -q|^2 
\end{align*}
we may deduce that $w \cdot q = 0$. 
Therefore we have
\begin{align*}
|w +q|^2 = 1 \Rightarrow |w| = \sqrt{1-|q|^2}. 
\end{align*}
We multiply the optimality condition \eqref{eq:FirstOrderOptimality} by $q$ and obtain
\begin{alignat*}{3}
&\qquad &&u \cdot q &&= 2\lambda (w -q) \cdot q + 2\mu (w +q) \cdot q \\
&\Leftrightarrow && u \cdot q &&= 2(\mu - \lambda)~ |q|^2 \\
&\Leftrightarrow &&(\mu - \lambda) &&= \frac{u}{2} \cdot \dfrac{q}{|q|^2}.
\end{alignat*}
Multiplying \eqref{eq:FirstOrderOptimality} by $w$ yields
\begin{align*}
u \cdot w = 2 (\lambda + \mu) |w|^2
\end{align*}
and another multiplication of \eqref{eq:FirstOrderOptimality} by $u$ yields
\begin{align*}
|u|^2 &= 2 (\lambda +\mu) w \cdot u + 2 (\mu -\lambda) q\cdot u \\
&= 4 (\lambda +\mu)^2 |w|^2 + \left( u\cdot \dfrac{q}{|q|} \right)^2,
\end{align*}
where we inserted the previous results in the last two steps. 
We rearrange and find
\begin{align*}
2 (\lambda + \mu) = \sqrt{ |u|^2 - \left( u \cdot \dfrac{q}{|q|} \right)^2} |w|^{-1}.
\end{align*}
Note that $|w| > 0$ since $|q| < 1$.
This finally leads us to 
\begin{align*}
u \cdot w &= 2 (\lambda + \mu) |w|^2 \\ 
&= \sqrt{ |u|^2 - \left( u \cdot \dfrac{q}{|q|} \right)^2} |w| \\
&= |u| \sqrt{\left(1 - \left( \dfrac{u}{|u|} \cdot \dfrac{q}{|q|} \right)^2\right) \left(1 - |q|^2\right)} \\
&= |u| \sqrt{\left(1 - |\cos(\varphi)|^2\right) \left(1 - |q|^2\right)} \\
&= |u| |\sin(\varphi)| \sqrt{\left(1 - |q|^2\right)}.
\end{align*}

\noindent
\textbf{4th case:} $|w -q|^2 < 1, |w +q|^2 < 1$. \\
Here the first KKT-condition yields $u =0$, 
which can only occur if the objective function $w\cdot u$ vanishes anyway. 
Summing up, we have 
\begin{align*}
 (f_1^* + f_2^*)^*(u) = \sum_{i \in \N} G(u_i,q_i) \leq \|u\|_{\lone(\R^m)}. 
\end{align*}

(2) It remains to show that
\begin{align*}
(f_1 \Box f_2)(u) &= \inf_{z \in \lone(\R^m)} \sum_{i \in \mathbb{N}} g_i(z_i) \\ 
&\leq (f_1^* + f_2^*)^*(u),
\end{align*}
where
\begin{align*}
 g_i(z_i) = |u_i - z_i| + |z_i| - q_i \cdot (u_i - 2z_i) \geq 0.
\end{align*}
Again we need to distinguish four cases.\\ 

\noindent
\textbf{1st case:} 
If $|q_i| < \cos(\varphi_i)$, we have $q_i \cdot u_i > 0$ and we can choose $z_i = 0$ to obtain
\begin{align*}
 g_i(z_i) = |u_i| - q_i \cdot u_i = |u_i| - |q_i \cdot u_i|. 
\end{align*}

\noindent
\textbf{2nd case:} 
Analogously if $|q_i| < -\cos(\varphi_i)$, we have $q_i \cdot u_i < 0$ and choose $z_i = u_i$, thus 
\begin{align*}
 g_i(z_i) = |u_i| + q_i \cdot u_i = |u_i| - |q_i \cdot u_i|. 
\end{align*}

\noindent
\textbf{3rd case:} 
If $|q_i| = 1$, we compute for $z_i = \frac{u_i}{2} - \frac{c}{2}q_i$ , $c > 0$,
\begin{align*}
 g_i(z_i) &= \left| \frac{u_i}{2} + \frac{c}{2} q_i \right| + \left| \frac{u_i}{2} - \frac{c}{2} q_i \right| - c|q_i|^2 \\
 &=\frac{c}{2} \left( \left| q_i + \frac{u_i}{c} \right| + \left| q_i - \frac{u_i}{c} \right| - 2 \right). 
\end{align*}
Using a Taylor expansion around $q$ we obtain 
\begin{align*}
 \left| q_i + \frac{u_i}{c} \right| &= |q_i| + \frac{q_i}{|q_i|} \cdot \frac{u_i}{c} + O (c^{-2}), \\
 \left| q_i - \frac{u_i}{c} \right| &= |q_i| - \frac{q_i}{|q_i|} \cdot \frac{u_i}{c} + O (c^{-2}).
\end{align*}
Hence with $|q_i| = 1$ we find
\begin{align*}
 g_i(z_i) = \frac{c}{2} (2 |q_i| + O(c^{-2}) - 2) = O(c^{-1}) \to 0 
\end{align*}
for $c \to \infty$.
Hence for every $\varepsilon$ there exists a $c_i > 0$ such that $g_i(z_i) \leq \varepsilon / 2^i$.\\

\noindent
\textbf{4th case:} 
Finally, if $|q_i| \geq |\cos(\varphi_i)|$ and $|q_i| < 1$, we pick $z_i = 2 \lambda_i (w_i - q_i)$, 
with $\lambda_i$ and $w_i$ being the Lagrange multiplier and the dual variable from the above computation of $(f_1^* + f_2^*)^*$.
It is easy to see that 
\begin{align*}
 g_i(z_i) = |u_i| | \sin(\varphi_i)| \sqrt{1 - |q_i|^2}.
\end{align*}

Hence we define $z := (z_i)_i$ such that 
\begin{align*}
 z_i = \begin{cases}
     0, & \text{ if } |q_i| < \cos(\varphi_i),\\
     u_i, & \text{ if } |q_i| < -\cos(\varphi_i),\\
      \frac{u_i}{2} - \frac{c_i}{2}q_i, & \text{ if } |q_i| = 1, \\
     \lambda_i (w_i - q_i) & \text{ if } |q_i| \geq |\cos(\varphi_i)|, \\
     & \hspace{1.25em} |q_i| < 1.
     \end{cases}
\end{align*}
Let $z^N$ denote $z$ truncated at index $N \in \N$, i.e. 
\begin{align*}
 z_i^N = \begin{cases}
          z_i, & \text{ if } i \leq N, \\
          0, & \text{ else.}
         \end{cases}
\end{align*}
Then trivially $z^N \in \lone(\R^m)$ and we compute
\begin{alignat*}{2}
 &(f_1 \Box f_2)(u) \leq  \sum_{i \in \N} g_i(z_i^N) \\
 \leq & \sum_{i = 1}^N \big( G(u_i, q_i) + \frac{\varepsilon}{2^i} \big) + \sum_{i=N+1}^\infty g_i(0) \\
 = &\sum_{i = 1}^\infty G(u_i, q_i) + \sum_{i = 1}^N \frac{\varepsilon}{2^i}\\
 &  + \sum_{i=N+1}^\infty \big( |u_i| - q_i \cdot u_i - G(u_i,q_i) \big) \\
 \leq & \sum_{i = 1}^\infty G(u_i, q_i) + \sum_{i = 1}^N \frac{\varepsilon}{2^i} + 3 \sum_{i=N+1}^\infty |u_i| \\
 \to& \sum_{i = 1}^\infty G(u_i, q_i) + \varepsilon
\end{alignat*}
as $N \to \infty$. 
This completes the proof.
\QED
\end{proof}

\section*{Acknowledgements}

This work was supported by ERC via Grant EU FP 7 - ERC Consolidator Grant 615216 LifeInverse. 
MB acknowledges support by the German Science Foundation DFG via  EXC 1003 Cells in Motion Cluster of Excellence, 
M\"unster, Germany.

\end{document}